\documentclass[12pt]{amsart}

\usepackage{amsmath}
\usepackage{amssymb}
\usepackage{amsthm}
\usepackage{amscd}
\usepackage{xypic}
\swapnumbers
\headsep=1cm \numberwithin{equation}{section}
\newtheorem{theorem}{Theorem}[section]
\newtheorem{lemma}[theorem]{Lemma}
\newtheorem{proposition}[theorem]{Proposition}
\newtheorem{corollary}[theorem]{Corollary}

\theoremstyle{definition}
\newtheorem{definition}[theorem]{Definition}
\newtheorem{definition and remark}[theorem]{Definition and Remark}
\newtheorem{remark}[theorem]{Remark}
\newtheorem{remark and definition}[theorem]{Remark and Definition}
\newtheorem{remark and notation}[theorem]{Remark and Notation}
\newtheorem{notation and remark}[theorem]{Notation and Remark}
\newtheorem{notation and convention}[theorem]{Notation and Convention}

\newtheorem{notation and reminder}[theorem]{Notation and Reminder}
\newtheorem{example}[theorem]{Example}

\newtheorem{construction and examples}[theorem]{Construction and Examples}
\newtheorem{problem}[theorem]{Problem}
\newtheorem{problem and remark}[theorem]{Problem and Remark}

\newcommand\Proj{\operatorname{Proj}}

\newcommand\Hom{\operatorname{Hom}}

\newcommand\Tor{\operatorname{Tor}}

\newcommand\depth{\operatorname{depth}}
\newcommand\codim{\operatorname{codim}}

\newcommand\reg{\operatorname{reg}}
\newcommand\Ker{\operatorname{Ker}}
\newcommand\Coker{\operatorname{Coker}}
\newcommand\e{\operatorname{e}}

\newcommand\Nor{\operatorname{Nor}}
\newcommand\Rad{\operatorname{Rad}}
\newcommand\Sec{\operatorname{Sec}}

\newcommand\Sing{\operatorname{Sing}}

\newcommand\sreg{\operatorname{sreg}}
\newcommand\mult{\operatorname{mult}}
\newcommand\sat{\operatorname{sat}}
\newcommand\Reg{\operatorname{Reg}}
\newcommand\CM{\operatorname{CM}}
\newcommand\length{\operatorname{length}}
\newcommand\T{\operatorname{T}}

\newcommand\en{\operatorname{end}}

\newcommand\Char{\operatorname{Char}}

\begin{document}

\title[PROJECTIVE SURFACES OF MAXIMAL SECTIONAL REGULARITY]
         {PROJECTIVE SURFACES OF MAXIMAL SECTIONAL REGULARITY}

\author{Markus BRODMANN, Wanseok LEE, Euisung PARK, Peter SCHENZEL}

\address{Universit\"at Z\"urich, Institut f\"ur Mathematik, Winterthurerstrasse 190, CH -- Z\"urich, Switzerland}
\email{brodmann@math.unizh.ch}

\address{Korea Institute for Advanced Study, School of Mathematics, 85 Hoegiro, Dongdemun-gu, Seoul 130-722, Republic of Korea}
\email{wslee@kias.re.kr}

\address{Korea University, Department of Mathematics, Anam-dong, Seongbuk-gu, Seoul 136-701, Republic of Korea}
\email{euisungpark@korea.ac.kr}

\address{Martin-Luther-Universit\"at Halle-Wittenberg,
Institut f\"ur Informatik, Von-Secken\-dorff-Platz 1, D -- 06120 Halle
(Saale), Germany}
\email{schenzel@informatik.uni-halle.de}

\date{Z\"urich, 10. May 2013}



\begin{abstract} We study projective surfaces $X \subset \mathbb{P}^r$ (with $r \geq 5$) of maximal sectional regularity and degree $d > r$, hence surfaces for which the
Castelnuovo-Mumford regularity $\reg(C)$ of a general hyperplane section curve $C = X \cap \mathbb{P}^{r-1}$ takes the maximally possible value $d-r+3$. We show that each of
these surfaces is either a cone over a curve $C \subset \mathbb{P}^{r-1}$ of maximal regularity or else a birational outer linear projection of a smooth rational surface
scroll $\widetilde{X} \subset \mathbb{P}^{d+1}$. We prove that the Castelnuovo-Mumford regularity of these surfaces satisfies the equality $\reg(X) = d-r+3$ and we compute
or estimate various of their cohomological invariants as well as their Betti numbers. We study the the extremal variety $\mathbb{F}(X)$ of these surfaces $X$, that is the
closed union of the extremal secant lines of all smooth hyperplane section curves of $X$. We show that $\mathbb{F}(X)$ is either a plane or that otherwise $r =5$ and
$\mathbb{F}(X)$ is a rational smooth threefold scroll $S(1,1,1) \subset \mathbb{P}^5$.
\end{abstract}

\maketitle
\thispagestyle{empty}

\section{Introduction}
\label{1. Introduction}

\noindent For a non-degenerate irreducible projective variety $X
\subset \mathbb{P}^r$ defined over an algebraically closed field
$K$, there are various interesting questions regarding the syzygies
of the homogeneous vanishing ideal $I_X$ of $X$. One of the
prominent classical problems in this context is to find (least)
upper bounds on the number $m$ for which the following properties
hold:
\begin{itemize}
\item[${\rm (A_m)}$] The hypersurfaces of degree $m$ cut out a complete linear system on $X$.
\item[${\rm (B_m)}$] $X$ is cut out in $\mathbb{P}^r$ by hypersurfaces of degree $m$ and $I_X$ is generated by homogeneous polynomials of degree $\leq m$.
\end{itemize}
This bounding problem was completely solved in 1893 for smooth curves $X$ in complex projective $3$-space by Castelnuovo \cite{C}. In 1966, Mumford \cite{M} introduced the
concept of Castelnuovo regularity (later called Castelnuovo-Mumford regularity) and reformulated the above problem as a bounding problem for this new invariant
in terms of the degree, the codimension, the Hilbert coefficients or other projective invariants of $X$. Recall that Mumford defined the variety
$X \subset \mathbb{P}^r$ to be $m$-\textit{regular} if its sheaf of vanishing ideals $\mathcal{J}_X \subseteq \mathcal{O}_{\mathbb{P}^r}$ satisfies the following
cohomological vanishing condition
$$H^i(\mathbb{P}^r, \mathcal{J}_X(m-i)) = 0 \mbox{ for all } i \geq 1.$$
Keep in mind, that $X$ is $k$-regular for all $k \geq m$ if it is
$m$-regular. This latter observation gives justification to define
the \textit{Castelnuovo-Mumford regularity} $\reg(X)$ of $X$ as the
least integer $m$ such that $X$ is $m$-regular. It is well known
that the $m$-regularity of $X$ implies the properties ${\rm
(A_{m-1})}$ and ${\rm (B_m)}$. Moreover, property ${\rm (B_m)}$ has
the geometric consequence that each $(m+1)$-secant line to $X$ is
actually contained in $X$.

A well known conjecture due to Eisenbud and Goto (see \cite{EG})
says that
$$\reg(X) \leq d-c+1,$$
where $d$ is the degree and $c$ is the codimension of $X \subset
\mathbb{P}^r$. This conjecture has been proved so far only for
irreducible but not necessarily smooth curves $X \subset
\mathbb{P}^r$ by Gruson-Lazarsfeld-Peskine \cite{GruLPe}, and for
smooth complex surfaces by Pinkham \cite{Pi} and Lazarsfeld
\cite{L}. Moreover in \cite{GruLPe} the curves in $\mathbb{P}^r$,
whose regularity takes the maximally possible value $d-r+2$ are
classified: they are either of degree $\leq r+1$ or else smooth
rational curves having a $(d-r+2)$-secant line.

An attempt to push further this latter classification is to study
arbitrary \textit{varieties of extremal regularity}, that is
non-degenerate irreducible varieties $X \subset \mathbb{P}^r$ of
codimension $c$ and degree $d$ which satisfy the inequality $\reg(X)
\geq d-c+1$. This idea is a basic guideline for our paper.

To present our approach in more detail, we suppose that the degree
$d$ and the codimension $c$ of our non-degenerate irreducible
variety $X \subset \mathbb{P}^r$ satisfy $3 \leq c < r$ and $d >
c+2$. A $(d-c+1)$-secant line $\mathbb{L}$ to $X$, which is not
contained in $X$ is called an \textit{extremal secant line} to $X$.
If $X$ admits such an extremal secant line, it must be of extremal
regularity. So, a possible generalization of the classification of
curves of extremal (and hence maximal) regularity given in
\cite{GruLPe}, is to classify varieties of extremal regularity by
the ``size" of their set
$$\Sigma^{\circ}_{d-c+1}(X) := \{ \mathbb{L} \in \mathbb{G}(1,\mathbb{P}^r) \mid \#(\mathbb{L} \cap X)= d-c+1 \}$$
of extremal secant lines, where $\mathbb{G}(k, \mathbb{P}^r)$ denotes the Grassmannian of $k$-spaces $\mathbb{P}^k \subset \mathbb{P}^r$ and $\#Z$ denotes the length of
the Noetherian scheme $Z$. As $\Sigma^{\circ}_{d-c+1}(X)$ is locally closed in $\mathbb{G}(1,\mathbb{P}^r)$, its size is naturally measured by its dimension
$$\mathfrak{d}(X) := \dim(\Sigma^{\circ}_{d-c+1}(X)) = \dim\big(\overline{\Sigma^{\circ}_{d-c+1}(X)}\big).$$
The \textit{special extremal secant lines} to $X$ -- hence the extremal secant lines to subspace section curves of $X$ which have maximal regularity -- are of particular
interest for our investigation. To make this more explicit, we write $\mathbb{P}\mathbb{U}(X)$ for the set of all subspaces $\mathbb{E} \in \mathbb{G}(c+1,\mathbb{P}^r)$
for which $X \cap \mathbb{E} \subset \mathbb{E} = \mathbb{P}^{c+1}$ is a (non-degenerate integral) curve of maximal regularity. Then -- as $c+1 \geq 4$ -- for
each $\mathbb{E} \in \mathbb{P}\mathbb{U}(X)$, the curve $X \cap \mathbb{E}$ has a unique $(d-c+1)$-secant line $\mathbb{L}_{\mathbb{E},X} \in \Sigma^{\circ}_{d-c+1}(X)$.
We now consider the set
$${}^*\Sigma^{\circ}_{d-c+1}(X) := \{ \mathbb{L}_{\mathbb{E},X} \mid \mathbb{E} \in \mathbb{P}\mathbb{U}(X)\}$$
of special extremal secant lines to $X$ and measure its size by the dimension of its closure, thus by
$${}^*\mathfrak{d}(X) := \dim\big(\overline{{}^*\Sigma^{\circ}_{d-c+1}(X)}\big).$$
Clearly ${}^*\mathfrak{d}(X)$ also measures the size of the set
$\mathbb{P}\mathbb{U}(X)$ of ''good" $c+1$ subspaces $\mathbb{E}$ of
$\mathbb{P}^r$, and one expects that $X$ behaves nicely if this
latter set is ``big" -- that is contains a dense open subset of
$\mathbb{G}(c+1,\mathbb{P}^r)$. In this case, we say that $X$ is of
\textit{maximal sectional regularity}. Our first main result relates
the above concepts as follows (see Theorem~\ref{2.3'' Theorem}).

\begin{theorem}\label{1.1 Theorem}
Assume that $1 \leq t \leq r-3$ and let $X \subset \mathbb{P}^r$ be a non-degenerate irreducible projective variety of dimension $t$ and degree $> r-t+2$.
Then, ${}^*\mathfrak{d}(X) = \mathfrak{d}(X) \leq 2t-2$. with equality at the second place if and only if $X$ is of maximal sectional regularity.
\end{theorem}

Besides of this brief look at varieties of extremal regularity of
arbitrary dimension, we shall concentrate to the case of surfaces of
extremal regularity, and within this case, we focus to
\textit{surfaces of maximal sectional regularity}, hence to
non-degenerate irreducible surfaces $X \subset \mathbb{P}^r$ with
$\mathfrak{d}(X) = {}^*\mathfrak{d}(X) = 2$.

We attack the investigation of these surfaces via a detour, which is
of interest on its own: namely, a study of \textit{sectionally
rational varieties}, hence of non-degenerate irreducible varieties
$X \subset \mathbb{P}^r$ of codimension $c$ such that $X \cap
\mathbb{E}$ is a rational curve for a general space $\mathbb{E} \in
\mathbb{G}(c+1,\mathbb{P}^r)$. It turns out that these varieties are
outer linear projections of varieties of minimal degree, provided
they are either surfaces with finite non-normal locus or else the
base field $K$ has characteristic $0$ (see Theorem~\ref{2.3'
Theorem}). Then we prove a bounding result for the
Castelnuovo-Mumford regularity of almost \textit{non-singular
projections} of varieties which satisfy the syzygetic property
$N_{2,p}$ of Green-Lazarsfeld (see Theorem~\ref{2.5' Theorem}). As
an application we show that sectionally rational surfaces with
finite non-normal locus satisfy the conjectural inequality of
Eisenbud-Goto (see Corollary~\ref{2.6' Corollary}). As surfaces of
maximal sectional regularity are sectionally rational, these results
directly apply to them.

For an arbitrary non-degenerate irreducible surface $X \subset \mathbb{P}^r$ of degree $d$, we may consider the closed union
$$\mathbb{F}^+(X) := \overline{\bigcup_{\mathbb{L} \in \Sigma^{\circ}_{d-r+3}(X)} \mathbb{L}}$$
of all extremal secant lines to $X$, which we call the \textit{extended extremal variety} of $X$.
Moreover, if $r \geq 5$, we call the closed union
$$\mathbb{F}(X) := \overline{\bigcup_{\mathbb{L} \in {}^*\Sigma^{\circ}_{d-r+3}(X)} \mathbb{L}}$$
of all special extremal secant lines to $X$ the \textit{extremal variety} of $X$. Observe that
$$\mathbb{F}(X) \subseteq \mathbb{F}^+(X).$$
Using these concepts, we can formulate the following survey of the structure theory of surfaces of maximal sectional regularity.

\begin{theorem}\label{1.2 Theorem} Let $r \geq 5$ and let $X \subset \mathbb{P}^r$ be a non-degenerate irreducible projective surface of degree $d \geq r+1$ which is of
maximal sectional regularity and not a cone. Then
\begin{itemize}
\item[\rm{(a)}] (See Theorem~\ref{3.2' Theorem} (a)(3)) There is a smooth rational normal surface scroll $\widetilde{X} \subset \mathbb{P}^{d+1}$ and a linear subspace
                $\Lambda \in \mathbb{G}(d-r,\mathbb{P}^{r+1})$ which avoids $\widetilde{X}$ and such that the linear projection map
$$\pi'_{\Lambda}:\mathbb{P}^{d+1} \setminus \Lambda \twoheadrightarrow \mathbb{P}^r$$
induces a finite birational morphism
$$\pi_{\Lambda}:= \pi'_{\Lambda}\upharpoonright  \widetilde{X} \twoheadrightarrow X.$$
\item[\rm{(b)}] (See Theorem~\ref{3.2' Theorem} (b)) $\reg(X) = d-r+3$, so that the surface $X$ satisfies in particular the Eisenbud-Goto conjecture.
\item[\rm{(c)}] (See Theorem~\ref{4.11'' Theorem} (a), (b) and Theorem~\ref{4.14'' Theorem} (a)) Either
                \begin{itemize}
                \item[\rm{(1)}] $r = 5$, $X \subset \mathbb{F}(X)$ and $\mathbb{F}(X)$ is projectively equivalent to the rational normal threefold scroll
                                $S(1,1,1) \subset \mathbb{P}^5$, or else
                \item[\rm{(2)}] $\mathbb{F}(X) = \mathbb{P}^2$ and $X\cap \mathbb{F}(X)$ is a plane curve of degree $d-r+3$.
                \end{itemize}
\item[\rm({d})] (See Proposition~\ref{4.3'' Lemma} (a),(b), Theorem~\ref{4.11'' Theorem} (e) and Theorem~\ref{4.14'' Theorem} (h)) The set ${}^*\Sigma^{\circ}_{d-r+3}(X)$
                of all special extremal secant lines to $X$ is open in its closure $\overline{{}^*\Sigma^{\circ}_{d-r+3}(X)}$ and hence locally closed in the
                Grassmannian $\mathbb{G}(1,\mathbb{P}^r)$. Moreover the following statements hold
                \begin{itemize}
                \item[\rm{(1)}] In the case (1) of statement (c) we have ${}^*\Sigma^{\circ}_{d-2}(X) = \Sigma^{\circ}_{d-2}(X)$ and the image of
                                $\overline{{}^*\Sigma^{\circ}_{d-2}(X)}$ under the Pl\"ucker embedding
                                $\psi: \mathbb{G}(1,\mathbb{P}^5) \rightarrow \mathbb{P}^{14}$ is a Veronese surface in $\mathbb{P}^5$.
                \item[\rm{(2)}] In the case (2) of statement (c), the image of $\overline{{}^*\Sigma^{\circ}_{d-r+3}(X)}$ under the Pl\"ucker embedding
                                $\psi: \mathbb{G}(1.\mathbb{P}^r) \rightarrow \mathbb{P}^{\binom{r+1}{2}-1}$ is a plane.
                \end{itemize}
\end{itemize}
\end{theorem}

If the surface $X \subset \mathbb{P}^r$ is obtained as an outer
linear projection of a smooth rational normal surface scroll
$\widetilde{X} \subset \mathbb{P}^{d+1}$ from a subspace $\Lambda =
\mathbb{P}^{d-r} \subset \mathbb{P}^{d+1} \setminus \widetilde{X}$
we write $X = \widetilde{X}_{\Lambda}$. Keep in mind that the class
group ${\rm Cl}(\widetilde{Y})$ of a smooth rational normal $n$-fold
scroll $\widetilde{Y} \subset \mathbb{P}^s$ is generated by a
hyperplane section $H = \widetilde{Y} \cap \mathbb{P}^{s-1}$ and a
ruling $F = \mathbb{P}^{n-1}$ of $\widetilde{Y}$. Using this
terminology, we can classify all surfaces of maximal sectional
regularity as follows.

\begin{theorem}\label{1.3 Theorem} Let $r \geq 5$ and let $X \subset \mathbb{P}^r$ be a non-degenerate irreducible projective surface of degree $d \geq r+1$ which is of
maximal sectional regularity.
\begin{itemize}
\item[\rm{(a)}] (See Theorem~\ref{4.11'' Theorem} (a),(b),(e)) The following two statements are equivalent:
                \begin{itemize}
                \item[\rm{(i)}]  $r=5$ and $\mathbb{F}(X)$ is a rational $3$-fold scroll $S(1,1,1) \subset \mathbb{P}^5$.
                \item[\rm{(ii)}] $X$ is contained in a scroll $S(1,1,1) \subset \mathbb{P}^5$ as a divisor which is linearly equivalent to $H + (d-3)F$.
                \end{itemize}
\item[\rm{(b)}] (See Definition and Remark~\ref{4.2' Definition and Remark} (C) and Theorem~\ref{5.x^iv Theorem} (a)) The following two statements are equivalent:
                \begin{itemize}
                \item[\rm{(i)}]  $\mathbb{F}(X) = \mathbb{P}^2$.
                \item[\rm{(ii)}] $X$ is either
                                 \begin{itemize}
                                 \item[\rm{(1)}]  a cone over a curve $C \subset \mathbb{P}^{r-1} (\subset \mathbb{P}^r)$ of maximal regularity, or else
                                 \item[\rm{(2)}]  is equal to $\widetilde{X}_{\Lambda}$, where $\Lambda = \mathbb{P}^{d-r} \subset \mathbb{P}^{d+1}$ is contained in the
                                                  linear span $\langle D \rangle = \mathbb{P}^{d-r+3} \subset \mathbb{P}^{d+1}$ of a divisor $D \in |H + (3-r)F|$, and
                                                  the induced linear projection map
                                                  $$\pi'_{\Lambda} \upharpoonright: \langle D \rangle \setminus \Lambda \twoheadrightarrow \mathbb{P}^2$$
                                                  is generically one-to-one along $D$.
                                 \end{itemize}
               \end{itemize}
\end{itemize}
\end{theorem}

Finally, we study some cohomological and homological invariants of a surface $X \subset \mathbb{P}^r$ of degree $d$ which is of maximal sectional regularity.
By Theorem~\ref{1.2 Theorem} (b) we know that such a surface is $(d-r+3)$-regular, but not $(d-r+2)$-regular. In particular the triplet
$$\nu(X) := \big(h^1(\mathbb{P}^r,\mathcal{J}_X(d-r+1)), h^2(\mathbb{P}^r, \mathcal{J}_X(d-r)), h^3(\mathbb{P}^r,\mathcal{J}_X(d-r-1))\big)$$
is non-zero, and moreover the \textit{index of normality}
$$N(X) := \sup\{n \in \mathbb{Z} \mid h^1(\mathbb{P}^r, \mathcal{J}_X(n)) \neq 0\}$$
of $X$ cannot exceed $d-r+1$. To measure, how far the surface $X$ is away from being locally Cohen-Macaulay, we also introduce the invariant
$$\e(X) := \sum_{x\in X, closed} {\rm length}\big(H^1_{\mathfrak{m}_{X,x}}(\mathcal{O}_{X,x})\big),$$
which has the property that
$$h^2(\mathbb{P}^r,\mathcal{J}_X(n)) = \e(X) \mbox{ for all } n\ll 0.$$
Moreover, for a closed subscheme $Z \subset \mathbb{P}^r$ with homogeneous vanishing ideal $I_Z \subset S$, we use
$$\depth(Z) := \depth(S/I_Z)$$
to denote the \textit{arithmetic depth} of $Z \subset \mathbb{P}^r$.

Besides of these cohomological invariants of $X$, we also shall
investigate the syzygetic behavior of $X$, hence the \textit{Betti
numbers}
$$\beta_{i,j}(X) := \dim\big(\Tor^S_i(K,S/I_X)\big)_{i+j} \quad (\mbox{ with } K := S/\bigoplus_{n>0} S_n).$$
In the special case in which $r = 5$ and $\mathbb{F}(X) = S(1,1,1)$, we describe $X$ as a divisor on $S(1,1,1)$ and hence may determine some of the previously
introduced invariants (see Theorem~\ref{4.11'' Theorem}).

\begin{theorem}\label{1.4 Theorem} Let $X \subset \mathbb{P}^5$ be a surface of degree $d > 5$ which is of maximal sectional regularity and such that $\mathbb{F}(X)
 = S(1,1,1)$. Then
$$N(X) = d-4, \quad \depth(X) = 1, \quad \e(X) = 0, $$
$$\nu(X) = \big(\binom{d-3}{2},0,0\big) \mbox{ and } \beta_{1,b-3}(X) = \binom{d-1}{2}.$$
\end{theorem}

In the general situation, that is if $\mathbb{F}(X)$ is a plane, we have the following result.

\begin{theorem}\label{1.5 Theorem} Let $5 \leq r < d$ and let $X \subset \mathbb{P}^r$ be a surface of degree $d$ and maximal sectional regularity such that $\mathbb{F}(X) =
 \mathbb{P}^2$. Set $Y := X \cup \mathbb{F}(X)$. Then
\begin{itemize}
\item[\rm{(a)}] (See Theorem~\ref{4.14'' Theorem} (d)(2),(4) and (e)(2))
                \begin{itemize}
                \item[\rm{(1)}] $\nu(X) = (h^1(\mathbb{P}^r,\mathcal{J}_X(d-r+1)),1,0)$;
                \item[\rm{(2)}] $\e(X) = h^2(\mathbb{P}^r,\mathcal{J}_X) = h^2(\mathbb{P}^r,\mathcal{J}_Y) + \binom{d-r+2}{2}$.
                \end{itemize}
\item[\rm{(b)}] (See Theorem~\ref{4.14'' Theorem} (f)) The pair $\tau(X) := \big(\depth(X),\depth(Y)\big)$ satisfies
                \begin{itemize}
                \item[\rm{(1)}] $\tau(X) = (2,3)\quad \quad \quad \quad \quad \quad$ if $r+1 \leq d \leq 2r-4$;
                \item[\rm{(2)}] $\tau(X) \in \{(1,1),(2,2),(2.3)\}$ if $2r-3 \leq d \leq 3r-7$;
                \item[\rm{(3)}] $\tau(X) \in \{(1,1). (2,2)\}$ $\mbox{ }\quad \quad$ if $3r-6 \leq d$.
                \end{itemize}
\item[\rm{(c)}] (See Theorem~\ref{4.17'' Proposition} (a)) The following four conditions are equivalent
                \begin{itemize}
                \item[\rm{(i)}]   $N(X) \leq d-r$;
                \item[\rm{(ii)}]  $\reg(Y) \leq d-r+2$;
                \item[\rm{(iii)}] $\beta_{i,d-r+2}(X) = \binom{r-2}{i-1}$ for all $i \geq 1$.
                \item[\rm{(iv)}]  $\beta_{r,d-r+2}(X) = 0$.
                \end{itemize}
\item[\rm{(d)}] (See Theorem~\ref{4.17'' Proposition} (b)) The following two conditions are equivalent
                \begin{itemize}
                \item[\rm{(i)}]  $\beta_{1,d-r+2}(X) = 1$;
                \item[\rm{(ii)}] The homogeneous vanishing ideal $I_Y$ of $Y$ in $S$ is generated by the homogeneous polynomials of degree $\leq d-r+2$ which are
                                 contained in the homogeneous vanishing ideal $I_X$ of $X$ in $S$, thus $I_Y = \big((I_X)_{\leq d-r+2}\big)$.
                \end{itemize}
                Moreover, these two conditions hold, if the equivalent conditions of statement (c) are satisfied, and they imply that each proper extremal secant line
                to $X$ is contained in $\mathbb{F}(X)$ and hence that the $(d-r+3)$-secant variety $\Sec_{d-r+3}(X)$ of $X$ is equal to $Y$.
\end{itemize}
\end{theorem}

We shall provide examples of surfaces $X$ of extremal regularity,
which show that $\mathfrak{d}(X)$ can take all values in the set
$\{-1,0,1,2\}$ (see Construction and Examples~\ref{7.1 Construction
and Examples}), even for smooth surfaces $X$ which are sectionally
rational and occur as divisors on smooth rational normal threefold
scrolls -- and that there are indeed many such examples with
$\mathfrak{d}(X) = -1$, that is without extremal secant lines. This
latter fact is noteworthy as Gruson-Lazarsfeld-Peskine's paper
\cite{GruLPe} lead to the expectation, that there are only ``a few
exceptional varieties of extremal regularity having no extremal
secant line".

We also shall provide examples, which show that in the general case
where $X \subset \mathbb{P}^r$ is a surface of maximal sectional
regularity whose extremal variety $\mathbb{F}(X)$ is a plane, all
pairs $\tau(X) = \big(\depth(X),\depth(Y)\big)$ listed in
Theorem~\ref{1.4 Theorem} (b) may indeed occur (see Construction and
Examples ~\ref{7.2 Construction and Examples}, Examples~\ref{7.3
Example},~\ref{7.4 Example} and ~\ref{7.5 Example}).

Finally let us mention, the following problem, which we believe to
have an affirmative answer (see Problem and Remark~ \ref{7.6 Problem
and Remark}).

\begin{problem}\label{1.6 Problem} \textit{Let $5 \leq r < d$ and let $X \subset \mathbb{P}^r$ be a non-degenerate irreducible surface of degree $d$ which is of maximal
sectional regularity. Is it true, that the following three statements are equivalent?}
\begin{itemize}
\item[\rm{(i)}]   $N(X) \leq d-r$.
\item[\rm{(ii)}]  $\beta_{1,d-r+2}(X) = 1$.
\item[\rm{(iii)}] $\mathbb{F}(X) = \mathbb{P}^2$.
\end{itemize}
\end{problem}

\noindent We know, that the implications (i) $\Rightarrow$ (ii)
$\Rightarrow$ (iii) hold (see the implication (i) $\Rightarrow$
(iii) in statement (a) and the implication (i) $\Rightarrow$ (ii) in
statement (b) of Theorem~\ref{4.17'' Proposition}).

\section{Geometry of Extremal Secant Lines and Varieties of Maximal Sectional Regularity}
\label{2. Geometry of Extremal Secant Lines and Varieties of Maximal Sectional Regularity}

\noindent In this section, we consider the set of all
\textit{extremal secant lines}, i.e. of all proper $d-c+1$-secant
lines of a non-degenerate irreducible projective variety $X \subset
\mathbb{P}^r$ of degree $d$ and codimension $c < r$. Varieties,
which admit such extremal secant lines are of \textit{extremal
(Castelnuovo-Mumford) regularity}. Among the varieties of extremal
regularity, those of \textit{maximal sectional regularity} are of
particular interest for our investigation. The main result of this
section characterizes varieties of maximal sectional regularity as
those which ``have the largest possible set of extremal secant
lines". We first fix a few notations, which we shall keep for the
rest of our paper.

\begin{notation and convention}\label{2.1'' Notation and Convention} (A) By ${\mathbb N}_0$ and ${\mathbb N}$ we respectively denote the set of non-negative and
of positive integers. If $(Z, \mathcal{O}_Z)$ is a Noetherian scheme, we respectively denote by $\Reg(Z)$, $\CM(Z)$, $\Nor(Z)$, $S_2(Z)$ the \textit{locus
of regular, Cohen-Macaulay, normal} and \textit{of $S_2$-points} of $Z$. Moreover we denote the \textit{length} of the scheme $Z$ by $\#Z$, thus
\[
\#Z := \length(\mathcal{O}_Z) \quad (\in \mathbb{N}_0\cup\{\infty\}).
\]
The \textit{singular locus} of a morphism $f: Y \rightarrow Z$ of Noetherian schemes will be denoted by $\Sing(f)$. So, $\Sing(f)$ is the least
closed subset $W$ of $Z$ such that restriction of $f$ gives rise to an isomorphism
$$f\upharpoonright: Y \setminus f^{-1}(W) \stackrel{\cong}{\longrightarrow} Z \setminus W.$$
If $\Sing(f)$ is a finite set, we say that $f$ is \textit{almost non-singular}.\\
(B) Once for all we fix an algebraically closed field $K$, an integer $r \in \mathbb{N}$ and always write $S$ for the polynomial ring $K[x_0, \ldots , x_r]$.
We furnish $S$ with its standard grading and write $S_+$ for the irrelevant ideal $\bigoplus_{n \in \mathbb{N}} S_n = (x_0,\ldots,x_r)$ of $S$. If $\mathfrak{a}
\subset S$ is a graded ideal of $S$, we use $\mathfrak{a}^{\sat}$ to denote the \textit{saturation} $\bigcup_{n\in \mathbb{N}} (\mathfrak{a}:_S (S_+)^n)$ of $\mathfrak{a}$.\\
(C) If $Z \subseteq \mathbb{P}^r := {\rm Proj}(S)$ is a closed subscheme, $I_Z \subseteq S$ is used to denote the \textit{homogeneous vanishing ideal} of $Z$ in $S$ and
$\mathcal{J}_Z \subseteq \mathcal{O}_{\mathbb{P}^r}$ is used to denote the \textit{sheaf of vanishing ideals} of $Z$ in $\mathcal{O}_{\mathbb{P}^r}$. Keep in mind, that
$I_Z \subset S$ is a graded saturated ideal, which equals $S$ if and only if $Z = \emptyset$, that $\mathcal{J}_Z = \widetilde{I_Z}$ is the coherent sheaf of
$\mathcal{O}_{\mathbb{P}^r}$-modules induced by $I_Z$, and that $I_Z = H^0_*(\mathbb{P}^r,\mathcal{J}_Z) = \bigoplus_{n \in \mathbb{Z}}H^0(\mathbb{P}^r,\mathcal{J}_Z(n))$.\\
(D) Moreover $X \subset {\mathbb P}^r$ always denotes an irreducible non-degenerate variety with sheaf of vanishing ideals
${\mathcal J} := {\mathcal J}_X \subset {\mathcal O}_{{\mathbb P}^r_K}$, homogeneous vanishing ideal $I := I_X \subseteq S$ and \textit{homogeneous coordinate ring}
$A = A_X  := S / I$.
\end{notation and convention}

As (Castelnuovo-Mumford) regularity is of basic significance for our paper, we recall a few facts and define a few notions related to this invariant.

\begin{notation and reminder}\label{2.4'' Notation and Reminder} (A) If $M = \bigoplus_{n \in \mathbb{Z}} M_n$ is a finitely generated graded $S$-module, and
$i \in \mathbb{N}_0$, we  use
$$H^i(M) = \bigoplus_{n \in \mathbb{Z}} H^i(M)_n$$
to denote the $i$-th local cohomology module of $M$ with respect to the irrelevant ideal $S_+ = \bigoplus_{n\in \mathbb{N}} S_n$ of $S$, furnished with its natural grading.
Moreover we also set in this situation
$$h^i(M)_n := \dim_K\big(H^i(M)_n\big) \mbox{ for all } i \in \mathbb{N}_0 \mbox{ and all } n \in \mathbb{Z}.$$
Keep in mind that $h^i(M)_n$ is always finite and vanishes whenever
$i > \dim (M)$ or $n \gg 0$.

If $\mathcal{F} := \widetilde{M}$ is the coherent sheaf of
$\mathcal{O}_{\mathbb{P}^r}$-modules induced by $M$, the
\textit{Serre-Grothendieck Correspondence} yields an exact sequence
of graded $S$-modules
$$0 \longrightarrow H^0(M) \longrightarrow M \longrightarrow H^0_{*}(\mathbb{P}^r, \mathcal{F}) \longrightarrow H^1(M) \longrightarrow 0$$
and isomorphisms of graded $S$-modules
$$H^i_{*}(\mathbb{P}^r,\mathcal{F}) \cong H^{i+1}(M) \mbox{ for all } i \in \mathbb{N},$$
where $H^i_{*}(\mathbb{P}^r,\mathcal{F})$ denotes the graded
$S$-module $\bigoplus _{n \in \mathbb{Z}} H^i(\mathbb{P}^r,
\mathbb{F}(n))$.

In particular, if $\emptyset \neq Z \subsetneq \mathbb{P}^r$ is a
closed subscheme with sheaf of vanishing ideals $\mathcal{J}_Z
\subset \mathcal{O}_{\mathbb{P}^r}$ and homogeneous vanishing ideal
$I_Z = \bigoplus H^0(\mathbb{P}^r, \mathcal{J}_Z(n)) \subset S$, we
have
$$H^i(\mathbb{P}^r, \mathcal{O}_Z(n)) \cong H^i(Z,\mathcal{O}_Z(n)) \cong H^{i+1}(S/I_Z)_n \mbox{ for all } i \in \mathbb{N} \mbox{ and all } n \in \mathbb{Z}$$
and
$$H^i(\mathbb{P}^r,\mathcal{J}_Z(n)) \cong  H^i(S/I_Z)_n \mbox{ if either } i \neq r \mbox{ and } n \in \mathbb{Z}, \mbox{ or else } i = r \mbox{ and } n \geq -r.$$

\noindent (B) In the setting of part (A) and for any $k \in
\mathbb{N}_0$ we denote the \textit{(Castelnuovo-Mumford) regularity
of $M$ at and above level $k$} by $\reg^k(M)$, hence
$$\reg^k(M) = \inf\{a \in \mathbb{Z}\mid H^i(M)_{i+n} = 0, \forall i \geq k, \forall n>a\}.$$
The \textit{(Castelnuovo-Mumford) Regularity } $\reg(M)$ of $M$ is defined to be the regularity of $M$ at and above level $0$, thus
$$\reg(M) := \reg^0(M) = \inf\{a \in \mathbb{Z}\mid H^i_{S_+}(M)_{i+n} = 0, \forall i\in\mathbb{N}_0, \forall n>a\},$$
The \textit{(Castelnuovo-Mumford) regularity} of a coherent sheaf of $\mathcal{O}_{\mathbb{P}^r}$-modules $\mathcal{F}$
will be denoted by $\reg(\mathcal{F})$, hence
$$\reg(\mathcal{F}) =  \inf\{a \in \mathbb{Z}\mid H^i(\mathbb{P}^r, \mathcal{F}(i+n)) = 0, \forall i \in \mathbb{N}, \forall n \geq a\}.$$
If $\mathcal{F} := \widetilde{M}$ denotes the coherent sheaf of $\mathcal{O}_{\mathbb{P}^r}$-modules induced by the finitely generated graded $S$-module $M$, it
follows from the Serre-Grothendieck Correspondence that
$$\reg(\mathcal{F}) = \reg^2(M).$$
If $\emptyset \neq Z \subsetneq \mathbb{P}^r$ is a closed subscheme
with sheaf of vanishing ideals $\mathcal{J}_Z  \subseteq
\mathcal{O}_{\mathbb{P}^r}$ and homogeneous vanishing ideal $I_Z
\subset S$, the \textit{(Castelnuovo-Mumford) regularity} of $Z$ is
denoted by $\reg(Z)$. So, according to the observations made in part
(A) we have
$$\reg(Z) = \reg(\mathcal{J}_Z) = \reg(I_Z) = \reg(S/I_Z) + 1.$$
\end{notation and reminder}

We now define the notions of variety of extremal regularity and of maximal sectional regularity, which latter is the basic concept of our paper.

\begin{remark and definition}\label{2.3'' Remark and Definition} (A) Let the irreducible non-degenerate variety $X \subset \mathbb{P}^r$ be of degree $d$ and codimension
$c < r$. Then, the conjectural regularity inequality of Eisenbud-Goto \cite{EG} says that
$$\reg(X) \leq d-c+1.$$
We say that $X$ is of \textit{extremal regularity} if $\reg(X) \geq
d-c+1$.

\noindent (B) If $C = X \subset \mathbb{P}^r$ is a curve, the
conjectural inequality of part (A) holds according to
Gruson-Lazarsfeld-Peskine \cite{GruLPe}, so that $\reg(C) \leq
d-r+2$ in this case. Thus, $C$ is of extremal regularity if and only
if $\reg(C)$ takes its maximally possible value $d-r+2$. We
therefore say in this situation, that the curve $C$ is of
\textit{maximal regularity}. Curves of maximal regularity have some
important properties, which shall be of particular interest for our
later investigations. Namely according to \cite{GruLPe} and
\cite{BS2}[Remark 3.1 (C)] we can say:
\begin{itemize}\item[\rm{(a)}]  If $r \geq 3$ and $d > r+1$, each curve $C \subset \mathbb{P}^r$ of degree $d$ and of maximal regularity is smooth and rational.
               \item[\rm{(b)}]  If $r \geq 3$ and $d > r+1$, each curve $C \subset \mathbb{P}^r$ of degree $d$ and of maximal regularity has a $(d-r+2)$-secant line
                                $\mathbb{L} = \mathbb{P}^1 \subset \mathbb{P}^r$.
               \item[\rm{(c)}]  If in addition $r \geq 4$, the $(d-r+2)$-secant line of statement (b) is uniquely determined by $C$.
\end{itemize}

\noindent (C) If $k \leq r$ is a non-negative integer, we write
$$\mathbb{G}(k,\mathbb{P}^r) : = \{\mathbb{P}^k \mid \mathbb{P}^k \subseteq \mathbb{P}^r\}$$
for the \textit{Grassmannian} of all $k$-subspaces $\mathbb{P}^k
\subseteq \mathbb{P}^r$.

Let $X \subset \mathbb{P}^r$ be as in part (A). We say that $X$ is
of \textit{maximal sectional regularity} if $X \cap \mathbb{P}^{c+1}
\subset \mathbb{P}^{c+1}$ is a curve of maximal regularity for a
general space $\mathbb{P}^{c+1} \in \mathbb{G}(c+1, \mathbb{P}^r)$.
So, according to statement (a) of part (B) we can say:
\begin{itemize}
\item[] If $c \geq 2$ and if $X \subset \mathbb{P}^r$ is of maximal sectional regularity, then $X \cap \mathbb{P}^{c+1} \subset \mathbb{P}^{c+1}$ is a non-degenerate smooth
and rational curve of degree $d$ and of regularity $d-c+1$ for a general space $\mathbb{P}^{c+1} \in \mathbb{G}(c+1, \mathbb{P}^r)$.
\end{itemize}

\noindent (D) If $Z \subset \mathbb{P}^r$ is a closed subscheme of
dimension $> 1$ and $\mathbb{H} = \mathbb{P}^{r-1} \subset
\mathbb{P}^r$ is a general hyperplane, we have $\reg(Z \cap
\mathbb{H}) \leq \reg(Z)$. So, by induction on the dimension
$\dim(X)$ of $X$, we obtain:
\begin{itemize}
\item[] A variety $X \subset \mathbb{P}^r$ of maximal sectional regularity is of extremal regularity.
\end{itemize}
\end{remark and definition}

The previously defined concepts are intimately related to the existence of highly secant lines. We therefore recall a few preliminary facts on secant lines
and secant varieties.

\begin{notation and reminder}\label{4.1'' Notation and Reminder} (A) Let $Z \subset \mathbb{P}^r$ be a closed subscheme, let $m \in \mathbb{N}_0$ and consider the set
$$\Sigma_m(Z) := \{\mathbb{L} \in \mathbb{G}(1,\mathbb{P}^r) \mid \#(Z \cap \mathbb{L}) \geq m\}$$
of all $m$-\textit{secant lines} to $Z$. This set is closed in $\mathbb{G}(1,\mathbb{P}^r)$. To see this, let $d \in \mathbb{N}$ be such that
the homogeneous vanishing ideal $I_Z \subset S = K[x_0,\ldots,x_r]$ is generated by homogeneous polynomials of degree $\leq d$. Let $\mathbb{L} \in
\mathbb{G}(1,\mathbb{P}^r)$ and let $I_{\mathbb{L}} \subset S$ denote the homogeneous vanishing ideal of $\mathbb{L}$. Then, the ideal $I_Z + I_{\mathbb{L}} \subset S$
is generated by homogeneous polynomials of degree $\leq d$ so that its regularity is bounded in terms of $d$ and $r$ only -- for example by the inequality
$\reg(I_Z + I_{\mathbb{L}}) \leq (2d)^{2^{r-1}}$ (see \cite{CaS}). Now, fix some integer
$t \geq \max\{(2d)^{2^{r-1}}, m\}$. Then, the vanishing ideal $I_{Z \cap \mathbb{L}} = (I_Z + I_{\mathbb{L}})^{\sat} \subset S$ of $Z \cap \mathbb{L}$ coincides
with $I_Z + I_{\mathbb{L}}$ in all degrees $\geq t$. With $N := \dim_K (S_t)$ it follows, that $\mathbb{L} \in \Sigma_m(Z)$ if and only if
$\dim_K((I_Z)_t + (I_{\mathbb{L}})_t) \leq N-m$. This means, that the set $\Sigma_m(Z)$ is the preimage of the set
$$V := \{\mathbb{P} \in \mathbb{G}(N-t-2, |\mathcal{O}_{\mathbb{P}^r}(t)|) \mid \dim\langle \mathbb{P}, |\mathcal{J}_Z(t)|\rangle < N-m\} \subset
\mathbb{G}(N-t-2, |\mathcal{O}_{\mathbb{P}^r}(t)|)$$
under the morphism
$$\Phi:\mathbb{G}(1,\mathbb{P}^r) \rightarrow \mathbb{G}(N-t-2, |\mathcal{O}_{\mathbb{P}^r}(t)|), \quad \mathbb{L} \mapsto |\mathcal{J}_{\mathbb{L}}(t)|
= |(I_\mathbb{L})_t|.$$ According to \cite[Chapter 6]{H} the above
set $V$ is closed in the Grassmannian $\mathbb{G}(N-t-2,
|\mathcal{O}_{\mathbb{P}^r}(t)|)$. Therefore $\Sigma_m(Z)$ is a
closed subset of the Grassmannian $\mathbb{G}(1,\mathbb{P}^r)$.

\noindent (B) Keep the previous notations and hypotheses, let $T
\subset \mathbb{G}(1,\mathbb{P}^r)$ be a closed set and consider the
\textit{coincidence variety}
$$Y(T) := \{(x,\mathbb{L}) \in \mathbb{P}^r \times \mathbb{G}(1,\mathbb{P}^r) \mid x \in \mathbb{L} \in T \}$$
which is a closed subset of $\mathbb{P}^r \times \mathbb{G}(1,\mathbb{P}^r)$. Moreover, consider the two morphisms
$$p: Y(T) \rightarrow \mathbb{P}^r \mbox{ and } q:Y(T) \rightarrow \mathbb{G}(1,\mathbb{P}^r).$$
induced by the canonical projections. Then, we have
$$\mathcal{S}(T) := \bigcup_{\mathbb{L} \in T} \mathbb{L} = p(q^{-1}(T))$$
and hence $\mathcal{S}(T) \subset \mathbb{P}^r$ is closed.

Applying this to the closed set $T : = \Sigma_m(Z) \subset
\mathbb{G}(1,\mathbb{P}^r)$, we obtain that the $m$-\textit{secant
variety}
$$\Sec_m(Z) := \mathcal{S}(\Sigma_m(Z)) = \bigcup_{\mathbb{L} \in \Sigma_m(Z)} \mathbb{L} \subset \mathbb{P}^r$$
of $Z$ is closed in $\mathbb{P}^r$.

\noindent (C) We also shall use the notations
$$\Sigma_{\infty}(Z) := \{\mathbb{L} \in \mathbb{G}(1,\mathbb{P}^r) \mid \#(Z \cap \mathbb{L}) = \infty\} = \{\mathbb{L} \in \mathbb{G}(1,\mathbb{P}^r) \mid \mathbb{L}
\subseteq Z\}$$
and
$$\Sec_{\infty}(Z) := \mathcal{S}(\Sigma_{\infty}(Z)) = \bigcup_{L \in \Sigma_{\infty}(Z)} \mathbb{L} = \bigcup_{\mathbb{P}^1 = \mathbb{L} \subseteq X} \mathbb{L}.$$
Observe that we have the inclusions
$$\Sigma_{\infty}(Z) \subseteq \Sigma_m(Z) \mbox{ and } \Sec_{\infty}(Z) \subseteq \Sec_m(Z) \mbox{ for all } m \in \mathbb{N}_0$$
with equality at both places if the vanishing ideal $I_Z \subseteq S$ is generated by polynomials of degree $< m$. Hence
$\Sigma_{\infty}(Z) \subseteq \mathbb{G}(1,\mathbb{P}^r)$ and $\Sec_{\infty}(Z) \subseteq \mathbb{P}^r$ are closed subsets, too.
In particular, for each $m \in \mathbb{N}_0$, the set
$$\Sigma^{\circ}_m(Z) := \Sigma_m(Z) \setminus \Sigma_{\infty}(Z)$$
of \textit{proper} $m$-\textit{secant lines} to $Z$ is locally closed in $\mathbb{G}(1,\mathbb{P}^r)$. Moreover $\Sigma^{\circ}_m(Z) \neq \emptyset$ implies that
$I_Z$ needs homogeneous generators of degree $\geq m$. So, for each $m \in \mathbb{N}_0$ we have the implication:
$$ \mbox{ If } \Sigma^{\circ}_m(Z) \neq \emptyset, \mbox{ then } \reg(Z) \geq m.$$
\end{notation and reminder}

In the sequel we are interested in the set $\Sigma^{\circ}_{d-c+1}(X)$ of proper extremal secant lines to a variety $X \subset \mathbb{P}^r$ of codimension $c$ and degree $d$.
Among these proper extremal secant lines, those which are extremal secant lines of a curve of maximal regularity $X \cap \mathbb{E}$ with $\mathbb{E} \subset
\mathbb{G}(c+1, \mathbb{P}^r)$ will be of particular interest later. We therefore introduce the following notions.

\begin{definition and remark}\label{4.1''' Definition and Remark} (A) Let  $X \subset \mathbb{P}^r$ be a non-degenerate irreducible projective variety of codimension
$ c < r$ and degree $d$, and let
$$\mathbb{P}\mathbb{U}(X) :=\{\mathbb{E} \in \mathbb{G}(c+1,\mathbb{P}^r) \mid X \cap \mathbb{E} \subset \mathbb{E} \mbox{ is an integral curve of maximal regularity}\}.$$
Observe, that this set contains a non-empty open subset of
$\mathbb{G}(c+1,\mathbb{P}^r)$ if and only if $X$ is of maximal
sectional regularity.

\noindent (B) Keep the above notations, assume that $c \geq 3$ and
$d > c+2$ and let $\mathbb{E} \in \mathbb{P}\mathbb{U}(X)$. Then,
according to statements (b) and (c) of Remark and
Definition~\ref{2.3'' Remark and Definition} (B) the curve of
maximal regularity $X \cap \mathbb{E} \subset \mathbb{E} =
\mathbb{P}^{c+1}$ has a unique $(d-c+1)$-secant line, which we
denote by $\mathbb{L}_{\mathbb{E},X}$. So, this secant line is
characterized by the property
$$\#\big((X \cap \mathbb{E}) \cap \mathbb{L}_{\mathbb{E},X}\big) = \#(X \cap \mathbb{L}_{\mathbb{E},X}) = d-c+1.$$
Using this notation we define the set
$$ {}^*\Sigma^{\circ}_{d-c+1}(X) := \{ \mathbb{L}_{\mathbb{E},X} \mid \mathbb{E} \in \mathbb{P}\mathbb{U}(X)\} \subseteq \Sigma^{\circ}_{d-c+1}(X)$$
and call the lines, which belong to this set the \textit{special
extremal secant lines} to $X$.

\noindent (C) Keep the previous notations and assume that $3 \leq c
< r-1$ and $d > c+2$. Let $\mathbb{H} = \mathbb{P}^{r-1} \subset
\mathbb{P}^r$ be a general hyperplane so that $X \cap \mathbb{H}
\subset \mathbb{H} = \mathbb{P}^{r-1}$ is a non-degenerate reduced
and irreducible variety of codimension $c \geq 3$ and of degree $d >
c+2$. Let $\mathbb{E} \in \mathbb{P}\mathbb{U}(X \cap \mathbb{H})$.
Then $X \cap \mathbb{E} = (X \cap \mathbb{H}) \cap \mathbb{E}$
yields that $\mathbb{E} \in \mathbb{P}\mathbb{U}(X)$ and
$\mathbb{L}_{\mathbb{E},X \cap \mathbb{H}} =
\mathbb{L}_{\mathbb{E},X}$. Therefore we get
$$ \mathbb{P}\mathbb{U}(X \cap \mathbb{H}) \subseteq \mathbb{P}\mathbb{U}(X) \cap \mathbb{G}(c+1,\mathbb{H})
 \mbox{ and } {}^*\Sigma^{\circ}_{d-c+1}(X \cap \mathbb{H}) \subseteq {}^*\Sigma^{\circ}_{d-c+1}(X) \cap \mathbb{G}(1, \mathbb{H}).$$
Thus we have
$$\overline{{}^*\Sigma^{\circ}_{d-c+1}(X \cap \mathbb{H})} \subseteq \overline{{}^*\Sigma^{\circ}_{d-c+1}(X)} \cap \mathbb{G}(1, \mathbb{H}).$$
\end{definition and remark}
\smallskip

\begin{definition and remark}\label{2.1'' Definition and Remark}
(A) Let $X \subset \mathbb{P}^r$ be an irreducible projective variety of degree $d$ and codimension $c < r$. The proper $(d-c+1)$-secant lines to $X$, hence the lines
$\mathbb{L}$ which belong to the locally closed subset $\Sigma^{\circ}_{d-c+1}(X)$ (see Notation and Reminder~\ref{4.1'' Notation and Reminder} (C)) are called
\textit{proper extremal secant lines} to $X$. To measure the size of the set of all proper extremal secant lines to $X$, we introduce the invariant
$$ \mathfrak{d}(X) : = \dim\big(\Sigma^{\circ}_{d-c+1}(X)\big) = \dim\big(\overline{\Sigma^{\circ}_{d-c+1}(X)}\big)$$
with the usual convention that $\dim(\emptyset) = -1$. According to Notation and Reminder~\ref{4.1'' Notation and Reminder} (C) we can say
$$\mbox{ If } \mathfrak{d}(X) \geq 0, \mbox{ then } X \subset \mathbb{P}^r \mbox{ is a variety of extremal regularity.}$$

\noindent (B) Keep the notations of part (A) and assume that $3\leq
c < r$ and $d > c+2$. To measure the size of the set
${}^*\Sigma^{\circ}_{d-c+1}(X)$ (see Definition and
Remark~\ref{4.1''' Definition and Remark}) of all special extremal
secant lines to $X$, we define the invariant
$$ {}^*\mathfrak{d}(X) := \dim\big(\overline{{}^*\Sigma^{\circ}_{d-c+1}(X)}\big).$$
\end{definition and remark}

Now, we are heading for the main result of this section. We begin with the following auxiliary result.

\begin{lemma}\label{2.2'' Lemma}
Let $\Sigma$ be a closed subset of $\mathbb{G}(1,\mathbb{P}^r)$ and let $\mathbb{H} = \mathbb{P}^{r-1} \subset \mathbb{P}^r$ be a general hyperplane.
Then, the following statements hold.
\begin{itemize}
\item[\rm{(a)}] If $\dim(\Sigma) \leq 1$, then $\Sigma \cap \mathbb{G}(1,\mathbb{H}) = \emptyset$.
\item[\rm{(b)}] If $\dim(\Sigma) \geq 2$, then each irreducible component $W$ of $\Sigma \cap \mathbb{G}(1,\mathbb{H})$ satisfies
                   $$\dim(W) = \dim(\Sigma) - 2.$$
\end{itemize}
\end{lemma}

\begin{proof}
A result of Kleiman (see \cite[Corollary 4]{K}) says that for an
integral algebraic group scheme $G$, an integral scheme $X$ with
transitive $G$-action, for two closed integral subschemes $Y, Z
\subset X$ and for general $g \in G$, all irreducible components of
$g(Y) \cap Z$ have dimension $\dim(Y)+\dim(Z)-\dim(X)$.

If we fix a hyperplane $\mathbb{H}_0 \subset \mathbb{P}^r$ and apply
this result with $G = {\rm Aut}(\mathbb{P}^r)$, $X =
\mathbb{G}(1,\mathbb{P}^r)$, $Y := \mathbb{G}(1,\mathbb{H}_0)$ and
$Z = \Sigma$, and keep in mind that
$$\dim(\mathbb{G}(1,\mathbb{P}^r)) - \dim(\mathbb{G}(1,\mathbb{H}_0)) = 2(r-2) - 2(r-1) = 2,$$
we get our claim.
\end{proof}

\begin{theorem}\label{2.3'' Theorem}
Assume that $1 \leq t \leq r-3$ and let $X \subset \mathbb{P}^r$ be a non-degenerate irreducible projective variety of dimension $t$ and degree $> r-t+2$.
Then, the following statements hold.
\begin{itemize}
\item[\rm{(a)}] ${}^*\mathfrak{d}(X) = \mathfrak{d}(X) \leq 2t-2$.
\item[\rm{(b)}] $X$ is of maximal sectional regularity if and only if $\mathfrak{d}(X) = 2t-2$.
\end{itemize}
\end{theorem}

\begin{proof}
Let $d := \deg(X)$ and let $c: = r-t$. As ${}^*\Sigma^{\circ}_{d-c+1}(X) \subseteq \Sigma^{\circ}_{d-c+1}(X)$ we have the inequality
$${}^*\mathfrak{d}(X) \leq \mathfrak{d}(X).$$
We proceed by by induction on $t = \dim(X)$.

First, let $t = 1$, so that $r \geq 4$ and $X \subset \mathbb{P}^r$
is a curve of degree $d
> r+1$. If $\mathfrak{d}(X) \geq 0$, it follows by Definition and
Remark~\ref{2.1'' Definition and Remark} (A), that $X$ is of maximal
regularity. But then, by Definition and Remark~\ref{2.3'' Remark and
Definition} (B) (c), $X$ has precisely one extremal secant line, so
that $\mathfrak{d}(X) = {}^*\mathfrak{d}(X) = 0$. This observation
proves statements (a) and (b) in the case $t = 1$.

So, let $t > 1$ and let $\mathbb{H} = \mathbb{P}^{r-1} \subset
\mathbb{P}^r$ be a general hyperplane. Then $X \cap \mathbb{H}
\subset \mathbb{H} = \mathbb{P}^{r-1}$ is a non-degenerate
irreducible projective variety of codimension $c \geq 3$, of
dimension $t-1$ and of degree $d > c+2$ -- which is of maximal
sectional regularity if and only if $X \subset \mathbb{P}^r$ is. So
by induction we first have
$${}^*\mathfrak{d}(X \cap \mathbb{H}) = \mathfrak{d}(X \cap \mathbb{H}) \leq 2t-4,$$
with equality at the second place if and only if and only if $X\cap
\mathbb{H} \subset \mathbb{H} = \mathbb{P}^{r-1}$ is of maximal
sectional regularity -- hence, if and only if $X \subset
\mathbb{P}^r$ is of maximal sectional regularity.

Moreover, we have
$$\Sigma^{\circ}_{d-c+1}(X \cap \mathbb{H}) = \Sigma^{\circ}_{d-c+1}(X) \cap \mathbb{G}(1,\mathbb{H}).$$
On application of Lemma~\ref{2.2'' Lemma}, with $\Sigma := \overline{\Sigma^{\circ}_{d-c+1}(X)}$ it follows hat
$$\mathfrak{d}(X) = \dim\big(\Sigma^{\circ}_{d-c+1}(X)\big) = \dim\big(\Sigma^{\circ}_{d-c+1}(X \cap \mathbb{H})\big) + 2 = \mathfrak{d}(X\cap\mathbb{H}) + 2 \leq 2t-2,$$
with equality at the last place if and only if $X$ is of maximal
sectional regularity.

It remains to show that ${}^*\mathfrak{d}(X) \geq \mathfrak{d}(X)$.
In view of Definition and Remark~\ref{4.1''' Definition and Remark}
(C) we have
$$\mathfrak{d}(X \cap \mathbb{H}) = {}^*\mathfrak{d}(X \cap \mathbb{H}) = \dim\big(\overline{{}^*\Sigma^{\circ}_{d-c+1}(X \cap \mathbb{H})}\big) \leq
\dim\big(\overline{{}^*\Sigma^{\circ}_{d-c+1}(X)} \cap \mathbb{G}(1, \mathbb{H})\big).$$
If we apply Lemma~\ref{2.2'' Lemma} with $\Sigma = \overline{{}^*\Sigma^{\circ}_{d-c+1}(X)}$, we get
$${}^*\mathfrak{d}(X) = \dim\big(\overline{{}^*\Sigma^{\circ}_{d-c+1}(X)}\big) = \dim\big(\overline{{}^*\Sigma^{\circ}_{d-c+1}(X)} \cap \mathbb{G}(1, \mathbb{H})\big) + 2$$
and hence
$${}^*\mathfrak{d}(X) \geq \mathfrak{d}(X \cap \mathbb{H}) + 2 = \mathfrak{d}(X).$$
\end{proof}


\section{Sectionally Rational Varieties}
\label{2'. Sectionally rational Varieties}

\noindent As we shall see, surfaces of maximal sectional regularity
have at most finitely many singular points and their generic
hyperplane section is a smooth rational curve. As a consequence,
these surfaces are almost non-singular projections of rational
normal surface scrolls. In this section, we approach this fact in a
more general setting, by showing first, that (under certain
restrictions) projective varieties whose general linear section
curve is rational, are birational projections of
varieties of minimal degree. We first fix some notations concerning projections.\\

\begin{notation and convention}\label{2.1' Notation and Convention}
Let $r' \geq r$ and let $X' \subset \mathbb{P}^{r'}$ be an irreducible and
non-degenerate variety. If $\Lambda = \mathbb{P}^{r'-r-1} \subset \mathbb{P}^{r'}$ is a linear subspace,
such that $X' \cap \Lambda = \emptyset$ and $\pi'_{\Lambda}:\mathbb{P}^{r'}\setminus\Lambda
\twoheadrightarrow \mathbb{P}^r$ is a \textit{linear projection with center $\Lambda$}, we write
$X'_{\Lambda} := \pi'_{\Lambda}(X')$ for the \textit{projected image} $\pi'_{\Lambda}(X') \subset \mathbb{P}^r$ of $X'$ and
$\pi_{\Lambda}: X' \twoheadrightarrow X'_{\Lambda}$ for the finite morphism induced by the projection $\pi'_{\Lambda}$. If
$X = X'_{\Lambda}$, we say that $X$ is \textit{a projection of $X'$ (from the center $\Lambda$)} and call $X' \subset \mathbb{P}^{r'}$
a \textit{projecting variety} of $X$. If the induced projection morphism $\pi_{\Lambda}:X'\twoheadrightarrow X$ is in addition almost
non-singular, we say that $X$ is an \textit{almost non-singular projection of $X'$ (from the center $\Lambda$)}.
\end{notation and convention}

Next, we define the basic concept of this section.

\begin{definition}\label{2.2' Definition} A non-degenerate irreducible variety $X\subset \mathbb{P}^r$ of codimension $c < r$ is said to be \textit{sectionally rational}
if $X \cap \mathbb{P}^{c+1} \subset \mathbb{P}^{c+1}$ is a (possibly singular) rational curve for a general space $\mathbb{P}^{c+1} \in \mathbb{G}(c+1,\mathbb{P}^r)$.
\end{definition}

Now, we are ready to prove the announced result on sectionally
rational varieties. The reader should be aware of the fact, that in
our proof, we use a Bertini Theorem for linear systems of ample
divisors found in \cite{FlOV}, which is known only in characteristic
$0$. Nevertheless, in the case of sectionally rational surfaces with
finite non-normal locus -- the case we are actually heading for --
we can avoid the use of the mentioned Bertini Theorem and thus get a
characteristic free statement.

\begin{theorem}\label{2.3' Theorem}
Let $X \subset \mathbb{P}^r$ be sectionally rational. Set $t := \dim X$ and $d := \deg X$. Assume that either
\begin{itemize}
\item[\rm{(1)}] $\Char K = 0$, or else
\item[\rm{(2)}] $t=2$ and $X \setminus \Nor(X)$ is finite.
\end{itemize}
Then, $X$ is a projection of a variety of minimal degree. More
precisely, there is a variety $\widetilde{X} \subset \mathbb{P}^{d+t-1}$ of minimal degree and a subspace
$\Lambda = \mathbb{P}^{d+t-r-2} \subset \mathbb{P}^{d+t-1}$ with $\widetilde{X} \cap \Lambda = \emptyset$ such that
(in the notation introduced in Notation and Convention~\ref{2.1' Notation and Convention})
\begin{itemize}
\item[\rm{(a)}] $\widetilde{X}_{\Lambda} = X$;
\item[\rm{(b)}] $X \setminus \Nor(X) = \Sing(\pi_{\Lambda}:\widetilde{X} \twoheadrightarrow X)$.
\end{itemize}
\end{theorem}

\begin{proof}
Let $\nu: Y \rightarrow X$ be the normalization of $X$, so that $Y$ is a normal projective variety and
$\nu$ is a finite surjective morphism with $\Sing(\nu) = X \setminus \Nor(X)$. Consider the ample invertible sheaf of $\mathcal{O}_Y$-modules
$\mathcal{L} := \nu^{*}(\mathcal{O}_X(1))$. Let $h_1,\ldots, h_{t-1} \in S_1$ be general linear forms and consider the irreducible
varieties $X_i := X \cap \Proj(S/ \sum^{t-i}_{j=1}h_jS)$ and their preimages $Y_i := \nu^{-1}(X_i) \quad (i=1,\ldots t)$.
As the linear forms $h_j$ are general, $X_i$ is not contained in $\Sing(\nu)$ and so, the schemes $Y_i$ are irreducible and the induced finite morphisms
$$\nu_i := \nu \upharpoonright : Y_i \twoheadrightarrow X_i, \quad i = 1\ldots, t $$
are birational. Assume first, that $\Char K = 0$. As the closed
subscheme $Y_i \subset Y$ is cut out by the $t-i$ general divisors
$\nu^{*}(h_j) \in |\mathcal{L}| \quad (j = 1,\ldots, t-i)$, it is
normal by the Bertini Theorem \cite[Corollary 3.4.2]{FlOV}.  So, the
sequence $Y_1 \subset Y_2 \subset \ldots \subset Y_t = Y$ forms a
ladder with normal rungs of the polarized variety $(Y,\mathcal{L})$
in the sense of Fujita \cite[Definition 3.1, pg.28]{Fu}. As $\nu_1:
Y_1\twoheadrightarrow X_1$ is birational, it follows that $Y_1 \cong
\mathbb{P}^1$.

Assume now, that $t = 2$ and the non-normal locus $X \setminus
\Nor(X) = \Sing(\nu)$ of $X$ is finite. As $X$ is a surface, it
follows that the singular locus $X \setminus \Reg(X)$ is finite too,
So, as $h_1$ is general, the curve $X_1 \subset X$ is smooth and
disjoint to $\Sing(\nu)$. It follows that $X_1 \cong \mathbb{P}^1$
and $\nu_1:Y_1 \twoheadrightarrow X_1$ is an isomorphism, so that
also in this case $Y_1 \cong \mathbb{P}^1$. Hence, again we have a
ladder $Y_1 \subset Y_2 = Y$ of the polarized variety
$(Y,\mathcal{L})$ with normal rungs and $Y_1 \cong \mathbb{P}^1$.

Thus, in both cases the sectional genera in the sense of Fujita
\cite[2.1, pp.25,28]{Fu} satisfy $g(Y,\mathcal{L}) = g(Y_1,
\mathcal{L}\upharpoonright_{Y_1})= 0$. Therefore, by
\cite[Proposition 3.4]{Fu}, the $\Delta$-genus
$\Delta(X,\mathcal{L})$ of the polarized variety $(X,\mathcal{L})$
equals zero. So, according to \cite[Corollary 4.12]{Fu} we have
$$\Gamma_{*}(Y,\mathcal{L}) := \bigoplus_{n\in \mathbb{N}_0} \Gamma(Y,\mathcal{L}^{\otimes^n}) = K[\Gamma(Y,\mathcal{L})]$$
and $\mathcal{L}$ is very ample over $Y$. According to Fujita's Classification Theorem \cite[Theorem 5.15]{Fu}, it now follows,
that $|\mathcal{L}|$ induces a closed immersion $i: Y \rightarrow \mathbb{P}^s = \mathbb{P}(\Gamma(Y,\mathcal{L}))$  such that $\widetilde{X} := i(Y) \subset
\mathbb{P}^s$ is a variety of minimal degree and that, in addition, there is a projection $\pi'_{\Lambda}: \mathbb{P}^s\setminus \Lambda \twoheadrightarrow \mathbb{P}^r$
from a subspace $\Lambda = \mathbb{P}^{s-r-1} \subset \mathbb{P}^s$ with $\widetilde{X} \cap \Lambda = \emptyset$ such that $\widetilde{X}_{\Lambda} = X$
and $\nu = \pi_{\Lambda} \circ i$. As $\pi_{\Lambda}: \widetilde{X} \twoheadrightarrow X$ is birational, we have
$\dim \widetilde{X} = t$ and $\deg \widetilde{X} = d$. As $\widetilde{X} \subset \mathbb{P}^s$ is of minimal degree, it follows
that $d = s-t+1$ and hence $s = d+t-1$. This proves our claim.
\end{proof}

Observe, that according to the previous result, sectionally rational surfaces with finite non-normal locus are almost non-singular projections of
surface scrolls or the Veronese surface. One important consequence of this is, that these surfaces satisfy the conjectural inequality of Eisenbud-Goto \cite{EG} for the
Castelnuovo-Mumford regularity. Again, we shall give a more general approach to this fact and show, that an almost non-singular outer projection $X'_{\Lambda}$
from a subspace $\Lambda = \mathbb{P}^{r'-r-1} \subset \mathbb{P}^{r'}$ of a variety $X' \subset \mathbb{P}^{r'}$ which satisfies the Green-Lazarsfeld condition
$N_{2,p}$ for some $p > r'-r$, satisfies the Eisenbud-Goto inequality. We begin with some preparations concerning Betti numbers.

\begin{notation and reminder}\label{2.4' Notation and Reminder}
(A) We have $K = S/S_+$ and denote the $i$-th \textit{ Betti number} in degree $i+j$ of a finitely generated graded $S$-module $M \neq 0$ by
$$\beta_{i,j}(M) := \dim_K\big(\Tor^S_i(K,M)_{i+j}\big) \mbox{ for } i \in \mathbb{N} \mbox{ and } j \in \mathbb{Z}.$$
Keep in mind the well known fact, that the \textit{ initial degree }
$${\rm Indeg}(M) := \inf\{n \in \mathbb{Z} \mid M_n \neq 0\},$$
the regularity and the depth of $M$ bound the range of pairs of indices $(i,j)$  with non-zero Betti numbers as follows:
\begin{align*}
 {\rm Indeg}(M) &= \min\{j \in \mathbb{Z} \mid \beta_{0,j}(M)\neq 0 \} = \min\{j \in \mathbb{Z} \mid \beta_{i,j}(M) \neq 0 \mbox{ for some } i \in \mathbb{N}_0\} \\
\reg(M) &= \max\{j \in \mathbb{Z} \mid \beta_{i,j}(M) \neq 0 \mbox{ for some } i \in \mathbb{N}_0 \},\\
\depth(M) &= \max\{i \in \mathbb{N}_0 \mid \beta_{r+1-i,j}(M) \neq 0 \mbox{ for some } j \in \mathbb{Z}\}.
\end{align*}
Moreover, if $i \in \mathbb{N}_0$ and $l \in \mathbb{Z}$ are such
that $\beta_{i,j}(M) = 0$ for all $j \leq l$, then $\beta_{k,j}(M) =
0$ for all $k \geq i$ and all $j \leq l$.

Finally, if $I \subsetneq S$ is a homogeneous ideal, we have
$$\beta_{i,j}(I) = \beta_{i+1,j-1}(S/I) \mbox{ for all } i \in \mathbb{N}_0 \mbox{ and all } j \in \mathbb{Z}.$$
If $\emptyset \neq Z \subsetneq \mathbb{P}^r$ is a closed subscheme with homogeneous vanishing ideal $I_Z \subset S$, we define the \textit{Betti numbers} of $Z$ by
$$\beta_{i,j} = \beta_{i,j}(Z) := \beta_{i,j}(S/I_Z) = \beta_{i-1,j+1}(I_Z) \mbox{ for all } i \in \mathbb{N} \mbox{ and all } j \in \mathbb{Z}.$$
Observe that $\beta_{0,0}(Z) = 1$ and $\beta_{0,j} = 0$ for all
$j\neq 0$. Moreover, we have $\depth(S/I_Z) = \depth(Z) > 0$. In
addition by Notation and Reminder~ \ref{2.4'' Notation and Reminder}
(B), we have $\reg(S/I_Z) = \reg(Z) -1$. Finally if $Z$ is
non-degenerate, we have $\beta_{1,0}(Z) = 0$ and hence
$\beta_{i,0}(Z) = 0$ for all $i \geq 1$. So, for a non-degenerate
closed subscheme $\emptyset \neq Z \subsetneq \mathbb{P}^r$ we
usually only list the Betti numbers
$$\beta_{i,j}(Z) \quad \mbox{ with } \quad 1 \leq i \leq r + 1 - \depth(Z) \quad(\leq r) \quad \mbox{ and } \quad 1 \leq j < \reg(Z).$$
(B) Let $p \in \mathbb{N}$. The graded ideal $I \subset S$ is said to satisfy
the \textit{(Green-Lazarsfeld) property} $N_{2,p}$ (see \cite{GL}) if the Betti numbers of $S/I$ satisfy the condition
$$\beta_{i,j} := \beta^S_{i,j}(S/I) = \beta^S_{i-1,j+1}(I) = 0 \mbox{ whenever } i \leq p \mbox{ and } j \neq 1,$$
-- hence if and only if the minimal free resolution of $I$ -- up to the homological degree $p$ -- has the form
$$\ldots \rightarrow S^{\beta_{p,1}}(-p-1)\rightarrow \ldots \rightarrow S^{\beta_{1,1}}(-2) \rightarrow I \rightarrow 0.$$
The closed subscheme $Z \subset \mathbb{P}^r$ is said to satisfy the \textit{property} $N_{2,p}$ if its homogeneous vanishing ideal $I_Z \subset S$ satisfies
the property $N_{2,p}$.
\end{notation and reminder}

Now, we may prove the announced regularity bound for almost non-singular projections of $N_{2,p}$-varieties.

\begin{theorem}\label{2.5' Theorem} Let $r' \geq r$ be integers, let $X' \subset \mathbb{P}^{r'}$ be a non-degenerate
projective variety of dimension $ \geq 2$ which satisfies the
property $N_{2,p}$ for some $p \geq\max\{2,r'-r+1\}$. Let $\Lambda =
\mathbb{P}^{r'-r-1}$ be a subspace such that $X' \cap \Lambda =
\emptyset$. Let
\begin{equation*}
X := X'_{\Lambda} \subset \mathbb{P}^r
\end{equation*}
and assume that the induced finite morphism $\pi_{\Lambda}: X'
\twoheadrightarrow X$ is almost non-singular. Then
\begin{itemize}
\item[\rm{(a)}] The homogeneous vanishing ideal $I_X \subset S$ of $X$ is generated by homogeneous polynomials of degrees $\leq r'-r+2$.
\item[\rm{(b)}] $\reg(X) \leq \max\{\reg(X'), r'-r+2\}$.
\end{itemize}
\end{theorem}

\begin{proof}
Let $I_{X'} \subset S' := K[x_0,\ldots,x_{r'}]$ be the homogeneous vanishing ideal of $X' \subset
\mathbb{P}^{r'} = \Proj(S')$ and let $A':= S'/I_{X'}$ be the homogeneous coordinate ring of $X'$. We consider
$A'$ as a finitely generated graded $S$-module and set $t:= r'-r$. As $X'$ satisfies the condition $N_{2,p}$ with $p \geq\max\{2,t+1\}$,
it follows by \cite[Theorem 3.6]{AK}, that the minimal free presentation of $A'$ has the shape
$$S^s(-2) \stackrel{v}{\rightarrow} S \oplus S^t(-1) \stackrel{q}{\rightarrow} A' \rightarrow 0$$
for some $s \in \mathbb{N}$. Moreover, the coordinate ring $A = S/I_X$ of $X$ is nothing else than the image $q(S)$ under $q$ of the direct summand
$S \subset S\oplus S^t(-1)$. Therefore
$$A'/A \cong \Coker\big(u:S^s(-2) \rightarrow S^t(-1)\big),$$
where $u$ is the composition of the map $v:S^s(-2) \rightarrow
S\oplus S^t(-1)$ with the canonical projection map $w: S\oplus
S^t(-1) \twoheadrightarrow S^t(-1)$. Hence, the $S$-module
$(A'/A)(1)$ is generated by $t$ homogeneous elements of degree $0$
and related in degree $1$. As $\Sing(\pi_{\Lambda})$ is finite, we
have $\dim(A'/A) \leq 1$. So, it follows by \cite[Corollary
2.4]{ChFN} that $\reg\big((A'/A)(1)\big) \leq t-1$, whence
$\reg(A'/A) \leq t$. Now, the short exact sequence $0\rightarrow A
\rightarrow A' \rightarrow A'/A \rightarrow 0$ implies that $\reg(A)
\leq \max\{\reg(A'), t+1\}$.  It follows that $\reg(X) = \reg(A) + 1
\leq \max\{\reg(A') + 1, t+2\} = \max\{\reg(X'), t+2\} =
\max\{\reg(X'), r'-r+2\}$. This proves claim (b).

To prove claim (a), observe that $I_X = \Ker(q) \cap S$ occurs in
the short exact sequence of graded $S$-modules $0 \rightarrow I_X
\rightarrow \rm{Im}(v) \stackrel{w\upharpoonright}{\rightarrow}
\rm{Im}(u) \rightarrow 0$, where $w\upharpoonright$ is the
restriction of the above projection map $w$. In particular, we may
identify $w\upharpoonright$ with the canonical map $S^s(-2)/\Ker(v)
\twoheadrightarrow S^s(-2)/\Ker(u)$. It follows, that $I_X \cong
\Ker(u)/\Ker(v)$. In view of the exact sequence $0\rightarrow
\Ker(u) \rightarrow S^s(-2) \stackrel{u}{\rightarrow}S^t(-1)
\rightarrow A'/A \rightarrow 0$, we now finally get $\reg(\Ker(u))
\leq t+2 = r'-r+2$. Therefore $\Ker(u)$ is generated in degrees
$\leq r'-r+2$, and hence so is $I_X$. This proves statement (a).
\end{proof}

The reader may have noticed, that our proof is inspired by arguments originally found in the papers \cite{GruLPe} and \cite{No}. Nevertheless, the use of
\cite[Corollary 2.4]{ChFN} allows us to argue in a more direct way.

\begin{corollary}\label{2.6' Corollary}
(Compare \cite[Theorem 5.2]{KwP}) Assume that the
projective variety $X \subset \mathbb{P}^r$ is sectionally rational
and that its non-normal locus $X \setminus \Nor(X)$ is finite.
Suppose that either $\Char K = 0$ or $X$ is a surface. Then, the
regularity of $X$ satisfies the Eisenbud-Goto inequality:
$$\reg(X) \leq \deg(X) - \codim_{\mathbb{P}^r}(X) + 1.$$
\end{corollary}

\begin{proof}
We may assume that $X \subsetneq \mathbb{P}^r$. Set $d := \deg(X)$ and $\dim(X) :=t$. According to Theorem~\ref{2.3' Theorem}
there is a variety $\widetilde{X} \subset \mathbb{P}^{d+t-1}$ of minimal degree and a subspace $\Lambda = \mathbb{P}^{d+t-r-2}$ such that
$\widetilde{X} \cap \Lambda = \emptyset$, $X = \widetilde{X}_{\Lambda}$ and $\Sing(\pi_{\Lambda}: \widetilde{X}\twoheadrightarrow X) = X \setminus \Nor(X)$.
In particular $\pi_{\Lambda}:\widetilde{X} \twoheadrightarrow X$ is almost non-singular. Moreover $\widetilde{X} \subset \mathbb{P}^{d+t-1}$ satisfies
the conditions $N_{2,p}$ for all $p \in \mathbb{N}$, so that $\reg(\widetilde{X}) = 2$. Therefore, by Theorem~\ref{2.5' Theorem} we get
$\reg(X) \leq (d+t-1)-r+2 =d-(r-t)+1$ and this proves our claim.
\end{proof}

Motivated by our aim to study surfaces of maximal sectional regularity, we now shall focus on sectionally rational surfaces $X \subset \mathbb{P}^r$ with
$5 \leq r < \deg(X)$ which have finite non-normal locus. One easy consequence of Theorem~\ref{2.3' Theorem} is, that these surfaces are almost non-singular
projections of (possibly singular) normal surface scrolls. Our proof of Theorem~\ref{2.3' Theorem} relies on Fujita's classification of polarized varieties \cite{Fu}
of $\Delta$-genus $0$, which on its turn relies on a powerful result of Ekedahl \cite{Ek}. We therefore take it for justified, to furnish in Theorem~
\ref{2.8' Proposition} a second and more direct argument, showing that sectionally rational surfaces with finite non-normal locus are projections of surface scrolls.
Instead of Fujita's geometric arguments, we use a purely cohomological idea, which on its turn has the disadvantage to apply only in the special situation we shall
be looking at. On the other hand, this specific approach gives slightly more insight and shows at once, that the normalization and the finite Macaulayfication
of our surfaces coincide. Again, we begin with some preparations.

\begin{notation and reminder}\label{2.7' Notation and Reminder} (A) (see \cite{BS6}) Assume now, that $X\subset \mathbb{P}^r = \Proj(S)$ is a non-degenerate
irreducible projective surface of degree $d$, homogeneous vanishing ideal $I \subset S$ and homogeneous coordinate ring $A = S/I$. Let ${\mathfrak a} \subseteq A_+
= S_+A$ be the graded radical ideal which defines the non-Cohen-Macaulay locus $X \setminus \operatorname{CM}(X)$ of $X$. Observe that height ${\mathfrak a} \geq 2,$
so that the ideal transform
\[
B(A):= D_{\mathfrak a}(A) = \varinjlim \Hom_A({\mathfrak a}^n, A) = \bigcup_{n\in \mathbb{N}}(A:_{\rm{Quot}(A)}\mathfrak{a}^n) =
\bigoplus _{n \in {\mathbb Z}} \Gamma (\operatorname{CM}(X), {\mathcal O}_X(n))
\]
of $A$ with respect to ${\mathfrak a}$ is a positively graded finite birational integral extension domain of $A$. In particular $B(A)_0 = K$. Moreover $B(A)$
has the second Serre-property $S_2$. As $\mbox{Proj}(B(A))$ is of dimension $2$, it thus is a locally Cohen-Macaulay scheme.\\
If $E$ is a finite graded integral extension domain of $A$ which
satisfies the property $S_2$, we have $A \subset B(A) \subset E$. So
$B(A)$ is the least finite graded integral extension domain which
has the property $S_2$. Therefore, we call $B(A)$ the $S_2${\it
-cover of} $A$. We also can describe $B(A)$ as the endomorphism ring
$\mbox{End}(K(A), K(A))$ of the canonical module $K(A) = K^3(A) =
\mbox{Ext}^{r-2}_S(A, S(-r-1))$ of $A$.

\noindent (B) Let the notations be as in part (A). Then, the
inclusion map $A \rightarrow B(A)$ gives rise to a finite morphism
$$\pi: \widetilde{X}:= \Proj(B(A)) \twoheadrightarrow X, \mbox{ with } \Sing(\pi) = X \setminus \CM(X).$$
In particular $\pi$ is almost non-singular and hence birational.
Moreover, for any finite morphism $\rho: Y \twoheadrightarrow X$
such that $Y$ is locally Cohen-Macaulay, there is a unique morphism
$\sigma:Y \rightarrow \widetilde{X}$ such that $\rho = \pi
\circ\sigma$. In addition $\sigma$ is an isomorphism if and only if
$\Sing(\rho) = X \setminus \CM(X)$. Therefore, the morphism
$\pi:\widetilde{X} \twoheadrightarrow X$ is addressed as the
\textit{finite Macaulayfication} of $X$. Keep in mind, that --
unlike to what happens with normalization -- there may be proper
birational morphisms $\tau:Z \twoheadrightarrow X$ with $Z$ locally
Cohen-Macaulay, which do not factor through $\pi$ (see \cite{B0}).

\noindent (C) Let $X \subset \mathbb{P}^r$ be as in part (A).  We
introduce the invariants
\[
\e_x(X) := \length(H^1_{\mathfrak{m}_{X,x}}(\mathcal{O}_{X,x})),
(x \in X \mbox{ closed }) \; \mbox{ and }
\e(X):= {\sum }_{x\in X, {\mbox{closed}}} \e_x(X).
\]
Note that the latter counts the {\it number of non-Cohen-Macaulay points} of $X$
in a weighted way. Keep in mind that
\[
\e(X) = h^1 (X, \mathcal{O}_X(n)) \mbox{ for all } n \ll 0 .
\]

\noindent (D) Let $X \subset \mathbb{P}^r$ and $A = S/I$ be as
above. We denote the \textit{arithmetic depth} of $X$ by
$\depth(X)$, hence $\depth(X) := \depth(S/I)$.
\end{notation and reminder}

Now, we are ready to formulate and to prove the conclusive result of this section. As previously announced and justified, we shall offer two different proves of
statement (c) of the theorem to follow.

\begin{theorem} \label{2.8' Proposition}
Let $d \in \mathbb{N}$ with $5 \leq r < d$. Assume that the non-degenerate irreducible surface $X \subset \mathbb{P}^r_K$  of degree $d$ is sectionally rational
and has finite non-normal locus $X \setminus \Nor(X)$. Then
\begin{itemize}
\item[\rm{(a)}] $\reg(X) \leq d-r+3$ and $\Reg(X) = \Nor(X) = \CM(X)$.
\item[\rm{(b)}] The cohomology of $X$ satisfies the following conditions
                \begin{itemize}
                \item[\rm{(1)}] For all $n \leq 0 \quad$ it holds $h^2(\mathbb{P}^r,\mathcal{J}_X(n)) = \e(X)$.
                \item[\rm{(2)}] For all $n \geq 0 \quad$ it holds $h^2(\mathbb{P}^r,\mathcal{J}_X(n+1)) \leq h^2(\mathbb{P}^r,\mathcal{J}_X(n))$.
                \item[\rm{(3)}] For all $n \geq -1$ it holds $h^3(\mathbb{P}^r,\mathcal{J}_X(n)) = 0.$
                \end{itemize}
\item[\rm{(c)}] The $S_2$-cover $B = B(A)$ of $A = S/I$ is the homogeneous coordinate ring of a surface scroll $\widetilde{X} \subset \mathbb{P}^{d+1}_K$ of degree $d$
                and there is a subspace $\Lambda = \mathbb{P}^{d-r} \subset \mathbb{P}^{d+1}$ such that $\widetilde{X} \cap \Lambda = \emptyset$ and
                \begin{itemize}
                \item[\rm{(1)}] $X = \widetilde{X}_{\Lambda}$,
                \item[\rm{(2)}] $\Sing(\pi_{\Lambda}:\widetilde{X}\twoheadrightarrow X) = X \setminus \CM(X)$.
                \end{itemize}
\item[\rm{(d)}] There is a unique non-negative integer $a \leq d/2$ such that the scroll $\widetilde{X} \subset \mathbb{P}^{d+1}$ of statement (c) is
                projectively equivalent to the standard rational surface scroll $S(a,d-a) \subset \mathbb{P}^{d+1}$, and moreover
                $a > 0$ if and only if $X$ is not a cone.
\item[\rm{(e)}] The morphism $\pi_{\Lambda}:\widetilde{X} \twoheadrightarrow X$ of statement (c) provides the normalization as well as the finite Macaulayfication of $X$.
\item[\rm{(f)}] If $\pi_{\Lambda}:\widetilde{X} \twoheadrightarrow X$ is as in statement (c) and $p \in X \setminus \Reg(X)$, we have
        \begin{itemize}
        \item[\rm{(1)}] $1< \#(\pi_{\Lambda})^{-1}(p) \leq \mult_p(X)$, with equality at the second place $X$ is not a cone with vertex $p$ and if
                $\mult_{\widetilde{p}}((\pi_{\Lambda})^{-1}(p)) \leq 2$ for all $\widetilde{p} \in (\pi_{\Lambda})^{-1}(p)$.
        \item[\rm{(2)}] $\#(\pi_{\Lambda})^{-1}(p) \leq \e_p(X) + 1$ with equality if and only if $p$ is a Buchsbaum point of the surface $X$.
        \end{itemize}
\item[\rm{(g)}] If $\mathbb{M} = \mathbb{P}^s \subset \mathbb{P}^r$ is a linear subspace whose intersection with $X$ is finite, then
                $$\#({\rm Reg}(X) \cap \mathbb{M}) + 2\#\big((X \cap \mathbb{M})_{\rm red}\setminus {\rm Reg}(X)\big)  \leq d-r+s+2.$$
\end{itemize}
\end{theorem}

\begin{proof} (b): Let $h \in S_1 \setminus\{0\}$ be a general linear form. As $X$ has at most finitely many singular points, the curve $C_h := \Proj(A/hA)$
is smooth and rational. Therefore $H^0(C_h,\mathcal{O}_{C_h}) = H^0(X,\mathcal{O}_X) = K$ and $H^1(C_h,\mathcal{O}_{C_h}(n)) = 0$ for all $n \geq 0$.
Hence, the exact sequences
$$ 0 \rightarrow H^0(X,\mathcal{O}_X(n)) \rightarrow H^0(C_h,\mathcal{O}_{C_h}(n))
\rightarrow H^1(X,\mathcal{O}_X(n-1)) \rightarrow H^1(X,\mathcal{O}_X(n))$$
$$\rightarrow H^1(C_h,\mathcal{O}_{C_h}(n))\rightarrow H^2(X,\mathcal{O}_X)(n-1)) \rightarrow H^2(X,\mathcal{O}_X(n)) \rightarrow 0$$
show that
$$h^1(X,\mathcal{O}_X(-1)) = h^1(X,\mathcal{O}_X), \quad h^1(X,\mathcal{O}_X(n+1)) \leq h^1(X,\mathcal{O}_X(n)) \mbox{ for all } n \geq 0$$
and
$$h^2(X,\mathcal{O}_X(n)) = 0 \mbox{ for all } n \geq -1.$$
Moreover, as the non-normal locus $X\setminus \Nor(X)$ of $X$ is finite, it follows from
\cite[Proposition 5.2]{AlB} that
\[
\e(X) \leq h^1(X,\mathcal{O}_X(n-1)) \leq \max\{h^1(X,\mathcal{O}_X(n)) - r, \e(X)\},
\mbox{ for all } n\leq 0.
\]
Therefore we obtain $\e(X) =  h^1(X,\mathcal{O}_X(n))$ for all $n
\leq 0$. As $h^i(\mathbb{P}^r,\mathcal{J}_X(n)) =
h^{i-1}(X,\mathcal{O}_X(n))$ for all $n \in \mathbb{Z}$ and $i =
2,3$, we get our claims (1), (2) and (3).
\smallskip

\noindent (c):\textit{First Proof}: According to Theorem~\ref{2.3'
Theorem} there is a surface $\widetilde{X} \subset \mathbb{P}^{d+1}$
of minimal degree and a subspace $\Lambda = \mathbb{P}^{d-r} \subset
\mathbb{P}^{d+1}$, such that
$$\widetilde{X} \cap \Lambda = \emptyset,\quad X = \widetilde{X}_{\Lambda} \mbox{ and } \Sing(\pi_{\Lambda}: \widetilde{X} \twoheadrightarrow X) = X \setminus \Nor(X).$$
Observe, that $\widetilde{X} \subset \mathbb{P}^{d+1}$ is either a
(possible singular) surface scroll or the Veronese surface in
$\mathbb{P}^5$. As $d \geq 6$, the latter cannot occur, and so
$\widetilde{X} \subset \mathbb{P}^{d+1}$ is a surface scroll. Now,
let $E$ be the homogeneous coordinate ring of $\widetilde{X} \subset
\mathbb{P}^{d+1}$. As $E$ is a Cohen-Macaulay ring, we have
canonical inclusions of graded rings $A \subset B(A) \subset E$ (see
Notation and Reminder~\ref{2.7' Notation and Reminder} (A)). As $X$
is a surface, we have $X \setminus \CM(X) \subset X \setminus
\Nor(X) = \Sing(\pi_{\Lambda}) = |\Proj(A/\mathfrak{a})|$ and hence
$E \subset \bigcup_{n\in
\mathbb{N}}(A:_{\rm{Quot}(A)}\mathfrak{a}^n) = B(A)$. Therefore $E =
B(A) = B$ and statement (c) is shown.

\textit{Alternative Proof}: Our first aim is to show that
$\dim_K(B_1) = d+2$. As $h \in S_1 \setminus \{0\}$ is general, we
have $\Rad{(\mathfrak{a},h)} = A_{+}$. So, comparing local
cohomology furnishes a short exact sequence of graded local
cohomology modules (see \cite[(8.1.2)]{BSh})
\[
0\to H^1_{(h)}(H^1_{\mathfrak{a}}(A)) \to H^2_{A_+}(A)
\to H^0_{(h)}(H^2_{\mathfrak{a}}(A)) \to 0.
\]
Observe that with $D := D_{A_+}(A) = \varinjlim \Hom_A((A_+)^n,A) = \bigoplus_{n\in \mathbb{Z}} \Gamma(X,\mathcal{O}_X(n))$, the kernel of the natural map
$B/A \twoheadrightarrow B/D$ is $S_+$-torsion. As $\Rad(\mathfrak{a}, h) = A_+$, it follows
$$H^1_{(h)}(H^1_{\mathfrak{a}}(A)) = H^1_{A_+}(H^1_{\mathfrak{a}}(A)) = H^1_{A_+}(B/A) = H^1_{A_+}(B/D).$$
By Lemma 2.4 of \cite{BS6} we have $\dim_K ((B/D)_n) = \e(X)$ for all $n \gg 0.$ Consequently
$\dim_K (D_{A_+}(B/D)_n) = \e(X)$ for all $n \in \mathbb{Z}$. As $(B/D)_0 = 0$ it follows that
$$\dim_K(H^1_{(h)}(H^1_{\mathfrak{a}}(A))_0) + \dim_K(H^1_{A_+}(B/D)_0) = \e(X).$$
By statement (b) we also have $\dim_K(H^2_{A_+}(A)_0) = H^1(X,\mathcal{O}_X) = \e(X)$. So, the above
sequence shows that $H^0_{(h)}(H^2_{\mathfrak{a}}(A))_0 = 0$. Therefore the multiplication map $h: H^2_{\mathfrak{a}}(A)_0 \rightarrow H^2_{\mathfrak{a}}(A)_1$
is injective. Now, applying the functor $D_{\mathfrak{a}}(\bullet)$ to the exact sequence $0\rightarrow A(-1) \stackrel{h}{\rightarrow} A \rightarrow A/hA \rightarrow 0$
and observing once more that $\Rad{(\mathfrak{a},h)} = A_{+}$, we get the exact sequence of $K$-vector spaces
\[
0 \to D_{\mathfrak{a}}(A)_0 \to D_{\mathfrak{a}}(A)_1 \to D_{A_+}(A/hA)_1
\to H^2_{\mathfrak{a}}(A)_0 \stackrel{h}{\rightarrow} H^2_{\mathfrak{a}}(A)_1.
\]
As $D_{\mathfrak{a}}(A)_0 = B_0 = K$ and as, in addition, the last
map in this sequence is injective, we end up with $\dim_K(B_1) =
\dim_K(D_{\mathfrak{a}}(A)_1) = \dim_K(D_{A_+}(A/hA)_1) + 1$. As
$C_h \subset \Proj(S/hS) = \mathbb{P}^{r-1}$ is a non-degenerate
smooth rational curve of degree $d$, the $K$-vector space
$D_{A_+}(A/hA)_1 \cong H^0(C_h, \mathcal{O}_{C_h}(1))$ has dimension
$d+1$, so that indeed $\dim_K(B_1) = d+2$.

Now, consider the non-degenerate closed subscheme  $\widetilde{X} :=
\Proj(K[B_1]) \subset \mathbb{P}^{d+1}.$ As $K[B_1]$ is a finite
birational integral extension domain of $A$, the scheme
$\widetilde{X} \subset \mathbb{P}^{d+1}$ is a non-degenerate
irreducible and reduced surface of degree $d$. It follows in
particular that $\widetilde{X} \subset \mathbb{P}^{d+1}$ is a
surface of minimal degree and hence (as $d \geq 6$) a (possibly
singular) surface scroll. In particular $K[B_1]$ is a Cohen-Macaulay
ring which contains $A$ and is contained in the $S_2$-cover $B(A)$
of $A$. Thus $K[B_1] = B(A)$ (see Notation and Reminder~ \ref{2.7'
Notation and Reminder} (A)) and hence $B = B(A)$ is the homogeneous
coordinate ring of the surface scroll $\widetilde{X} \subset
\mathbb{P}^{d+1}$.

Moreover, the inclusion map $A\rightarrow B(A)$ gives rise to a
finite morphism $\pi_{\Lambda}:\widetilde{X} \twoheadrightarrow X$,
induced by a linear projection $\pi'_{\Lambda}: \mathbb{P}^{d+1}
\setminus \Lambda \twoheadrightarrow \mathbb{P}^r$ from a subspace
$\Lambda = \mathbb{P}^{d-r} \subset \mathbb{P}^{d+1}$ disjoint to
$\widetilde{X} \subset \mathbb{P}^{d+1}$, so that indeed $X =
\widetilde{X}_{\Lambda}$. Finally, by Notation and
Reminder~\ref{2.7' Notation and Reminder} (B), we have
$\Sing(\pi_{\Lambda}:\widetilde{X}\twoheadrightarrow X) = X
\setminus \CM(X)$. So, statement (c) is shown.
\smallskip

\noindent (d): Clearly, $\widetilde{X}$ is projectively equivalent
to a rational surface scroll $S(a,d-a) \subset \mathbb{P}^{d+1}$ for
some non-negative integer $a \leq d/2$. The uniqueness of $a$
follows for example by \cite[Lemma 5.5]{BS6}. The remaining claim is
clear from the well known fact that $\widetilde{X}_{\Lambda}$ is a
cone if and only if $\widetilde{X}$ is -- and hence if and only if
$a = 0$.
\smallskip

\noindent (a): By statement (c) and Corollary~\ref{2.6' Corollary},
we see that $\reg(X) \leq d-r+3$. As $X$ is a surface, we have
$\Reg(X) \subseteq \Nor(X)\subseteq \CM(X)$. It thus remains to show
that $\CM(X) \subseteq \Reg(X)$. To do so, let $x\in \CM(X)$ be a
closed point. Then, by statement (d) we have $x \notin
\Sing(\pi_{\Lambda})$. Hence $x$ has a unique preimage
$\widetilde{x} \in \widetilde{X}$ and moreover
$\mathcal{O}_{\widetilde{X},\widetilde{x}} \cong \mathcal{O}_{X,x}$.
Assume that $\widetilde{x} \notin \Reg(\widetilde{X})$. Then
$\widetilde{x}$ is a vertex point of $\widetilde{X}$ and hence the
tangent space $\T_{\widetilde{x}}(\widetilde{X})$ of $\widetilde{X}$
at $\widetilde{x}$ has dimension $d+1 > r \geq \dim_K(\T_x(X))$.
This leads to the contradiction that $x \in \Sing(\pi_{\Lambda})$.
Therefore $\widetilde{x} \in \Reg(\widetilde{X})$ and hence $x \in
\Reg(X)$.
\smallskip

\noindent (e): By its construction, $\pi_{\Lambda}: \widetilde{X}
\twoheadrightarrow X$ provides the finite Macaulayfication of $X$
(see Notation and Reminder ~\ref{2.7' Notation and Reminder} (B)).
According to statement (a) we have $\Sing(\pi_{\Lambda}) = X
\setminus \Nor(X) = \Sing(\rho)$, where $\rho:Y \twoheadrightarrow
X$ is the normalization of $X$. As $Y$ is a normal surface, it is
locally Cohen-Macaulay. Therefore $\pi_{\Lambda}$ provides also the
normalization of $X$ (see Notation and Reminder~\ref{2.7' Notation
and Reminder} (B)).
\smallskip

\noindent (f): As $p$ is an isolated point of $X \setminus \Reg(X) =
\Sing(\pi_{\Lambda})$, it follows that the
$\mathcal{O}_{X,p}$-module
$((\pi_{\Lambda})_{*}\mathcal{O}_{\widetilde{X}})_p/\mathcal{O}_{X,p}$
is of finite length. As $\widetilde{X}$ is locally Cohen-Macaulay,
the finitely generated $\mathcal{O}_{X,p}$-module
$((\pi_{\Lambda})_{*}\mathcal{O}_{\widetilde{X}})_p$ is
Cohen-Macaulay. From this, we get
$$\mult_p(X) = \mult_{\mathfrak{m}_{X,p}}\big(((\pi_{\Lambda})_{*}\mathcal{O}_{\widetilde{X}})_p\big) \geq \length\big(((\pi_{\Lambda})_{*}\mathcal{O}_{\widetilde{X}})_p
/\mathfrak{m}_{X,p}((\pi_{\Lambda})_{*}\mathcal{O}_{\widetilde{X}})_p\big) = \#(\pi^{-1}_{\Lambda}(p))$$
and also $((\pi_{\Lambda})_{*}\mathcal{O}_{\widetilde{X}})_p/\mathcal{O}_{X,p} \cong H^1_{\mathfrak{m}_{X,p}}(\mathcal{O}_{X,p})$.
As $\big((\pi_{\Lambda})_{*}\mathcal{O}_{\widetilde{X}}\big)_p$ is a proper finite integral extension domain of $\mathcal{O}_{X,p}$,the Nakayma Lemma yields that
$\length\big(((\pi_{\Lambda})_{*}\mathcal{O}_{\widetilde{X}})_p /\mathfrak{m}_{X,p}((\pi_{\Lambda})_{*}\mathcal{O}_{\widetilde{X}})_p\big) > 1$.
Therefore we have
$$1 < \#((\pi_{\Lambda})^{-1}(p)) = \length\big(((\pi_{\Lambda})_{*}\mathcal{O}_{\widetilde{X}})_p/\mathfrak{m}_{X,p}((\pi_{\Lambda})_{*}\mathcal{O}_{\widetilde{X}})_p\big)
\leq \length\big(((\pi_{\Lambda})_{*}\mathcal{O}_{\widetilde{X}})_p/\mathfrak{m}_{X,p}\big) =$$
$$= \length\big(((\pi_{\Lambda})_{*}\mathcal{O}_{\widetilde{X}})_p/\mathcal{O}_{X,p}\big) + 1 = \e_p(X) + 1.$$
This proves the inequalities in claims (1) and (2).

If $p$ is not a vertex point of $X$, we have
$(\pi_{\Lambda})^{-1}(p) \subset \Reg(\widetilde{X})$ and the
equality in claim (1) follows easily (see \cite[Lemma 3.2]{BP2}).
Observe that $p$ is a Buchsbaum point of $X$ if and only if
$\mathfrak{m}_{X,p} H^1_{\mathfrak{m}_{X,p}}(\mathcal{O}_{X,p}) =
0$, hence if and only if
$\mathfrak{m}_{X,p}((\pi_{\Lambda})_{*}\mathcal{O}_{\widetilde{X}})_p
= \mathfrak{m}_{X,p}$. This proves the equivalence of claim (2).
\smallskip

\noindent (g): Let $\Lambda = \mathbb{P}^{d-r} \subset
\mathbb{P}^{d+1}$ be as in (c) and let
$\pi'_{\Lambda}:\mathbb{P}^{d+1} \setminus \Lambda
\twoheadrightarrow \mathbb{P}^r$ be a projection centered at
$\Lambda$ such that $\pi_{\Lambda}$ coincides with the restriction
$\pi'_{\Lambda} \upharpoonright_{\widetilde{X}}$ of $\pi'_{\Lambda}$
to $\widetilde{X}$. By our hypotheses
$$\#\big(\pi^{-1}_{\Lambda}(X \cap \mathbb{M})\big) = \#\big(\widetilde{X} \cap \overline{(\pi'_{\Lambda})^{-1}(\mathbb{M})}\big) < \infty.$$
As $\widetilde{X} \subset \mathbb{P}^{d+1}$ is a rational normal scroll and $(\pi'_{\Lambda})^{-1}(\mathbb{M}) = \mathbb{P}^{s+d-r+1} \subset \mathbb{P}^{d+1}$ we have
the inequality $\#\big(\widetilde{X} \cap \overline{(\pi'_{\Lambda})^{-1}(\mathbb{M})}\big) \leq d-r+s+2$, so that
$$\#\big((\pi_{\Lambda})^{-1}(X \cap \mathbb{M})\big) \leq d-r+s+2.$$
As $\Sing(\pi_{\Lambda}) = X \setminus {\rm Reg}(X)$ we have an isomorphism $(\pi_{\Lambda})^{-1}({\rm Reg}(X) \cap \mathbb{M}) \cong {\rm Reg}(X) \cap \mathbb{M}$,
so that
$$\#\big((\pi_{\Lambda})^{-1}({\rm Reg}(X) \cap \mathbb{M})\big) = \#\big({\rm Reg}(X) \cap \mathbb{M}\big).$$
By claim (1) of statement (f) we have $\#((\pi_{\Lambda})^{-1}(p) \geq 2$ for all $p \in (X \cap \mathbb{M}) \setminus {\rm Reg(X)}$, so that
$$2 \#\big((X\cap \mathbb{M})_{\rm red} \setminus {\rm Reg}(X)\big) \leq \#\big((\pi_{\Lambda})^{-1}(X \cap \mathbb{M})\setminus {\rm Reg}(X)\big).$$
As the fiber $(\pi_{\Lambda})^{-1}(X \cap \mathbb{M})$ is the disjoint union of its two subschemes $(\pi_{\Lambda})^{-1}({\rm Reg}(X) \cap \mathbb{M})$ and
$(\pi_{\Lambda})^{-1}((X \cap \mathbb{M}) \setminus {\rm Reg}(X))$, our claim follows.
\end{proof}


\section{Surfaces of Maximal Sectional Regularity}
\label{3'. Surfaces of Maximal Sectional Regularity}

\noindent In section 2 we already have introduced varieties of
maximal sectional regularity. In this section, we will focus to the
special case of surfaces of maximal sectional regularity. We first
recall a few basic facts around the notion of surface of maximal
sectional regularity. Then, we give a characterization of these
surfaces in terms of what we call their sectional regularity.
Finally, we prove a structure theorem which gives a number of
fundamental properties of such surfaces and which mainly follows
from Theorem~\ref{2.8' Proposition}.

\begin{notation and reminder}\label{3.1' Notation and Reminder}
(A) Let $ X\subset \mathbb{P}^r = \Proj(S)$ is a non-degenerate irreducible projective surface
of degree $d$, with homogeneous vanishing ideal $I \subset S$ and homogeneous coordinate ring $A = S/I$. For each linear form $h \in S_1 \setminus \{0\}$
we write
$$\mathbb{H}_h := \Proj(S/hS)$$
for the hyperplane in $\mathbb{P}^r$ defined by $h$ and
$$C_h := \Proj(A/hA) = X\cap \mathbb{H}_h$$
for the corresponding hyperplane section of $X$.\\
(B) (see \cite{BS5}) We keep the above hypotheses and notations. The \textit{sectional regularity} $\sreg(X)$ of the surface $X$ is defined
as the least regularity of a hyperplane section of $X$:
\[
\sreg(X) := \min\{\reg (C_h)  \mid h\in S_1 \setminus \{0\}\}.
\]
For each $h \in S_1 \setminus\{0\}$ we have $\reg(C_h) \leq \reg(X)$ so that in particular
$$\sreg(X) \leq \reg(X).$$
(C) Let the notations be as in part (A). Let $i \in \mathbb{N}$, $n \in \mathbb{Z}$ and $h \in S_1\setminus\{0\}$. Then the exact sequence of $K$-vector spaces
$$H^i(\mathbb{P}^r,\mathcal{J}_X(n-1))\stackrel{h}{\rightarrow} H^i(\mathbb{P}^r,\mathcal{J}_X(n))\rightarrow H^i(\mathbb{H}_h,\mathcal{J}_{C_h}(n))
\rightarrow $$
$$H^{i+1}(\mathbb{P}^r, \mathcal{J}_X(n-1))\stackrel{h}{\rightarrow} H^{i+1}(\mathbb{P}^r,\mathcal{J}_X(n))$$
shows, that the set
$$\mathbb{W}^i_n(X) := \{h\in S_1 \setminus \{0\} \mid H^i(\mathbb{H}_h, \mathcal{J}_{C_h}(n)) = 0\}$$
is open in $S_1$ for all $n \in \mathbb{Z}$ and is equal to $S_1\setminus\{0\}$ if $n \geq \reg(X)$. So, for each $s \in \mathbb{Z}$ the set
$\{h\in S_1 \setminus\{0\}\mid \reg(C_h) \leq s \} = \bigcap^{r+1}_{i=1}\bigcap_{n \geq s}\mathbb{W}^i_{n-i}(X)$ is open in $S_1$. Applying this with $s := \sreg X$,
we see that the set
$$\mathbb{W}(X) := \{h\in S_1\setminus \{0\} \mid \reg(C_h) = \sreg(X)\}$$
is open and dense in $S_1$ and keeping in mind the first Bertini Theorem, we obtain that the set
$$\mathbb{U}(X) := \{h \in \mathbb{W}(X) \mid C_h \mbox{ is integral } \}$$
is open and dense in $S_1$. Observe that in the notations of Definition and Remark~\ref{4.1''' Definition and Remark} (A) we have
$$\mathbb{P}\mathbb{U}(X) = \{\mathbb{H}_h \mid h \in \mathbb{U}(X)\}.$$
(D) Keep the previous notations and hypotheses. Then, for each $h \in \mathbb{U}(X)$, the regularity bound for curves due to Gruson-Lazarsfeld-Peskine
\cite{GruLPe} yields that $\reg(C_h) \leq d-r+3$, so that
$$\sreg(X) \leq d-r+3.$$
\end{notation and reminder}
\smallskip

In \cite{BS5} we did define surfaces of maximal sectional regularity as those, whose sectional regularity takes the maximal possible value. That this definition coincides
with our definition given in Remark and Definition~\ref{2.3'' Remark and Definition} (C) is the subject of the following result, in which we use the notations introduced
in Definition and Remarks~\ref{4.1''' Definition and Remark} and ~\ref{2.1'' Definition and Remark}.

\begin{proposition}\label{3.1'' Proposition}
Let $5 \leq r < d$ and let $X \subset \mathbb{P}^r$ be a non-degenerate irreducible projective surface of degree $d$.
\begin{itemize}
\item[\rm{(a)}] The following statements are equivalent
                \begin{itemize}
                \item[\rm{(i)}]    The surface $X$ is of maximal sectional regularity;
                \item[\rm{(ii)}]   $\sreg(X) = d-r+3$;
                \item[\rm{(iii)}]  ${}^*\mathfrak{d}(X) = 2$;
                \item[\rm{(iv)}]   $\mbox{ }\mathfrak{d}(X) = 2$.
                \end{itemize}
\item[\rm{(b)}] If the equivalent conditions (i) -- (iv) of statement (a) are satisfied, then we have
                \begin{itemize}
                \item[\rm{(1)}]   For each $h \in \mathbb{U}(X)$, the curve $C_h \subset \mathbb{H}_h$ is of maximal regularity $d-r+3$ and hence smooth, rational
                                  and with a unique extremal secant line $\mathbb{L}_h := \mathbb{L}_{\mathbb{H}_h,X}$.
                \item[\rm{(2)}]   ${}^*\Sigma^{\circ}_{d-r+3}(X) = \{\mathbb{L}_h \mid h \in \mathbb{U}(X)\}$.
                \end{itemize}
\end{itemize}
\end{proposition}

\begin{proof}
(a): (i) $\Leftrightarrow$ (ii): According to Notation and Reminder~\ref{3.1' Notation and Reminder} (C), the set
$$\mathbb{U}(X) = \{h\in S_1 \setminus \{0\} \mid C_h \subset \mathbb{H}_h \mbox{ is an integral curve with} \reg(C_h) = \sreg(X) \}$$
is dense and open in $S_1$. According to Remark and
Definition~\ref{2.3'' Remark and Definition} (C) the surface $X
\subset \mathbb{P}^r$ is of maximal sectional regularity if and only
if there is a dense open subset $\mathcal{U}$ of $S_1\setminus\{0\}$
such that $C_h \subset \mathbb{H}_h$ is an integral curve with
$\reg(C_h) = d-r+3$. This gives the requested equivalence.

The equivalences (i)$\Leftrightarrow$(iii)$\Leftrightarrow$(iv) hold
by Theorem~\ref{2.3'' Theorem}.
\smallskip

\noindent (b): Assume that the equivalent conditions (i) -- (iv) of
statement (a) hold. Claims (1) and (2) follow from the fact that
$\mathbb{P}\mathbb{U}(X) = \{\mathbb{H}_h \mid h \in
\mathbb{U}(X)\}$ (see Notation and Reminder~\ref{3.1' Notation and
Reminder} (C)) and Definition and Remark~\ref{4.1''' Definition and
Remark} (C).
\end{proof}

Now we are able to formulate and to prove our first main result on the structure of surfaces of maximal sectional regularity.

\begin{theorem}
\label{3.2' Theorem} Let $5\leq r < d$ and assume that the non-degenerate irreducible surface $X \subset \mathbb{P}^r$ of degree $d$ is of maximal sectional regularity.
Then
\begin{itemize}
\item[\rm{(a)}] The surface $X$ is sectionally rational and $X \setminus \Nor(X)$ is finite. In particular
               \begin{itemize}
               \item[\rm{(1)}] $\reg(X) = d-r+3$ and $\Reg(X) = \Nor(X) = \CM(X)$.
               \item[\rm{(2)}] For all $n \leq 0 \quad$ it holds $h^2(\mathbb{P}^r,\mathcal{J}_X(n)) = \e(X)$,\\
                               for all $n \geq 0 \quad$ it holds $h^2(\mathbb{P}^r,\mathcal{J}_X(n+1)) \leq h^2(\mathbb{P}^r,\mathcal{J}_X(n))$, and\\
                               for all $n \geq -1$ it holds $h^3(\mathbb{P}^r,\mathcal{J}_X(n)) = 0$.
               \item[\rm{(3)}] There is a unique non-negative integer $a \leq \frac{d}{2}$ and a subspace $\Lambda = \mathbb{P}^{d-r} \subset \mathbb{P}^{d+1}$
                               disjoint to the rational surface scroll $\widetilde{X} := S(a,d-a) \subset \mathbb{P}^{d+1}$ such that $X = \widetilde{X}_{\Lambda}$
                               and $\Sing(\pi_{\Lambda}:\widetilde{X} \twoheadrightarrow X) = X \setminus \Reg(X)$. Moreover
                               the morphism $\pi_{\Lambda}:\widetilde{X} \twoheadrightarrow X$ provides the normalization and the
                               finite Macaulayfication of $X$ and the number $a$ is positive if and only if $X \subset \mathbb{P}^r$ is a not cone.
               \item[\rm{(4)}] If $\pi_{\Lambda}:\widetilde{X} \twoheadrightarrow X$ is as in claim (3) and $p \in X \setminus \Reg(X)$, we have
                \begin{itemize}
            \item[\rm{(i)}]   $1 < \#(\pi_{\Lambda})^{-1}(p) \leq \mult_p(X)$, with equality at the second place if $X$ is not a cone with vertex $p$ and
                                      if $\mult_{\widetilde{p}}((\pi_{\Lambda})^{-1}(p)) \leq 2$ for all $\widetilde{p} \in (\pi_{\Lambda})^{-1}(p)$.
            \item[\rm{(ii)}]  $\#(\pi_{\Lambda})^{-1}(p) \leq \e_p(X) + 1$ with equality if and only if $p$ is a Buchsbaum point of the surface $X$.
                    \end{itemize}
               \item[\rm{(5)}] If $\mathbb{M} = \mathbb{P}^s \subset \mathbb{P}^r$ is a linear subspace whose intersection with $X$ is finite, then
                                  $$\#({\rm Reg}(X) \cap \mathbb{M}) + \#\big((X \cap \mathbb{M})_{\rm red}\setminus {\rm Reg}(X)\big)  \leq d-r+s+2.$$
               \end{itemize}
\item[\rm{(b)}] $\reg(X) = d-r+3$, $\mathbb{W}(X) = S_1 \setminus \{0\}$ and $\depth(X) \leq 2$.
\item[\rm{(c)}] If $h \in \mathbb{U}(X)$, then
                $$W_h := {\rm Join}(X,\mathbb{L}_h) \subset \mathbb{P}^r$$
                is a rational $4$-fold scroll of type $S(0,0,b_h,r-b_h-3)$ with $0 \leq b_h \leq \frac{r-3}{2}$ and $\mathbb{L}_h = S(0,0)$.
\item[\rm{(d)}] $h^0(\mathbb{P}^r,\mathcal{J}_X(2)) \geq \binom{r-3}{2}$.
\end{itemize}
\end{theorem}

\begin{proof}
(a): According to Proposition~\ref{3.1'' Proposition} (b) (1), the
hyperplane section curve $X \cap \mathbb{H}_h = C_h$ is smooth and
rational for all $h$ in the dense open subset $\mathbb{U}(X)$ of
$S_1$. This shows at once, that $X$ is sectionally rational and has
finite non-normal locus. So, we get statement (a) as well as its
claims (1)--(5) on application of Notation and Reminder~\ref{3.1'
Notation and Reminder} (B) and Theorem~\ref{2.8' Proposition}.
\smallskip

\noindent (b): By Theorem~\ref{2.8' Proposition} (a) and Notation
and Reminder~\ref{3.1' Notation and Reminder} (B) we have $\reg(X) =
d-r+3 = \sreg X$. It follows that $\reg(C_h) \leq \reg(X) = d-r+3$
for all $h \in S_1 \setminus \{0\}$ and hence $\mathbb{W}(X) = S_1
\setminus \{0\}$ (see Notation and Reminder~ \ref{3.1' Notation and
Reminder} (C)). To prove the inequality $\depth(X) \leq 2$, assume
that $\depth(X) > 1$ and let $h \in \mathbb{U}(X)$. Then, the ring
$S/(I,h)$ is the homogeneous coordinate ring of the curve $C_h
\subset \mathbb{H}_h$, which is of maximal regularity and hence of
arithmetic depth $1$. It follows that $\depth(X) = \depth(S/I)  =
2$.
\smallskip

\noindent (c): Let $h \in \mathbb{U}(X)$ and observe that $X \cap
\mathbb{L}_h \subset \Reg(X)$. Let $W_h:= \rm{Join}(X,\mathbb{L}_h)
\subset \mathbb{P}^r$. Then $W_h \subset \mathbb{P}^r$ is a
non-degenerate irreducible projective variety of dimension $4$. Let
$\mathbb{I} := \mathbb{P}^{r-2} \subset \mathbb{P}^r$ be a general
$(r-2)$-plane and let $\varrho_h: X\setminus \mathbb{L}_h
\rightarrow W_h \cap \mathbb{I} \subset \mathbb{I}$ be the finite
dominant morphism obtained by projecting from $\mathbb{L}_h$. A
general $(r-2)$-plane $\mathbb{E} = \mathbb{P}^{r-2} \subset
\mathbb{P}^r$, which contains $\mathbb{L}_h$, satisfies $\#(X \cap
\mathbb{E}) = d$. Therefore it follows that $W_h \cap \mathbb{I}
\subset \mathbb{I}$ is of degree $d-(d-r+3)=r-3=
\codim_{\mathbb{P}^r}(W_h) + 1$. So, $W_h \subset \mathbb{P}^r$ is
of minimal degree and moreover $\mathbb{L}_h \subset \Sing(W_h)$.
Hence, $W_h$ is either projectively equivalent to a scroll
$S(0,0,b_h,r-b_h-3) \subset \mathbb{P}^r$ for some non-negative
integer $b_h \leq \frac{r-3}{2}$ and $\mathbb{L}_h = S(0,0)$, or
else $r=7$ and $W_h$ is a cone with vertex $\mathbb{L}_h$ over a
Veronese surface contained in some subspace $\mathbb{P}^5 \subset
\mathbb{P}^7$. As $X$ is the projected image of a surface scroll
$\widetilde{X} \subset \mathbb{P}^{r+1}$, it contains a
one-dimensional family of lines disjoint to $\mathbb{L}_h$, so that
$W_h$ contains a one-dimensional family of $3$-spaces containing
$\mathbb{L}_h$. This excludes the second case.
\smallskip

\noindent (d): This follows from statement (c), as
$h^0(\mathbb{P}^r,\mathcal{J}_X(2)) \geq h^0(\mathbb{P}^r,
\mathcal{J}_{W_h}(2)) = \binom{r-3}{2}$.
\end{proof}

\section{Extremal Varieties}
\label{4'. Extremal Varieties}

\noindent In this section we first define the notion of extremal
variety $\mathbb{F}(X)$ of a surface $X \subset \mathbb{P}^r$ of
maximal sectional regularity as the closed union of all special
extremal secant lines to $X$ hence of all lines in the set
${}^*\Sigma^{\circ}_{d-r+3}(X) = \{\mathbb{L}_h \mid h \in
\mathbb{U}(X))$ (see Propostion~\ref{3.1'' Proposition} (b)). As a
main result of this section, we will show that $\mathbb{F}(X)$ is
either a plane or a smooth rational $3$-fold scroll and that the
latter case only occurs if $r = 5$.

\begin{definition and remark}\label{4.2' Definition and Remark} (A) Let $5 \leq r < d$ and let $X\subset \mathbb{P}^r$ be a surface of degree $d$ which is of
maximal sectional regularity. We write
$$\mathbb{F}(X) := \overline{\bigcup_{h \in \mathbb{U}(X)}\mathbb{L}_h} = \overline{\bigcup_{\mathbb{L} \in {}^*\Sigma^{\circ}_{d-r+3}(X)} \mathbb{L}} \subset \mathbb{P}^r$$
for the closure of the union of all special extremal secant lines to
$X$ and call $\mathbb{F}(X)$ the \textit{extremal variety} of $X$.

\noindent (B) As $\mathbb{U}(X)$ is an infinite set of
$(d-r+3)$-secant lines of $X$, we clearly must have
$$\dim (X\cap \mathbb{F}(X)) \geq 1 \mbox{ and } \dim(\mathbb{F}(X)) \geq 2.$$

\noindent (C) Assume first, that our surface $X \subset
\mathbb{P}^r$ is a cone with vertex $p$ over the irreducible
non-degenerate curve $C = X\cap \mathbb{H} \subset \mathbb{H}$ of
degree $d$, where $\mathbb{H} = \mathbb{P}^{r-1} \subset
\mathbb{P}^r$ is a hyperplane which avoids $p$. Now, for all $h \in
\mathbb{U}(X)$ we must have $p \notin \mathbb{H}_h$, so that $C$ and
$C_h$ are isomorphic via the projection $\pi_p: \mathbb{H}
\stackrel{\cong}{\rightarrow} \mathbb{H}_h$ from $p$. This shows,
that $C \subset \mathbb{H}$ is a curve of maximal regularity and
that its unique extremal secant line $\mathbb{L}$ is mapped onto
$\mathbb{L}_h$ by $\pi_p$. Therefore, if $X$ is a cone, we see:
\begin{itemize}
               \item[\rm{(1)}] $\mathbb{F}(X) = \mathbb{P}^2 \subset \mathbb{P}^r \mbox{ with } p \in \mathbb{F}(X)$.
               \item[\rm{(2)}] $\{\mathbb{L}_h \mid h \in \mathbb{U}(X)\} = \{\mathbb{L} \in \Sigma_{d-r+3}(X) \mid p \notin \mathbb{L}\}$.
\end{itemize}

\noindent (D) Assume now, that $X \subset \mathbb{P}^r$ is an
arbitrary non-degenerate irreducible surface of degree $d$. Then,
instead of the previously introduced extremal variety
$\mathbb{F}(X)$ of $X$ one also can introduce the \textit{extended
extremal variety} of $X$ as the closed union
$$\mathbb{F}^+(X) := \overline{\bigcup_{\mathbb{L} \in \Sigma^{\circ}_{d-r+3}(X)} \mathbb{L}}$$
of all proper extremal secant lines to $X$. If $X \subset \mathbb{P}^r$ is of maximal sectional regularity, we clearly have
$$\mathbb{F}(X) \subseteq \mathbb{F}^+(X),$$
and it might indeed be of interest to know, under which circumstances we have equality in this context.
\end{definition and remark}
\smallskip

In fact, we may extend claim (2) of Definition and Remark~\ref{4.2' Definition and Remark} (C) to arbitrary surfaces of maximal sectional regularity as follows.

\begin{proposition}\label{4.3'' Lemma}
Let the notations and hypotheses be as in Definition and Remark~\ref{4.2' Definition and Remark}.
\begin{itemize}
\item[\rm{(a)}] The following equalities hold
\begin{align*}
\{\mathbb{L}_h \mid h \in \mathbb{U}(X)\} & = \{\mathbb{L} \in \Sigma_{d-r+3}(X) \mid X \cap \mathbb{L} \mbox{ is a finite subset of } \Reg(X)\}\\
& =  \{\mathbb{L} \in \Sigma_3(X) \mid X \cap \mathbb{L} \mbox{ is a finite subset of } \Reg(X)\}.
\end{align*}
\item[\rm{(b)}] The set $\{\mathbb{L}_h \mid h \in \mathbb{U}(X) \} = {}^*\Sigma^{\circ}_{d-r+3}(X)$ is locally closed in $\mathbb{G}(1, \mathbb{P}^r)$.
\item[\rm{(c)}] If $p \in X \setminus {\rm Reg}(X)$, then $\dim(X \cap \langle p, \mathbb{L}_h \rangle) > 0$ for all $h \in \mathbb{U}(X)$.
\end{itemize}
\end{proposition}

\begin{proof}
(a): The inclusion $''\subseteq``$ between the first and the second
set follows from the fact, that $C_h := X \cap \mathbb{H}_h$ is
smooth for each $h \in \mathbb{U}(X)$ and hence can only contain
smooth points of $X$. The inclusion between the second and the third
set in our statement is immediate.

So, let $\mathbb{L} \in \Sigma_3(X)$ such that $X \cap \mathbb{L}$
is finite and contained in $\Reg(X)$, and assume that $\mathbb{L}_h
\neq \mathbb{L}$ for all $h \in \mathbb{U}(X)$. We aim for a
contradiction.

By Theorem~\ref{3.2' Theorem} (a)(3) we may write $X =
\widetilde{X}_{\Lambda}$ and $\Sing(\pi_{\Lambda}: \widetilde{X}
\twoheadrightarrow X) = X \setminus \Reg(X)$, where $\widetilde{X}
\subset \mathbb{P}^{d+1}$ is a surface scroll, $\Lambda =
\mathbb{P}^{d-r} \subset \mathbb{P}^{d+1}$ is a subspace disjoint to
$\widetilde{X}$ and $\pi_{\Lambda}$ is induced by a linear
projection $\pi'_{\Lambda}:\mathbb{P}^{d+1}\setminus \Lambda
\twoheadrightarrow \mathbb{P}^r$. Let $\widetilde{\mathbb{L}}:=
\overline{(\pi'_{\Lambda})^{-1}(\mathbb{L})} = \mathbb{P}^{d-r+2}
\subset \mathbb{P}^{d+1}$. Then, the set $\widetilde{X} \cap
\widetilde{\mathbb{L}} = (\pi_{\Lambda})^{-1}(X\cap \mathbb{L})$ is
finite.

Let $\widetilde{\mathbb{H}} = \mathbb{P}^d \subset \mathbb{P}^{d+1}$
be a general hyperplane which contains the space
$\widetilde{\mathbb{L}}$. If $\widetilde{X}$ is not a cone, we may
conclude by \cite[Remark 2.3 (B)]{BP2}, that the intersection
$\widetilde{X} \cap \widetilde{\mathbb{H}} \subset
\widetilde{\mathbb{H}}$ is a rational normal curve. If
$\widetilde{X}$ is a cone, the fact that $\mathbb{L}$ avoids the
singular locus of $X$ implies that $\widetilde{\mathbb{L}}$ does not
contain the vertex of $\widetilde{X}$ and we end up again with the
conclusion that $\widetilde{X} \cap \widetilde{\mathbb{H}} \subset
\widetilde{\mathbb{H}}$ is a rational normal curve.

By Theorem~\ref{3.2' Theorem} (b), there is some $h \in
\mathbb{W}(X)$ such that $\mathbb{H}_h =
\pi'_{\Lambda}(\widetilde{\mathbb{H}}\setminus \Lambda) =
\mathbb{P}^{r-1} \subset \mathbb{P}^r$. As $\widetilde{\mathbb{H}}$
is general and $X \cap \mathbb{L}$ avoids the finite set $X
\setminus \Reg(X) = \Sing(\pi_{\Lambda})$, the intersection $C_h = X
\cap \mathbb{H}_h$ avoids the set $\Sing(\pi_{\Lambda})$. Therefore
the induced map $\pi_{\Lambda}\upharpoonright: \widetilde{X} \cap
\widetilde{\mathbb{H}} \rightarrow C_h$ is an isomorphism and hence
$C_h \subset \mathbb{H}_h$ is a smooth rational curve of degree $d$,
so that $h \in \mathbb{U}(X)$. By our assumption we have $\mathbb{L}
\neq \mathbb{L}_h$, hence $\mathbb{V} := \langle
\mathbb{L}_h,\mathbb{L}\rangle = \mathbb{P}^s \subset \mathbb{H}_h$
with $s \in \{2,3\}$. As $C_h \subset \mathbb{H}_h =
\mathbb{P}^{r-1}$ is an irreducible non-degenerate curve and $s \leq
3 < r-1$, the intersection $X \cap \mathbb{V}$ is finite.

As $\mathbb{L} \neq \mathbb{L}_h$ we have $\#(X \cap (\mathbb{L}
\cup \mathbb{L}_h)) \geq \#(X \cap \mathbb{L}) + \#(X \cap
\mathbb{L}_h) - \varepsilon$ with $\varepsilon = 1$ if $\mathbb{L}$
and $\mathbb{L}_h$ meet in a point of $X$, and $\varepsilon = 0$
otherwise. In the first case, we have $s = 2$, so that always
$3-\varepsilon \geq s$. Therefore, we obtain
$$\infty > \#(X \cap \mathbb{V}) \geq \# (X \cap (\mathbb{L} \cup \mathbb{L}_h)) \geq \#(X \cap \mathbb{L}) +
\#(X\cap \mathbb{L}_h) - \varepsilon \geq 3 + d-r+3 - \varepsilon
\geq d-r+s+3.$$ As $\mathbb{L} \cup \mathbb{L}_h \subset {\rm
Reg}(X)$ this contradicts Theorem~\ref{3.2' Theorem} (a)(5).
\smallskip

\noindent (b): This is clear, as $\Sigma_{d-r+3}(X) \subset
\mathbb{G}(1,\mathbb{P}^r)$ is closed (see Notation and
Reminder~\ref{4.1'' Notation and Reminder} (B)) and $X \setminus
{\rm Reg}(X)$ is finite (see Theorem~\ref{3.2' Theorem} (a)).
\smallskip

\noindent (c): Let $h \in \mathbb{U}(X)$, and observe that by
statement (a) we have $\mathbb{M} := \langle p,\mathbb{L}_h \rangle
= \mathbb{P}^2$. Assume that $\#(X \cap \mathbb{M}) < \infty$. By
statement (a) we have $X \cap \mathbb{L}_h \subset {\rm Reg}(X) \cap
\mathbb{M}$. By our choice of $p$ we have $p \in (X \cap \mathbb{M})
\setminus {\rm Reg}(X)$. Therefore we get
$$d-r+5 \leq \#(X \cap \mathbb{L}_h) + 2 \leq \#\big({\rm Reg}(X) \cap \mathbb{M}\big) + \# \big((X \cap \mathbb{M})_{\rm red} \setminus {\rm Reg}(X)\big).$$
But this contradicts Theorem~\ref{3.2' Theorem} (a)(5) and hence proves our claim.
\end{proof}

We now give a first criterion for the planarity of the extremal variety $\mathbb{F}(X)$ of our surface $X \subset \mathbb{P}^r$ of maximal sectional regularity.
We begin with the following auxiliary result.

\begin{lemma}\label{4.4'' Lemma} Let $5 \leq r < d$, let $X \subset \mathbb{P}^r$ be a surface of degree $d$, and let
$\mathbb{F} = \mathbb{P}^2 \subset \mathbb{P}^r$ be a plane, such that $\dim(X \cap \mathbb{F}) = 1$. Then, the following statements hold
\begin{itemize}
\item[\rm{(a)}] If $\deg_{\mathbb{F}}(X\cap \mathbb{F}) \geq d-r+3$, then $X$ is of maximal sectional regularity and $\mathbb{F}(X) = \mathbb{F}$.
\item[\rm{(b)}] If $X$ is of maximal sectional regularity and $\deg_{\mathbb{F}}(X \cap \mathbb{F}) \geq 3$, then $\mathbb{F}(X) = \mathbb{F}$.
\item[\rm{(c)}] If $X$ is of maximal sectional regularity and $\mathbb{L}_h \subset \mathbb{F}$ for general $h \in \mathbb{U}(X)$, then $\mathbb{F}(X) = \mathbb{F}$.
\end{itemize}
\end{lemma}

\begin{proof}
(a): Set $t := \deg_{\mathbb{F}}(X \cap \mathbb{F})$ and let
$\mathbb{H} = \mathbb{P}^{r-1} \subset \mathbb{P}^r$ be a general
hyperplane. Then, the line $\mathbb{L} := \mathbb{F} \cap
\mathbb{H}$ is $t$-secant to the integral curve $C := X\cap
\mathbb{H} \subset \mathbb{H}$ of degree $d$. Therefore $t \leq
\reg(C) \leq d-r+3$, whence $t = d-r+3$. Thus $X \cap \mathbb{H} = C
\subset \mathbb{H} = \mathbb{P}^{r-1}$ is a curve of maximal
regularity $d-r+3$. As $\mathbb{H}$ is general, we may assume that
$\mathbb{H} = \mathbb{H}_h$ for some $h \in \mathbb{U}(X)$ and
therefore $\sreg(X) = d-r+3$.
\smallskip

\noindent (b): Let $\mathbb{P}^1 = \mathbb{L} \subset \mathbb{F}$ be
a general line. Then $X \cap \mathbb{L} \subset \Reg(X)$ and $\#(X
\cap \mathbb{L}) = \#\big((X \cap \mathbb{F})\cap \mathbb{L}\big) =
\deg_{\mathbb{F}}(X \cap \mathbb{F}) \geq 3$. So, by
Proposition~\ref{4.3'' Lemma} (a) we have $\deg_{\mathbb{F}}(X \cap
\mathbb{F}) = \#(X \cap \mathbb{L}) \geq d-r+3$. Now, we may
conclude by statement (a). \smallskip

\noindent (c): Clearly, for general $h \in \mathbb{U}(X)$, we have
$$d-r+3 = \#(C_h \cap \mathbb{L}_h) = \#(X \cap \mathbb{L}_h) \leq \#\big((X \cap \mathbb{F})\cap \mathbb{H}_h\big) = \deg(X \cap \mathbb{F}).$$
So, our claim follows by statement (b).
\end{proof}

Now, we can prove the announced criterion for the planarity of the extremal variety.

\begin{proposition}\label{4.5'' Proposition} Let $5 \leq r < d$ and let $X \subset \mathbb{P}^r$ be surface of degree $d$ and of maximal sectional regularity. Then,
the following conditions are equivalent.
\begin{itemize}
\item[\rm{(i)}] $\mathbb{F}(X)$ is a plane.
\item[\rm{(ii)}] $\dim(\mathbb{F}(X)) = 2$
\item[\rm{(iii)}] For all $h_1, h_2 \in \mathbb{U}(X)$ the lines $\mathbb{L}_{h_1}$ and $\mathbb{L}_{h_2}$ meet.
\item[\rm{(iv)}] There is a non-empty open set $U \subset \mathbb{U}(X)$ such that the lines $\mathbb{L}_{h_1}$ and $\mathbb{L}_{h_2}$ meet for all $h_1, h_2 \in U$.
\end{itemize}
\end{proposition}

\begin{proof}
(i) $\Rightarrow$ (ii) is obvious.

(ii) $\Rightarrow$ (i): Assume that $\mathbb{F}(X)$ is of dimension
$2$ and write $\mathbb{F}(X) = D_1 \cup \cdots \cup D_t \cup E$,
where $t \in \mathbb{N}$, $D_1,\ldots, D_t \subset \mathbb{P}^r$ are
the different $2$-dimensional irreducible components of
$\mathbb{F}(X)$ and $E \subset \mathbb{P}^r$ is closed, reduced and
of dimension $\leq 1$. For general $h \in \mathbb{U}(X)$, there is
an index $i_h \in \{1,\ldots,t\}$ such that $\mathbb{L}_h \subset
D_{i_h}$. So, without loss of generality we may assume that there is
a non-empty open set $U \subset \mathbb{U}(X)$ such that
$\mathbb{L}_h \subset D_1$ for all $h \in U$. By Bertini we thus
find a non-empty open set $W \subset U$ such that $D_1 \cap
\mathbb{H}_h \subset \mathbb{P}^r$ is an integral closed subscheme
of dimension $1$ and of degree $\deg(D_1)$ for all $h \in W$. As
$\mathbb{L}_h \subset D_1 \cap \mathbb{H}_h$ for all $h \in U$, it
follows, that $D_1 \cap \mathbb{H}_h = \mathbb{L}_h$ for all $h \in
W$. This shows, that $D_1 \subset \mathbb{P}^r$ is a plane which
contains $\mathbb{L}_h$ for all $h \in W$, and hence proves our
claim by Lemma~\ref{4.4'' Lemma} (c).

The implications (i) $\Rightarrow$ (iii) and (iii) $\Rightarrow$
(iv) are obvious, so that it remains to show the implication (iv)
$\Rightarrow$ (i). We find $h_1, h_2 \in U$ such that
$\mathbb{L}_{h_1}$ and $\mathbb{L}_{h_2}$ are distinct, and hence
span a plane $\mathbb{P}^2 =\mathbb{F} \subset \mathbb{P}^r$.
Clearly, there is an non-empty open set $W \subset U$ such that
$\mathbb{L}_h$ avoids the common point of $\mathbb{L}_{h_1}$ and
$\mathbb{L}_{h_2}$ for all $h \in W$. So, for all $h \in W$ we must
have $\mathbb{L}_h \subset \mathbb{F}$. By Lemma~\ref{4.4'' Lemma}
(c) it follows that $\mathbb{F}(X) = \mathbb{F}$.
\end{proof}

Our next aim is to give a link between the invariant $b_h$ of Theorem~\ref{3.2' Theorem} (c) and the nature of the extremal variety $\mathbb{F}(X)$ of $X$. We begin with a
few preparations.

\begin{notation and remark}\label{4.6'' Notation and Remark} (A) Let $5 \leq r \leq d$ and let $C \subset \mathbb{P}^{r-1}$ be a curve of degree $d$, of maximal
regularity and with extremal secant line $\mathbb{L}_C$. Then, the
variety ${\rm Join}(\mathbb{L}_C,C) \subset \mathbb{P}^{r-1}$ is
known to be a threefold scroll of type $S(0,0,r-3)$ with vertex
$\mathbb{L}_C = S(0,0) \subset S(0,0,r-3) = {\rm
Join}(\mathbb{L}_C,C)$.

\noindent (B) Keep the above notations and let $X =
\widetilde{X}_{\Lambda} \subset \mathbb{P}^r$ be a non-conic surface
of maximal sectional regularity, where $\widetilde{X} = S(a,d-a)
\subset \mathbb{P}^{d+1}$, and where the subspace $\Lambda =
\mathbb{P}^{d-r} \subset \mathbb{P}^{d+1}$ and also the induced
projection morphism $\pi_{\Lambda}: \widetilde{X} \twoheadrightarrow
X$ are defined as in Theorem~\ref{3.2' Theorem} (a). Observe in
particular that $a > 0$, so that the scroll $\widetilde{X}$ is
smooth. We fix a canonical projection $\varphi:\widetilde{X}
\twoheadrightarrow \mathbb{P}^1$. For each closed point $x  \in
\mathbb{P}^1$, let $\mathbb{L}(x)$ denote the ruling line
$\varphi^{-1}(x)$ of $\widetilde{X}$ and let
$\mathbb{L}_{\Lambda}(x) := \pi_{\Lambda}(\mathbb{L}(x)) =
\mathbb{P}^1 \subset \mathbb{P}^r$. Then it obviously holds
$$X = \bigcup_{x \in \mathbb{P}^1}\mathbb{L}_{\Lambda}(x).$$

\noindent (C) Keep the previous notations and hypotheses and let $p
\in X$ be a closed point. As $\pi_{\Lambda}: \widetilde{X}
\twoheadrightarrow X$ is finite and almost non-singular with
$\Sing(\pi_{\Lambda}) = X \setminus \Reg(X)$, we have
$$1 \leq \#\{x \in \mathbb{P}^1 \mid p \in \mathbb{L}_{\Lambda}(x)\} < \infty, \mbox{ with equality at the first place if } p \in \Reg(X).$$
Now, let $C \subset X$ be a closed integral subscheme of dimension $1$ whose linear span $\langle C \rangle \subseteq \mathbb{P}^r$ satisfies the condition
$2 \leq \dim \langle C \rangle \leq r-1$. As $\Sing(\pi_{\Lambda}) = X \setminus \Reg(X)$
is a finite set, the closed subscheme $\widetilde{C} := \overline{(\pi_{\Lambda})^{-1}(C \cap \Reg(X))} \subset \widetilde{X}$ is integral and of dimension $1$ with
$2 \leq \dim \langle \widetilde{C} \rangle \leq d$, hence a degenerate prime divisor on $X$, and thus a curve section of $\widetilde{X}$. Therefore
$$\#(\widetilde{C} \cap \mathbb{L}(x)) = 1 \mbox{ for all } x \in \mathbb{P}^1.$$
From this it follows by the previous observation that
$$ \#(C \cap \mathbb{L}_{\Lambda}(x)) = 1, \mbox{ for a general closed point } x \in \mathbb{P}^1.$$
\end{notation and remark}

\begin{lemma}\label{4.7'' Lemma}
Let the notations and hypotheses be as in Notation and Reminder~\ref{4.6'' Notation and Remark} (A). Let $C \subset Y$, where
$Y = S(0,0,r-3) \subset \mathbb{P}^{r-1}$ is a threefold scroll. Then the vertex $\mathbb{L} = S(0,0) \subset S(0,0,r-3)$ of $Y$ equals $\mathbb{L}_C$ and $Y =
{\rm Join}(\mathbb{L}_C,C)$
\end{lemma}

\begin{proof}
We assume that $\mathbb{L} \neq \mathbb{L}_C$ and aim for a contradiction. As $\#(C \cap \mathbb{L}_C) > 2$ we have $\mathbb{L}_C \subset Y$ and hence
$\langle \mathbb{L}, \mathbb{L}_C \rangle \subset Y$, so that $\mathbb{L}$ and $\mathbb{L}_C$ are coplanar. Now, the linear projection map
$\pi_{\mathbb{L}}: Y \setminus \mathbb{L} \rightarrow S(r-3) \subset \mathbb{P}^{r-3}$ induces a dominant morphism $C\setminus (C\cap \mathbb{L}) \rightarrow S(r-3)$.
As $C$ is smooth, this morphism may be extended to a surjective morphism $\phi: C \twoheadrightarrow S(r-3)$. This implies that
$$\deg(\phi) = \frac{d-\#(C \cap \mathbb{L})}{\deg_{\mathbb{P}^{r-3}}(S(r-3))} \leq \frac{d}{r-3}.$$
As $\mathbb{L}$ and $\mathbb{L}_C$ are coplanar, there is a point $z \in S(r-3)$ such that $\phi(C\cap \mathbb{L}_C) = \{z\}$. As $S(r-3)$ is smooth, this implies that
$$\deg(\phi) = \#\phi^{-1}(z) \geq \#(C \cap \mathbb{L}_C) = d-r+3.$$
The two previous inequalities imply that $\frac{d}{r-3} \geq d-r+3$, which is impossible if $d \geq r$. This contradiction shows that $\mathbb{L} = \mathbb{L}_C$
and hence proves our claim.
\end{proof}

\begin{lemma}\label{4.8'' Lemma}
Let $5 \leq r < d$ and let $X \subset \mathbb{P}^r$ be a surface of degree $d$ and of maximal sectional regularity.
Then, for each $h \in \mathbb{U}(X)$, the linear projection $\pi'_{\mathbb{L}_h}: \mathbb{P}^r \setminus \mathbb{L}_h \twoheadrightarrow \mathbb{P}^{r-2}$ from
$\mathbb{L}_h = \mathbb{P}^1 \subset \mathbb{P}^r$ induces a birational morphism
$$\pi_{\mathbb{L}_h}: X \setminus (X \cap \mathbb{L}_h) \rightarrow Z_h := \overline{\pi_{\mathbb{L}_h}\big(X\setminus (X\cap \mathbb{L}_h)\big)} \subset \mathbb{P}^{r-2},$$
and $Z_h \subset \mathbb{P}^{r-2}$ is a rational surface scroll of
type $S(b_h,r-3-b_h)$, where $b_h \leq \frac{r-3}{2}$ is as in
Theorem~\ref{3.2' Theorem} (c). Moreover, if $X$ is non-conic and
with $\mathcal{S}_h := \{x \in \mathbb{P}^{1} \mid
\mathbb{L}_{\Lambda}(x) \cap \mathbb{L}_h \neq \emptyset\}$, the
following statements hold:
\begin{itemize}
\item[\rm{(a)}] $1 \leq \#\mathcal{S}_h \leq d-r+3$, and $\pi_{\mathbb{L}_h}(\mathbb{L}_{\Lambda}(x))$ is a line if $x \in \mathbb{P}^1 \setminus \mathcal{S}_h$ and a
point if $x \in \mathcal{S}_h$.
\item[\rm{(b)}] If $z \in Z_h$ is a general point, there is a unique point $x_z \in \mathbb{P}^1$ with $z \in \pi_{\mathbb{L}_h}(\mathbb{L}_{\Lambda}(x_z))$.
\item[\rm{(c)}] The set $T_h := \Sing(\pi_{\mathbb{L}_h}) \cap \pi_{\mathbb{L}_h}(X \setminus(X\cap \mathbb{L}_h))$ is finite.
\item[\rm{(d)}] If $b_h > 0$ and $C \subset X$ is an integral closed subscheme of dimension $1$ with $4 \leq \dim \langle C \rangle \leq r-1$,
then $C' := \overline{\pi_{\mathbb{L}_h}(C \setminus (C \cap \mathbb{L}_h)}$ is a curve section of $Z_h$.
\end{itemize}
\end{lemma}

\begin{proof}
Observe that $\pi_{\mathbb{L}_h}$ coincides with the restriction $\pi'_{\mathbb{L}_h}\upharpoonright_{X \setminus (X\cap \mathbb{L}_h)}$ of
the linear projection map $\pi'_{\mathbb{L}_h}: \mathbb{P}^r \setminus \mathbb{L}_h \twoheadrightarrow \mathbb{P}^{r-2}$ to
$X \setminus (X \cap \mathbb{L}_h)$. Now, in the notations of Theorem~\ref{3.2' Theorem} (c), we have $\mathbb{L}_h = S(0,0) \subset S(0,0,b_h,r-3-b_h) = W_h =
{\rm Join}(\mathbb{L}_h,X)$, so that
$$Z_h = \overline{\pi'_{\mathbb{L}_h}(W_h \setminus \mathbb{L}_h)}  = S(b_h,r-b_h-3) \subset \mathbb{P}^{r-2}.$$
As $X$ is a union of lines, the same is true for the (irreducible non-degenerate) closed subset $Z_h \subset \mathbb{P}^{r-2}$. In particular we must have $\dim Z_h \geq 2$.
Now, let $p \in Z_h$ be a general point. Then $\pi^{-1}_{\mathbb{L}_h}(p)$ is a finite non-empty set and $\mathbb{M} := \overline{(\pi'_{\mathbb{L}_h})^{-1}(p)} =
\mathbb{P}^2 \subset \mathbb{P}^r$ is a plane which contains $\mathbb{L}_h$. As $X \cap \mathbb{L}_h \subset {\rm Reg}(X)$ (see Proposition~\ref{4.3'' Lemma} (a), as
$X \setminus {\Reg}(X)$ is finite and because $p \in Z_h$ is general, we may assume that $X \cap \mathbb{M} \subset {\rm Reg}(X)$. By Theorem~\ref{3.2' Theorem} (a)(5)
we thus obtain
$$\#(\pi^{-1}_{\mathbb{L}_h}(p)) + \#(X \cap \mathbb{L}_h) = \#(\pi^{-1}_{\mathbb{L}_h}(p) \cup \big(X\cap \mathbb{L}_h)\big) = \#(X \cap \mathbb{M}) \leq d-r+4.$$
As $\#(X \cap \mathbb{L}_h) = d-r+3$, it follows that $\#(\pi^{-1}(p)) \leq 1$. So $\pi_{\mathbb{L}_h}:X \rightarrow Z_h$ is indeed birational.

Assume from now on, that $X$ is non-conic.
\smallskip

\noindent (a): This follows by Notation and Remark~\ref{4.6''
Notation and Remark} (C) and the fact that $X \cap \mathbb{L}_h
\subset \Reg(X)$ (see Proposition~\ref{4.3'' Lemma} (a)).
\smallskip

\noindent (b): Let $q \in \mathbb{L}$ be general. Then
$\pi^{-1}_{\mathbb{L}_h}(q)$ consists of a single point $ p \in X$,
which indeed belongs to $\Reg(X)$. Now, we may again conclude by
Notation and Remark~\ref{4.6'' Notation and Remark} (C).
\smallskip

\noindent (c): Let $t \in U_h = T_h \cap {\rm Reg}(Z_h) =
\Sing(\pi_{\mathbb{L}_h}) \cap \pi_{\mathbb{L}_h}(X\setminus (X \cap
\mathbb{L}_h)) \cap {\rm Reg}(Z_h)$. As  $t$ is a regular point of
$Z_h$ and the morphism $\pi_{\mathbb{L}_h}: X \rightarrow Z_h$ is
birational, it follows that $\pi^{-1}_{\mathbb{L}_h}(t) \subset X$
has pure dimension $1$. As $X$ is a surface, $U_h$ must be finite.
As $Z_h$ has at most one singularity, this proves our claim.
\smallskip

\noindent (d): Observe, that $C' \subset Z_h$ is an integral closed
subscheme of dimension $1$ but not a line. So, it suffices to show,
that $\#(C' \cap \mathbb{L}) \leq 1$ for a general ruling line
$\mathbb{L}$ of a fixed ruling family of $Z_h$ . Let $z \in Z_h$ be
general. We consider the line
$\mathbb{L}:=\pi_{\mathbb{L}_h}(\mathbb{L}_{\Lambda}(x_z))$ of
statement (b). Assume first, that $Z_h \neq S(1,1)$ , so that $Z_h$
admits only one family of ruling lines.  As $z$ is general in $Z_h$,
the line $\mathbb{L}$ is the unique ruling line of $Z_h$ passing
through $z$. By statement (c) we may assume that $\mathbb{L}$ avoids
the set $\Sing(\pi_{\mathbb{L}_h}) \subset Z_h$, so that
$\pi^{-1}_{\mathbb{L}_h}(\mathbb{L}) = \mathbb{L}_{\Lambda}(x_z)$.
Hence by Notation and Remark~\ref{4.6'' Notation and Remark} (C) we
get $\#\pi^{-1}_{\mathbb{L}_h}(C' \cap \mathbb{L}) = \#(C \cap
\mathbb{L}_{\Lambda}(x_z)) \leq 1$.

If $Z_h = S(1,1)$, one of the two ruling families of $Z_h$ contains
the line $\mathbb{L}:=\pi_{\mathbb{L}_h}(\mathbb{L}_{\Lambda}(x_z))$
for general $z \in Z_h$. Now we may conclude as above.
\end{proof}

\begin{proposition}\label{4.9'' Theorem}
Let $5 \leq r < d$ and let $X \subset \mathbb{P}^r$ be a surface of degree $d$ and of maximal sectional regularity. Let the
notations be as in Theorem~\ref{3.2' Theorem} (c).
\begin{itemize}
\item[\rm{(a)}] If $b_h = 0$ for some $h \in \mathbb{U}(X)$, then it holds
\begin{itemize}
\item[\rm{(1)}] $b_{h'} = 0$ and $\mathbb{F}(X) = S(0,0,0) \subset S(0,0,0,r-3) = W_{h'}$ for all $h' \in \mathbb{U}(X)$;
\item[\rm{(2)}] $\mathbb{F}(X) = \mathbb{P}^2$.
\end{itemize}
\item[\rm{(b)}] If $b_h > 0$ for some $h \in \mathbb{U}(X)$, then it holds
\begin{itemize}
\item[\rm{(1)}] $r = 5$ and, in addition, for all $h' \in \mathbb{U}(X)$ we have $b_{h'} = 1$ and $\mathbb{L}_{h'}$ is either equal or disjoint to $\mathbb{L}_h$;
\item[\rm{(2)}] $X \subset \mathbb{F}(X)$ and $\dim(\mathbb{F}(X)) = 3$.
\end{itemize}
\end{itemize}
\end{proposition}

\begin{proof}
Fix some $h \in \mathbb{U}(X)$ and consider the non-empty open subset of $\mathbb{U}(X)$ defined by
$$U:= \{h' \in \mathbb{U}(X) \mid \mathbb{L}_{h'} \nsubseteq \mathbb{H}_h\} = \{h' \in \mathbb{U}(X) \mid \mathbb{L}_{h'} \neq \mathbb{L}_h\}.$$
As in Theorem~\ref{3.2' Theorem} (c), let $W_h := {\rm Join}(\mathbb{L}_h, X) = S(0,0,b_h,r-3-b_h) \subset \mathbb{P}^r$. Then, for any $h' \in U$ we have
$\#(W_h \cap \mathbb{L}_{h'}) \geq \#(X \cap \mathbb{L}_{h'}) \geq d-r+3 > 2$, thus $\mathbb{L}_{h'} \subset W_h$ and hence
$$C_{h'} \cup \mathbb{L}_{h'} \subset \mathbb{H}_{h'} \cap W_h.$$
Keep in mind, that
$$V_{h,h'} := \mathbb{H}_{h'} \cap W_h \subset \mathbb{H}_{h'} = \mathbb{P}^{r-1}$$
is a threefold scroll of type $S(0,b_h,r-3-b_h)$. Observe that the vertex $q$ of $V_{h,h'}$ is the intersection of $\mathbb{L}_h$ with $\mathbb{H}_{h'}$. Observe also
that the restriction of the linear projection $\pi'_{\mathbb{L}_H}: \mathbb{P}^r\setminus \mathbb{L}_h \twoheadrightarrow \mathbb{P}^{r-2}$ to $\mathbb{H}_{h'} \setminus
\{q\}$ yields the linear projection $\pi'_q$ centered at q, thus:
$$\mathbb{L}_h \cap \mathbb{H}_{h'} = \{q\} \quad \mbox{ and } \quad \pi'_{\mathbb{L}_h}\upharpoonright_{\mathbb{H}_{h'}\setminus\{q\}}=\pi'_q: \mathbb{H}_{h'}
\setminus\{q\}\twoheadrightarrow \mathbb{P}^{r-2}.$$

Suppose first, that $b_h = 0$ and let $\mathbb{L}$ be the line
$S(0,0) \subset V_{h,h'} = S(0,0,r-3)$. By Lemma~\ref{4.7'' Lemma}
it follows that $\mathbb{L} = \mathbb{L}_{h'}$ and hence that
$\mathbb{L}_{h'} \subset S(0,0,0) \subset S(0,0,0,r-3-b_h) = W_h$.
Now, Lemma~\ref{4.4'' Lemma} (c) implies that $\mathbb{F}(X) =
S(0,0,0) = \mathbb{P}^2$.

Suppose now, that $b_h > 0$ and let $h' \in U$. According to
Lemma~\ref{4.8'' Lemma} (d), the curve $C' :=
\overline{\pi_{\mathbb{L}_h}(C_{h'} \setminus (C_{h'} \cap
{\mathbb{L}_h}))}$ is a curve section of $Z_h = S(b_h,r-3-b_h)$,
with $\langle C' \rangle = \pi'_{\mathbb{L}_h}(\mathbb{H}_{h'}) =
\mathbb{P}^{r-2}$. Moreover by Lemma~\ref{4.8'' Lemma} (c), the set
$C' \cap \Sing(\pi_{\mathbb{L}_h})$ is finite, so that the induced
morphism $\pi_{\mathbb{L}_h}\upharpoonright: C_{h'} \setminus
(C_{h'} \cap \mathbb{L}_h) \rightarrow C'$ is birational. As $C'$ is
smooth and rational, this latter morphism extends to a unique
isomorphism
$$\varphi: C_{h'} \stackrel{\cong}{\rightarrow} C'.$$
We aim to show first, that the two distinct lines $\mathbb{L}_h$ and
$\mathbb{L}_{h'}$ are disjoint. Assume that $\mathbb{L}_h$ and
$\mathbb{L}_{h'}$ are not disjoint. Then they must meet in the
vertex $q$ of the scroll $V_{h,h'} = S(0,b_h,r-3-b_h)$. Hence, the
point $p = \pi'_{\mathbb{L}_h}(\mathbb{L}_{h'}\setminus\{q\}) \in
Z_h$ satisfies $p \in C'$ and $\#\varphi^{-1}(p) =
\#(\mathbb{L}_{h'}\cap C_{h'}) = d-r+3 > 1$, a contradiction. This
proves the stated disjointness of the extremal secant lines
$\mathbb{L}_{h'}$ and $\mathbb{L}_{h}$.

By Proposition~\ref{4.5'' Proposition} it now follows, that
$\mathbb{F}(X)$ is not a plane if $b_h > 0$. Applying the previous
arguments to all $h' \in \mathbb{U}(X)$ instead of $h$, we see that
either $b_{h'} = 0$ for all $h' \in \mathbb{U}(X)$ or $b_{h'} > 0$
for all $h' \in \mathbb{U}(X)$. This observation completes in
particular the proof of statement (a).

It remains to complete the proof of statement (b). We first aim to
show, that
$$\dim(\mathbb{F}(X)) \leq 1 + \dim(X \cap \mathbb{F}(X)).$$
To do so, we consider the coincidence set
$$Y: =\{(x,\mathbb{L}_h) \mid h \in \mathbb{U}(X) \mbox{ and } x \in \mathbb{L}_h\} \subset \mathbb{P}^r \times \mathbb{G}(1,\mathbb{P}^r)$$
-- which is locally closed in $\mathbb{G}(1,\mathbb{P}^r)$ by Proposition~\ref{4.3'' Lemma} (b) -- and its locally closed subset
$$T := Y \cap \big (X \times \mathbb{G}(1,\mathbb{P}^r)\big) = \{(x,\mathbb{L}) \in Y \mid x \in X \} \subset \mathbb{P}^r \times \mathbb{G}(1,\mathbb{P}^r).$$
Now, the projection morphism $\varrho: \mathbb{P}^r \times \mathbb{G}(1,\mathbb{P}^r) \twoheadrightarrow \mathbb{P}^r$
maps $T$ to $X \cap \mathbb{F}(U)$ and hence induces a morphism $\varrho\upharpoonright: T \rightarrow X \cap \mathbb{F}(X)$. As any two distinct lines $\mathbb{L}_h$ and
$\mathbb{L}_{h'}$ with $h,h' \in \mathbb{U}(X)$ are disjoint, the map $\varrho\upharpoonright$ is injective, so that
$$\dim(T) \leq \dim(X \cap \mathbb{F}(X)).$$
Moreover we have $\varrho(Y) =\bigcup_{h \in \mathbb{U}(X)} \mathbb{L}_h$, so that $\mathbb{F}(X) = \overline{\varrho(Y)}$ and hence
$$\dim(\mathbb{F}(X)) \leq \dim(Y).$$
The projection morphism $\sigma: \mathbb{P}^r \times \mathbb{G}(1,\mathbb{P}^r) \twoheadrightarrow \mathbb{G}(1,\mathbb{P}^r)$ satisfies $\sigma(Y) = \sigma(T) =
\{\mathbb{L}_h \mid h \in \mathbb{U}(X)\}$. Moreover for each $h \in \mathbb{U}(X)$ we have $\sigma^{-1}(\mathbb{L}_h) \cap Y =
\{(x,\mathbb{L}_h) \mid x \in \mathbb{L}_h\} \cong \mathbb{P}^1$. This shows, that
$$\dim(Y) \leq \dim\big(\overline{\sigma(T)}\big) + 1 \leq \dim(T) + 1,$$
so that indeed $\dim(\mathbb{F}(X)) \leq \dim(Y) \leq \dim(T) + 1
\leq \dim(X \cap \mathbb{F}(X)) + 1$.

As $\mathbb{F}(X) \neq \mathbb{P}^{2}$ we have $\dim(\mathbb{F}(X))
\geq 3$ (see Proposition~\ref{4.5'' Proposition} and Definition and
Remark~ \ref{4.2' Definition and Remark} (B)). It follows that $X
\subset \mathbb{F}(X)$ and $\dim(\mathbb{F}(X)) = 3$. This proves
claim (2) of statement (b).

It remains to complete the prove of claim (1) of statement (b).
Observe that the line $\mathbb{M}_{h'} :=
\pi'_{\mathbb{L}_h}(\mathbb{L}_{h'}) \subset \mathbb{P}^{r-2}$
satisfies $\varphi(C_{h'} \cap \mathbb{L}_{h'}) =
\pi'_{\mathbb{L}_{h}}(C_{h'} \cap \mathbb{L}_{h'}) \subset Z_h \cap
\mathbb{M}_{h'}$. As $\varphi$ is an isomorphism, we have $\# (C'
\cap \mathbb{M}_{h'}) = \#\varphi(C_{h'} \cap \mathbb{L}_{h'}) =
\#(C_{h'} \cap \mathbb{L}_{h'}) = d-r+3 > 2$, and hence
$\mathbb{M}_{h'} \subset Z_h$.

As $\mathbb{L}_{h'} \cap \mathbb{L}_h = \emptyset$, the line
$\mathbb{L}_{h'} \subset V_{h,h'}$ avoids the vertex $q$ of
$V_{h.h'}$ and hence is not contained in any of the ruling planes of
$V_{h,h'}$. As the ruling lines of the surface scroll $Z_h$ are
precisely the images of the ruling planes of $V_{h,h'}$ under the
linear projection map $\pi'_q: \mathbb{H}_{h'}\setminus \{q\}
\twoheadrightarrow \mathbb{P}^{r-2}$ centered at $q$, it follows
that $\mathbb{M}_{h'} = \pi'_{\mathbb{L}_h}(\mathbb{L}_{h'}) =
\pi'_q(\mathbb{L}_{h'}) \subset Z_h$ is not a ruling line of $Z_h$.
So $\mathbb{M}_{h'}$ must be a line section of $Z_h$, and hence $b_h
= 1$.

Our next aim is to show that $r=5$. Assume to the contrary, that $r
\geq 6$. Then $\mathbb{L}_{h'} \subset V_{h,h'} = S(0,1,r-4) \subset
S(0,0,1,r-4) = W_h$ shows that $\mathbb{L}_{h'} \subset S(0,0,1) =
\mathbb{P}^3$ for all $h' \in \mathbb{U}(X)$, so that $\mathbb{F}(X)
\subset \mathbb{P}^3$. But now, by claim (b)(2) we get that $X
\subset \mathbb{P}^3$, and this contradiction shows, that indeed
$r=5$.

Applying this to arbitrary $h' \in \mathbb{U}(X)$ instead of $h$, we
get in particular that $b_{h'} = 1$ for all $h \in \mathbb{U}(X)$,
and claim (1) of statement (b) is shown completely.
\end{proof}

We may summarize the previous result as follows:

\begin{theorem}\label{4.10''' Theorem}
Let $5 \leq r < d$ and let $X \subset \mathbb{P}^r$ be a surface of degree $d$ which is of maximal sectional regularity. Then, in the notations of Theorem~\ref{3.2' Theorem}
(c) we have
\begin{itemize}
\item[\rm{(a)}] $\mathbb{F}(X)$ is a plane if and only if $b_h = 0$ for all $h \in \mathbb{U}(X)$, and this is always the case if $r \geq 6$.
\item[\rm{(b)}] $\mathbb{F}(X)$ is not a plane if and only if $r=5$ and $b_h = 1$ for all $h \in \mathbb{U}(X)$.
\end{itemize}
\end{theorem}

\begin{proof}
This is immediate by Proposition~\ref{4.9'' Theorem}.
\end{proof}

We now aim to describe the extremal variety $\mathbb{F}(X)$ of a surface $X \subset \mathbb{P}^5$ of maximal sectional regularity in case it is not a plane. We begin
with some preparations.

\begin{notation and reminder} \label{4.x Notation and Reminder}
(A) Let $s > 2$ and let $Y \subset \mathbb{P}^s$ be a smooth
rational normal surface scroll with projection morphism $\varphi:Y
\twoheadrightarrow \mathbb{P}^1$. Let $F := \varphi^{-1}(x) =
\mathbb{P}^1 \in {\rm Div}(Y)$ denote a fibre and $H := Y \cap
\mathbb{P}^{s-1} \in {\rm Div}(Y)$ a general hyperplane section, and
let $\overline{F}, \overline{H} \in {\rm Cl}(Y)$ be the
corresponding divisor classes. As $H \subset Y$ is a section of
$\varphi$ we have ${\rm Cl}(Y) = \mathbb{Z}\overline{H}
\oplus\mathbb{Z}\overline{F}$, so that any divisor $D \in {\rm
Div}(Y)$ is linearly equivalent to a divisor of the form $aH + bF$,
with uniquely determined integers $a, b \in \mathbb{Z}$ (see
\cite[Proposition V.2.3]{Ha}).

\noindent (B) Let the notations be as in part (A) and keep in mind
that $H\cdot H = \deg(Y)$, $H \cdot F = 1$ and $F \cdot F = 0$. Let
$a, b \in \mathbb{Z}$ and let $D \in |aH + bF|$. If follows that
$\deg(D) = D \cdot H = a\deg(Y) + b$. As $D$ is effective, it is
linearly equivalent to a divisor of the form $\sum_{i=1}^t n_iC_i +
cF$ with pairwise distinct prime divisors $C_1,\ldots,C_t$ which are
not fibers, with $n_1,\ldots,n_t > 0$, $c \geq 0$ and with $t \geq
0$. If follows, that
$$a = (aH + bF)\cdot F = D\cdot F = (\sum_{i=1}^tn_iC_i + cF)\cdot F = \sum_{i=1}^t n_i C_i \cdot F \geq \sum_{i=1}^t n_i,$$
with equality at the last place if and only if $C_1,\ldots,C_t \subset Y$ are sections.
\end{notation and reminder}

\begin{lemma} \label{4.10'' Lemma}
Let $H$ and $F$ be respectively a general hyperplane section and a ruling line of the rational normal surface scroll $Y := S(1,2) \subset
\mathbb{P}^4$. Let $C \subset \mathbb{P}^4$ be a non-degenerate integral curve of degree $d > 5$ which is of maximal regularity. If $C \subset S(1,2)$, then $C$ is linearly
equivalent to the divisor $H + (d-3)F$ and the line $S(1) \subset S(1,2)$ is the extremal secant line to $C$.
\end{lemma}

\begin{proof}
Let $C$ be linearly equivalent to $aH + bF$. By Notation and Reminder~\ref{4.x Notation and Reminder} (B) we have $a \geq 1$. As the surface $Y$ is
arithmetically Cohen-Macaulay, we have
 $H^i(\mathbb{P}^4, \mathcal{J}_Y(1)) = 0$ for $i=1,2$ and so, the short exact sequence
$$0 \longrightarrow \mathcal{J}_Y \longrightarrow \mathcal{J}_C \longrightarrow \mathcal{O}_Y(-C) \longrightarrow 0$$
implies an isomorphism $H^1(\mathbb{P}^4, \mathcal{J}_C(1)) \cong
H^1\big(\mathbb{P}^4,\mathcal{O}_Y((1-a)H - bF)\big)$. If $a >1$, we
have $H^1\big(\mathbb{P}^4, \mathcal{O}_Y((1-a)H - bF )\big) = 0$,
and we get the contradiction that the curve of maximal regularity $C
\subset \mathbb{P}^4$ is linearly normal (see Proposition 2.7 a) of
\cite{BS2}). Therefore it holds $a=1$. As $d = \deg(C) = \deg(Y) + b
= 3+b$ (see Notation and Reminder~\ref{4.x Notation and Reminder}
(A)), we obtain $b = d-3$ and $C$ is linearly equivalent to $H +
(d-3)F$.

Moreover the line section $\mathbb{L} = S(1)$ of $Y = S(1,2)$
satisfies the condition
$$\#(C \cap \mathbb{L}) = C\cdot \mathbb{L} = (H +(d-3)F)\cdot \mathbb{L} = d-2,$$
and hence $\mathbb{L}$ is indeed the extremal secant line to $C$.
\end{proof}

\begin{notation and reminder}\label{4.16'' Notation and Reminder}
Let $n \in \mathbb{Z}$. We say, that the closed subscheme $Z \subset \mathbb{P}^r$ is $n$-\textit{normal} if
$H^1(\mathbb{P}^r, \mathcal{J}_Z(n)) = 0$ and we introduce the \textit{index of normality} of $Z$ as
$$N(Z) := \sup\{n \in \mathbb{Z} \mid Z \mbox{ is not $n$-normal} \} = \en(\bigoplus _{n \in \mathbb{Z}} H^1(\mathbb{P}^r,\mathcal{J}_Z(n)) = {\rm end}(H^1(S/I_Z)).$$
Keep in mind that $N(Z) \leq \reg(Z)-2$ and that $N(Z) = -\infty$ if $\depth Z > 1$. In particular, if $5 \leq r < d$ and $X \subset \mathbb{P}^r$ is a surface
of maximal sectional regularity and degree $d$, we have
$$N(X) \leq d-r+1.$$
\end{notation and reminder}
\smallskip

\begin{lemma}\label{4.10''' Lemma}
Let $5 \leq r < d$, let $Y \subset \mathbb{P}^r$ be an irreducible
surface of degree $d$ which is contained in the smooth threefold
scroll $Z:= S(1,a, r-a-3)$ with $1 \leq a \leq \frac{r-2}{2}$ as a
divisor linearly equivalent to $H+(d-r+2)F$. Then, the following
statements hold:
\begin{enumerate}
\item[\rm{(a)}]  $N(Y) = d-r+1$ and
$$h^1(\mathbb{P}^r, \mathcal{J}_Y(d-r+1))=\begin{cases}
1 & \mbox{if $a\geq 2$},\\
d-r & \mbox{if $a=1$ and $r \geq 6$, and}\\
\binom{d-3}{2} & \mbox{if $a =1$ and $r=5$ }.
\end{cases}\\$$
\item[\rm{(b)}] The minimal number of generators in degree $d-r+3$ of the homogeneous vanishing ideal $I_Y \subset S$ of $Y$ is given by
$$\beta_{1,d-r+2}(Y)=\begin{cases}
1 & \mbox{if $a\geq 2$},\\
d-r+3 & \mbox{if $a=1$ and $r \geq 6$, and}\\
\binom{d-1}{2} & \mbox{if $a=1$ and  $r=5$}.
\end{cases}$$
\end{enumerate}
\end{lemma}

\begin{proof}
(a): We set $b := r-a-3$ and consider the short exact sequence
\begin{equation*}
0 \rightarrow  \mathcal{J}_Z  \rightarrow \mathcal{J}_Y \rightarrow \mathcal{O}_Z(-Y) \rightarrow 0.
\end{equation*}
Then, since $Z$ is arithmetically Cohen-Macaulay, we have
\begin{equation*}
H^1(\mathbb{P}^r, \mathcal{J}_Y(n)) \cong H^1(Z,\mathcal{O}_Z(-Y + nH)) \mbox{ for all } n \in \mathbb{Z}.
\end{equation*}
Setting $\mathcal{E}: =\mathcal{O}_{\mathbb{P}^1}(1) \oplus \mathcal{O}_{\mathbb{P}^1}(a)\oplus \mathcal{O}_{\mathbb{P}^1}(b)$ we may write
\begin{equation*}
\begin{array}{ll}
H^1(Z,\mathcal{O}_Z(-Y+nH )) &=H^1(Z, \mathcal{O}_Z((n-1)H-(d-r+2)F))\\
                           &=H^1(\mathbb{P}^1, {\rm sym}^{n-1}\mathcal{E} \otimes \mathcal{O}_{\mathbb{P}^1}(-d+r-2)) \\
                           &=\underset{0 \leq i,j, i+j \leq n-1}{\bigoplus}H^1(\mathbb{P}^1, \mathcal{O}_{\mathbb{P}^1}(i+ja+(n-i-j)b-d+r-2)).
\end{array}
\end{equation*}
Altogether, we obtain that $h^1(\mathbb{P}^r, \mathcal{J}_Y (n) = 0$ for all $n > d-r+1$ and
\begin{equation*} h^1(\mathbb{P}^r,\mathcal{J}_Y(d-r+1)) =
\begin{cases}
                              1 & \mbox{if $a\geq 2$},\\
                              d-r & \mbox{if $a=1$ and $b \geq 2$, and}\\
                             \binom{d-3}{2} & \mbox{if $a=b=1$}.
\end{cases}
\end{equation*}
This proves statement (a).
\smallskip

\noindent (b): For this statement, see \cite{P}.
\end{proof}

Now, we are ready to prove the following structure result on non-planar extremal varieties.

\begin{theorem}\label{4.11'' Theorem}
Let $5 \leq r < d$ and let $X \subset \mathbb{P}^r$ be a surface of degree $d$ and of maximal sectional regularity such that
$\mathbb{F}(X)$ is not a plane. Then we have:
\begin{itemize}
\item[\rm{(a)}] $r=5$ and $\mathbb{F}(X) = S(1,1,1)$.
\item[\rm{(b)}] $X$ is contained in $\mathbb{F}(X)$, smooth and linearly equivalent to the divisor $H + (d-3)F$, where $H$ is the hyperplane divisor and $F$ is
ruling plane of $\mathbb{F}(X) = S(1,1,1)$.
\item[\rm{(c)}] $N(X) = d-4$, $\e(X) = 0$ and moreover
                \begin{itemize}
                \item[\rm{(1)}] $h^1(\mathbb{P}^r,\mathcal{J}_X(d-4)) = \binom{d-3}{2}$;
                \item[\rm{(2)}] $h^2(\mathbb{P}^r,\mathcal{J}_X(n)) = 0$ for all $n \in \mathbb{Z}$;
                \item[\rm{(3)}] $h^3(\mathbb{P}^r,\mathcal{J}_X(n)) = 0$ for all $n \geq 0$.
                \end{itemize}
\item[\rm{(d)}] $\beta_{1,d-3}(X) = \binom{d-1}{2}$.
\item[\rm{(e)}] ${}^*\Sigma^{\circ}_{d-2}(X) = \Sigma^{\circ}_{d-2}(X) = \Sigma_3(X) \setminus \Sigma_{\infty}(X)$ and $\mathbb{F}^+(X) = \mathbb{F}(X)$. Moreover
                $\overline{{}^*\Sigma^{\circ}_{d-2}(X)} = \Sigma_3(X)$ and the image of this set under the Pl\"ucker embedding
                $\psi: \mathbb{G}(1,\mathbb{P}^5) \rightarrow \mathbb{P}^{14}$ is a Veronese surface in a subspace $\mathbb{P}^5 \subset \mathbb{P}^{14}$.
\end{itemize}
\end{theorem}

\begin{proof}
(a): By Proposition~\ref{4.9'' Theorem} we already know that $r=5$,
$X \subset \mathbb{F}(X)$ and $\dim(\mathbb{F}(X)) = 3$. Now, let
$h, h' \in \mathbb{U}(X)$ such that $\mathbb{L}_{h'} \neq
\mathbb{L}_h$. Then, according to Proposition~\ref{4.9'' Theorem}
(b), the two $4$-scrolls $W_h$ and $W_{h'}$ in $\mathbb{P}^5$ are
both of type $S(0,0,1,1)$ and $\langle \mathbb{L}_h, \mathbb{L}_{h'}
\rangle \subset \mathbb{P}^5$ is a $3$-space.

As $X \subset W_h$ and $\#(X\cap \mathbb{L}_{h''}) = d-2 > 2$ for
all $h'' \in \mathbb{U}(X)$, we have $\mathbb{F}(X) \subset W_h$.
Moreover, as $\mathbb{L}_h$ is the vertex of $W_h = S(0,0,1,1)$ and
$\mathbb{L}_{h'} \subset W_h$ it holds $\langle \mathbb{L}_h,
\mathbb{L}_{h'} \rangle \subset W_h$. Hence, by symmetry we get
$$X \subset \mathbb{F}(X) \subset W_h \cap W_{h'} \mbox{ and } \mathbb{P}^3 = \langle \mathbb{L}_h,\mathbb{L}_{h'}\rangle \subset W_h \cap W_{h'}.$$
As $W_h$ and $W_{h'}$ are two distinct integral hyperquadrics in
$\mathbb{P}^5$, the intersection $W_h \cap W_{h'}$ is arithmetically
Cohen-Macaulay, satisfies $\dim(W_h \cap W_{h'}) = 3$ and $\deg(W_h
\cap W_{h'}) = 4$. In particular it follows that $W_h \cap W_{h'} =
\langle \mathbb{L}_h, \mathbb{L}_{h'} \rangle \cup D$, where $D
\subset \mathbb{P}^5$ is an non-degenerate integral closed subscheme
of dimension $3$ and degree $3$. As $X \subset \mathbb{P}^5$ is
non-degenerate and contained in $W_h \cap W_{h'}$, we have $X
\subset D$. As $\#(X \cap \mathbb{L}_{h''}) = d-r+3 > 2$, it follows
that $\mathbb{L}_{h''} \subset D$ for all $h'' \in \mathbb{U}(X)$,
and hence that $\mathbb{F}(X) = D$.

Observe that $D$ is a scroll of type $S(1,1,1)$ or $S(0,1,2)$ or
$S(0,0,3)$. We aim to exclude the latter two cases. Assume first,
that $D = S(0,0,3)$. As $X \subset D$ is a Weil divisor, it follows,
that $X$ is arithmetically Cohen-Macaulay, which contradicts
Theorem~\ref{3.2' Theorem} (b). Assume now, that $D = S(0,1,2)$ and
let $h'' \in \mathbb{U}(X)$ be general.  Then, we have
$$C_{h''} \subset D \cap \mathbb{H}_{h''} = S(1,2) \subset H_{h''} = \mathbb{P}^4,$$
and according to Lemma~\ref{4.10'' Lemma}, the line section $S(1)$
of $S(1,2)$ coincides with the extremal secant line
$\mathbb{L}_{h''}$ of $C_{h''}$, so that $\mathbb{L}_{h''}$ is
contained in the plane $\mathbb{P}^2 = S(0,1) \subset S(0,1,2) = D$
for general $h'' \in \mathbb{U}(X)$, and this contradicts
Proposition~\ref{4.9'' Theorem} (b)(1). Therefore indeed
$\mathbb{F}(X) = D = S(1,1,1)$, and statement (a) is shown
completely.
\smallskip

\noindent (b): For all $h \in \mathbb{U}(X)$ we have
$$C_h = X \cap \mathbb{H}_h = X \cap (D \cap \mathbb{H}_h) \subset D \cap \mathbb{H}_h = S(1,2) \subset \mathbb{H}_h = \mathbb{P}^4$$
and Lemma~\ref{4.10'' Lemma} yields that the divisor $X$ is linearly
equivalent to $H + (d-3)F$. It remains to show, that $X$ is smooth.

As $X$ is a divisor of the smooth variety $D = S(1,1,1)$ it is a
Cohen-Macaulay variety. So, by Theorem~\ref{3.2' Theorem} (a)(1) it
follows indeed that $X$ is smooth.
\smallskip

\noindent (c): The equality $N(X) = d-4$ and claim (1) follow
immediately from statements (a) and (b) by Lemma~\ref{4.10''' Lemma}
(a), whereas the vanishing of $\e(X)$ follows from the fact that $X$
is smooth. Claims (2) and (3) follow from the fact that $\e(X) = 0$
by Theorem~\ref{3.2' Theorem}(a)(2).
\smallskip

\noindent (d): This follows from statements (a) and (b) by
Lemma~\ref{4.10''' Lemma} (b).
\smallskip

\noindent (e): By statement (b) the surface $X$ is smooth. So, the
equalities ${}^*\Sigma^{\circ}_{d-2}(X) = \Sigma^{\circ}_{d-2}(X) =
\Sigma_3(X) \setminus \Sigma_{\infty}(X)$ follow immediately by
Proposition~\ref{4.3'' Lemma} (a),(b). The equality
$\overline{{}^*\Sigma^{\circ}_{d-2}(X)} = \Sigma_3(X)$ now follows
easily as $X \subset \mathbb{F}(X)$ (see statement (b)).

To prove the remaining claim, we identify $\mathbb{F}(X) = S(1,1,1)$
with the image of the Segre embedding $\sigma:\mathbb{P}^1 \times
\mathbb{P}^2 \rightarrow \mathbb{P}^5$. Let
$$\Theta := \{\mathbb{P}^1 \times \{q\} \mid q \in \mathbb{P}^2\} \subset \mathbb{G}(1,\mathbb{P}^5)$$
denote the closed subset of all fibers under the canonical
projection $\mathbb{P}^1 \times \mathbb{P}^2 \twoheadrightarrow
\mathbb{P}^2$, and let
$$\Omega := \bigcup_{p \in \mathbb{P}^1} \mathbb{G}(1,\{p\} \times \mathbb{P}^2) \subset \mathbb{G}(1,\mathbb{P}^5)$$
denote the closed subset of all lines contained in some fiber of the
canonical projection $\mathbb{P}^1 \times \mathbb{P}^2
\twoheadrightarrow \mathbb{P}^1$, hence in a ruling plane of
$S(1,1,1)$. By Proposition~\ref{4.9'' Theorem} (b)(1) we have
$$\#\big({}^*\Sigma^{\circ}_{d-2}(X) \cap \mathbb{G}(1,\{p\} \times \mathbb{P}^2)\big) \leq 1 \mbox{ for all } p \in
\mathbb{P}^1,$$
and hence $\dim({}^*\Sigma^{\circ}_{d-2}(X) \cap \Omega) \leq 1$.
Therefore $U := {}^*\Sigma^{\circ}_{d-2}(X) \setminus \Omega$ is a
locally closed dense subset of
$\overline{{}^*\Sigma^{\circ}_{d-2}(X)}$, consists of line sections
of $S(1,1,1)$, and hence is contained in $\Theta$. It follows that
$\overline{{}^*\Sigma^{\circ}_{d-2}(X)} \subseteq \Theta$, so that
we get the inclusion
$$\psi(\overline{{}^*\Sigma^{\circ}_{d-2}(X)}) \subseteq \psi(\Theta).$$
By standard arguments on Pl\"ucker embeddings one sees that
$\psi(\Theta)$ is the Veronese surface in some subspace
$\mathbb{P}^5 \subset \mathbb{P}^{14}$. As the left hand side of the
previous inclusion is a surface, we get our claim.
\end{proof}

\section{Planar Extremal Varieties}
\label{6. Planar Extremal Varieties}

\noindent In this section, we give a few results which concern the
``general`` case in which the extremal variety is a plane. The case
of surfaces of maximal sectional regularity which are cones is
understood by what is said in Definition and Remark~\ref{4.2'
Definition and Remark} (C). Therefore, we restrict ourselves to
consider the case of non-conic surfaces of maximal sectional
regularity. Observe, that according to Theorem~\ref{4.10'''
Theorem}, our results apply to all non-conic surfaces of maximal
sectional regularity in $\mathbb{P}^r$ with $r \geq 6$. We begin
with the following auxiliary results.

\begin{lemma}\label{4.12'' Lemma}
Let $s > 1$, let $C \subset \mathbb{P}^s$ be a closed subscheme of dimension $1$ and degree $d$ and let $\mathbb{H} = \mathbb{P}^{s-1} \subset \mathbb{P}^s$ be a
hyperplane. Then
$$\#(C \cap \mathbb{H}) \geq d \mbox{  with equality if and only if } {\rm Ass}_C(\mathcal{O}_C) \cap \mathbb{H} = \emptyset.$$
\end{lemma}

\begin{proof}
Let $R = K\oplus R_1\oplus R_2\oplus \ldots = K[R_1]$ be the homogeneous coordinate ring of $C$ and let $f \in R_1$ be such that $C\cap \mathbb{H}
= \Proj(R/fR)$. Let $H_R(t) = dt + c$ be the Hilbert polynomial of $R$. Then, the two exact sequences
$$0\rightarrow fR \rightarrow R \rightarrow R/fR\rightarrow 0 \mbox{ and } 0 \rightarrow(0:_R f)(-1) \rightarrow R(-1) \rightarrow fR \rightarrow 0$$
yield that the Hilbert polynomial of $R/fR$ is given by
$$H_{R/fR}(t) = d + H_{(0:_R f)}(t-1).$$
Observe that the polynomial $H_{(0:_R f)}(t-1)$ vanishes if and only if $(0:_R f)_t = 0$ for all $t \gg 0$, hence if and only if
$$f\notin \bigcup_{\mathfrak{p} \in {\rm Ass}(R)\setminus \{R_+\}} \mathfrak{p}.$$
But this latter condition is equivalent to the requirement that ${\rm Ass}(C) \cap \mathbb{H} = \emptyset$.
\end{proof}

\begin{lemma}\label{4.13'' Lemma}
Let $5 \leq r$ and let $C \subset \mathbb{P}^{r-1}$ be a curve of degree $d \geq 3r-6$ which is of maximal regularity and with extremal secant line
$\mathbb{L} \subset \mathbb{P}^{r-1}$. Then $\depth(C \cup \mathbb{L}) = 1$.
\end{lemma}

\begin{proof}
Since $C$ is contained in a rational $3$-fold scroll $\mathbb{S} := S(0,0,r-3)$, we have
 $$\dim_K(I_C)_2 \geq \dim_K(I_{\mathbb{S}})_2 = \binom{r-1}{2} - 2r+5.$$
Assume now, that $\depth(C \cup \mathbb{L}) \neq 1$, so that $C \cup \mathbb{L}$ is arithmetically Cohen Macaulay. Then, by Proposition 3.6 of~\cite{BS2} it follows that
$\dim_K(C)_2 = \binom{r}{2}-d-1$, whence $\binom{r-1}{2} - 2r + 5 \leq \binom{r}{2} - d - 1$, thus the contradiction that $d \leq 3r-7$.
\end{proof}

Now, we can formulate and prove the first main result of this section. We use the notations introduced in Notation and Reminder~\ref{2.4'' Notation and Reminder} and
~\ref{2.7' Notation and Reminder} (C), (D).

\begin{theorem}\label{4.14'' Theorem}
Let $5 \leq r < d$ and let $X \subset \mathbb{P}^r$ be a surface of degree $d$ and of maximal sectional regularity which is not a cone.
Assume that $\mathbb{F} = \mathbb{F}(X)$ is a plane, set $C : = X \cap \mathbb{F}$, $Y := X \cup \mathbb{F}$ and let $I$ and $L$ respectively denote the homogeneous
vanishing ideal of $X$ and of $\mathbb{F}$ in $S$. Then the following statements hold
\begin{itemize}
\item[\rm{(a)}] Each line $\mathbb{L} \subset \mathbb{F}$ which is not contained in $X$, satisfies $\#(C \cap \mathbb{L}) = \#(X \cap \mathbb{L}) = d-r+3$. In particular,
                $C \subset \mathbb{F}$ is a curve of degree $d-r+3$ and has no closed associated points.
\item[\rm{(b)}] $X \setminus {\rm Reg}(X) \subseteq C \setminus {\rm Reg}(C)$.
\item[\rm{(c)}] $I_{d-r+3} \setminus I \cap L \neq \emptyset$ and for each $f \in I_{d-r+3} \setminus I \cap L$ it holds $I = (I \cap L, f)$.
\item[\rm{(d)}] \begin{itemize}
                \item[\rm{(1)}] $h^1(\mathbb{P}^r, \mathcal{J}_X(n)) = h^1(\mathbb{P}^r,\mathcal{J}_Y(n))$ for all $n\in \mathbb{Z}$.
                \item[\rm{(2)}] $h^2(\mathbb{P}^r, \mathcal{J}_X(d-r)) = 1$ and $h^2(\mathbb{P}^r,\mathcal{J}_X(n)) = 0$ for all $n > d-r$.
                \item[\rm{(3)}] $h^2(\mathbb{P}^r,\mathcal{J}_Y(n)) = 0$ for all $n \geq d-r$.
                \item[\rm{(4)}] $h^3(\mathbb{P}^r,\mathcal{J}_X(n-1)) = h^3(\mathbb{P}^r,\mathcal{J}_Y(n)) = 0$ for all $n\geq 0$.
                \end{itemize}
\item[\rm{(e)}] \begin{itemize}
                \item[\rm{(1)}] $h^2(\mathbb{P}^r,\mathcal{J}_X(n)) = h^2(\mathbb{P}^r,\mathcal{J}_Y(n)) + \max\{0,\binom{-n+d-r+2}{2}\}$ for all $n \geq 0$.
                \item[\rm{(2)}] ${\rm e}(X) = h^2(\mathbb{P}^r,\mathcal{J}_Y(0)) + \binom{d-r+2}{2} = h^2(\mathbb{P}^r, \mathcal{J}_X(n))$ for all $n \leq 0$.
                \item[\rm{(3)}] $h^2(\mathbb{P}^r,\mathcal{J}_Y(n+1)) \geq h^2(\mathbb{P}^r,\mathcal{J}_Y(n))$ for all $n \leq 0$, with equality for $n = 0$.
                \item[\rm{(4)}] $h^2(\mathbb{P}^r,\mathcal{J}_Y(n+1)) \leq {\rm max}\{0,h^2(\mathbb{P}^r,\mathcal{J}_Y(n)) - 1\}$ for all $n > 0$.
                \item[\rm{(5)}] If $h^2(\mathbb{P}^r, \mathcal{J}_Y(0) = 0$, then $h^2(\mathbb{P}^r,\mathcal{J}_Y(n)) = 0$ for all $n \in \mathbb{Z}$.
                \end{itemize}
\item[\rm{(f)}] For the pair $\tau(X) := \big(\depth(X),\depth(Y)\big)$ we have
                \begin{itemize}
                \item[\rm{(1)}] $\tau(X) = (2,3)\quad \quad \quad \quad \quad \quad$ if $r+1 \leq d \leq 2r-4$;
                \item[\rm{(2)}] $\tau(X) \in \{(1,1),(2,2),(2.3)\}$ if $2r-3 \leq d \leq 3r-7$;
                \item[\rm{(3)}] $\tau(X) \in \{(1,1). (2,2)\}$ $\mbox{ }\quad \quad$ if $3r-6 \leq d$.
                \end{itemize}
\item[\rm{(g)}] $\reg(Y) \leq d-r+3$.
\item[\rm{(h)}] The image of $\overline{{}^*\Sigma^{\circ}_{d-r+3}(X)}$ under the Pl\"ucker embedding
                $\psi:\mathbb{G}(1,\mathbb{P}^r) \rightarrow \mathbb{P}^{\binom{r+1}{2}-1}$ is a plane.
\end{itemize}
\end{theorem}

\begin{proof}
In the proof, we prefer to use local cohomology instead of sheaf cohomology. So keep in mind that
$$H^i(\mathbb{P}^r,\mathcal{J}_X(n)) = H^i(S/I)_n \mbox { and } H^i(\mathbb{P}^r,\mathcal{J}_Y(n)) = H^i(S/I\cap L)_n \mbox{ for } i = 1,2,3 \mbox{ and } n \in \mathbb{Z}.$$

\noindent (a): First let $h \in \mathbb{U}(X)$. Then $\mathbb{L}_h
\subset \mathbb{F}$ and $\#(C\cap \mathbb{L}_h) = \#(X \cap
\mathbb{L}_h)
 = \#(C_h \cap \mathbb{L}_h) = d-r+3$. This shows, that $C \subset \mathbb{F}$ is a closed subscheme of dimension $1$ and degree $d-r+3$. Now, let $\mathbb{L} \subset
\mathbb{F}$ be an arbitrary line which is not contained in $X$. As
$C \subset \mathbb{F}$ is of dimension $1$ and of degree $d-r+3$ we
have $\#(X \cap \mathbb{L}) = \#(C \cap \mathbb{L}) \geq d-r+3$. As
$\reg(X) = d-r+3$ (see Theorem~\ref{3.2' Theorem} (b)) we also have
$\#(X \cap \mathbb{L}) \leq d-r+3$, so that indeed $\#(X \cap
\mathbb{L}) = d-r+3$. Now, it follows by Lemma~\ref{4.12'' Lemma}
that $C$ has no closed associated point.
\smallskip

\noindent (b): Let $p \in X \setminus {\rm Reg}(X)$. We first show,
that $p \in C$. Assume that this is not the case so that $p \notin
\mathbb{F}$. Now, as seen previously, a general hyperplane
$\mathbb{H} = \mathbb{P}^{r-1} \subset \mathbb{P}^r$ which runs
through $p$ has the property that $(X \cap \mathbb{H})_{\rm red}
\subset \mathbb{H}$ is a non-degenerate irreducible curve. Set
$\mathbb{L} := \mathbb{H} \cap \mathbb{F}$. Then $\#(X \cap \langle
p, \mathbb{L} \rangle) < \infty$. In particular the line $\mathbb{L}
\subset \mathbb{F}$ is not contained in $X$, so that $\#(X \cap
\mathbb{L}) = d-r+3$ by statement (a). As $\mathbb{H}$ is general,
the line $\mathbb{L}$ also avoids the finite set $X\setminus {\rm
Reg}(X)$. It follows by Proposition~\ref{4.3'' Lemma} (a) that there
is some $h \in \mathbb{U}(X)$ such that $\mathbb{L} = \mathbb{L}_h$,
and statement (c) of that same proposition yields the contradiction
that $\#(X \cap \langle p, \mathbb{L}\rangle) = \infty$. Therefore
we have indeed $p \in C$.

Now, let $\mathbb{L} \subset \mathbb{F}$ be a general line which
runs through $p$. As $\mathbb{L}$ avoids all singular points of $X$
different from $p$ it follows by Theorem~\ref{3.2' Theorem} (a)(5)
that $\#(X \cap \mathbb{L} \setminus {p}) + 2 \leq d-r+3$. So, by
statement (a) we get
$$d-r+3 - \mult_p(C \cap \mathbb{L}) = \#(C \cap \mathbb{L} \setminus {p}) \leq \#(X \cap \mathbb{L} \setminus {p}) \leq d-r+1$$
and hence $\mult_p(C \cap \mathbb{L}) \geq 2$. This shows, that $p$
is a singular point of $C$.
\smallskip

\noindent (c): According to statement (a), there is a homogeneous
polynomial $g \in S_{d-r+3} \setminus L$ such that the homogeneous
vanishing ideal $(I + L)^{\sat} \subset S$ of $C$ in $S$ can be
written as $(L,g)$. In particular we have $I_{\leq d-r+2} \subset
L$. As $\reg(X) = d-r+3$, the ideal $I \subset S$ is generated by
homogeneous polynomials of degree $\leq d-r+3$. As $g \notin L$ it
follows that $I_{d-r+3}$ is not contained in $L$ and hence that $I+L
= (L,f)$ for all $f \in I_{d-r+3} \setminus I \cap L$. Therefore $ I
= I\cap(I+L) = I\cap (L,f) = (I\cap L,f)$ for all such $f$.
\smallskip

\noindent (d): According to statement (c) we have an exact sequence
$$0 \rightarrow (S/L)(-d+r-3) \rightarrow S/I\cap L \rightarrow S/I \rightarrow 0.$$
Applying cohomology and observing that $H^i((S/L)(-d+r-3))_n$ vanishes if either $i \neq 3$ or else if $i=3$ and $n \geq d-r+1$, and keeping in mind
Theorem~\ref{3.2' Theorem} (a)(2), we get claims (1) and (4) and also
that
$$H^2(S/I)_n = H^2(S/I \cap L)_n = 0 \mbox{ for all } n > d-r.$$
It remains to show that $H^2(S/I)_{d-r} \cong K$ and $H^2(S/I\cap L)_{d-r} = 0$. To this end, let $h \in \mathbb{U}(X)$ and consider the induced exact
sequence
$$H^1(S/I) \longrightarrow H^1(S/(I,h)^{\rm sat}) \longrightarrow H^2(S/I)(-1) \stackrel{h}{\longrightarrow} H^2(S/I) \longrightarrow H^2(S/(I,h)^{\rm sat}).$$
As $S/(I,h)^{\rm sat}$ is the homogeneous coordinate ring of the
curve of maximal regularity $C_h \subset \Proj(S/hS) =
\mathbb{P}^{r-1}$ we get from Proposition 2.7 of \cite{BS2}, that
$H^1(S/(I,h)^{\rm sat})_{d-r+1} \cong K$ and $H^2(S/(I,h)^{\rm
sat})_n = 0$ for all $n \geq 0$.

As $H^2(S/I)_{d-r+1} = 0$, it follows, that $h^2(S/I)_{d-r} =
\dim(H^2(S/I)_{d-r}) \leq 1$. As $H^2(S/L) = 0$ the first exact
sequence used above induces exact sequences
$$ 0 \rightarrow H^2(S/I \cap L)_n \longrightarrow H^2(S/I))_n \longrightarrow H^3(S/L)_{n-d+r-3} \longrightarrow H^3(S/I\cap L)_n$$
for all $n \in \mathbb{Z}$. In view of claim (4) we thus get exact sequences
$$0 \rightarrow H^2(S/I \cap L)_n \longrightarrow H^2(S/I))_n \longrightarrow H^3(S/L)_{n-d+r-3} \longrightarrow 0 \mbox{ for all } n \geq 0.$$
We apply this with $n = d-r$. As $H^3(S/L)_{-3} \cong H^2(\mathbb{P}^2, \mathcal{O}_{\mathbb{P}^2}(-3)) \cong K$ it follows that
$\dim(H^2(S/I)_{d-r}) = \dim(H^2(S/I\cap L)_{d-r}) + 1$. In view of our previous observation, this shows that
$$H^2(S/I)_{d-r} \cong K \mbox{  and } H^2(S/I \cap L)_{d-r} = 0.$$
\smallskip

\noindent (e): If we apply the last exact sequence used in the proof
of statement (d) for all $n \geq 0$ and bear in mind that
$h^3(S/L)_{n-d+r-3} = h^2(\mathbb{P}^2,
\mathcal{O}_{\mathbb{P}^2}(n-d+r-3)) = \binom{-n-d-r+2}{2}$ we get
claim (1).

Claim (2) follows immediately from claim (1), as $h^2(S/I)_0 =
h^1(X,\mathcal{O}_X) = {\rm e}(X)$ (see Theorem~\ref{3.2' Theorem}
(a)(2)).

Now, let $h \in \mathbb{U}(X)$ be general. Then $C_h = X \cap
\mathbb{H}_h \subset \mathbb{H}_h = \mathbb{P}^{r-1}$ is a curve of
degree $d$ and maximal regularity, with extremal secant line
$\mathbb{L}_h:= \mathbb{F} \cap \mathbb{H}_h$. Moreover the ring
$S/(I\cap L ,h)^{\rm sat}$ is the homogeneous coordinate ring of
$C_h \cup \mathbb{L}_h$ in $S$. So, by Proposition 2.7 c),d) of
\cite{BS2} we have $H^1(S/(I \cap L, h)^{\rm sat})_{n+1} = 0$ for
all $n \leq 0$ and -- in addition -- that $H^2(S/(I \cap L, h)^{\rm
sat})_{n+1} = 0$ for all $n \geq 0$. Now, the induced exact sequence
$$H^1(S/(I \cap L, h)^{\rm sat}) \longrightarrow H^2(S/I\cap L)(-1) \stackrel{h}{\longrightarrow} H^2(S/I\cap L) \longrightarrow H^2(S/(I \cap L, h)^{\rm sat})$$
proves claim (3) and shows that the map $H^2(S/I\cap L)(-1)
\stackrel{h}{\longrightarrow} H^2(S/I\cap L)$ is an epimorphism in
all positive degrees.

Finally, by Remark 3.2 B) of \cite{BS2}, the graded $S$-module
$H^1(S/(I\cap L,h)^{\rm sat})$ is generated by homogeneous elements
of degree $2$. Consequently the same holds for the kernel of the map
$H^2(S/I\cap L)(-1) \stackrel{h}{\longrightarrow} H^2(S/I\cap L)$ in
the above sequence. As this map is an epimorphism in all positive
degrees, we get claim (4). Claim (5) is an immediate consequence of
claim (4).
\smallskip

\noindent (f): We keep the above notations. It follows easily by
Theorem~\ref{3.2' Theorem} (c) and the claims proved in statement
(d), that
$$\tau(X) \in \{(1,1),(2,2),(2,3)\}.$$
This proves in particular claim (2).

Now, let $d \leq 2r-4 = 2(r-1) -2$. Then, by Proposition 3.5 of
\cite{BS2} it follows that the $2$-dimensional ring $S/(I \cap
L,h)^{\rm sat}$ is Cohen-Macaulay, so that $H^1(S/(I \cap L, h)^{\rm
sat}) = 0$. Thus, the above exact sequence shows that the map
$H^2(S/I\cap L)(-1) \stackrel{h}{\longrightarrow} H^2(S/I\cap L)$ is
injective. The short exact sequence
$$H^1(S/I\cap L)(-1) \stackrel{h}{\longrightarrow} H^1(S/I\cap L) \longrightarrow H^1(S/(I \cap L, h)^{\rm sat})$$
also shows, that the map $H^1(S/I\cap L)(-1)
\stackrel{h}{\longrightarrow} H^1(S/I\cap L)$ is surjective. It
follows, that $H^1(S/I\cap L) = H^2(S/I\cap L) = 0$, so that
$\depth(S/I\cap L) = 3$. In view of the above observation it follows
that $\tau(X) = (2,3)$, and this proves claim (1).

Now, let $d \geq 3r - 6$. It follows by Lemma~\ref{4.13'' Lemma}
that $\depth(S/(I\cap L, h)^{\rm sat}) = 1$. Consequently, we must
have $\depth(S/I\cap L) \leq 2$ and our previous observation gives
$\tau(X) \in \{(1,1),(2,2)\}$. This proves claim (3).
\smallskip

\noindent (g): This is immediate by claims (1),(3) and (4) of
statement (d).
\smallskip

\noindent (h): By statement (a) we have
$\overline{{}^*\Sigma^{\circ}_{d-r+3}(X)} = \mathbb{G}(1,\mathbb{F})
= \mathbb{G}(1,\mathbb{P}^2)$. Standard arguments show that
$\psi\big(\mathbb{G}(1,\mathbb{P}^2)\big)$ is a plane in
$\mathbb{P}^{\binom{r+1}{2}-1}$. This proves our claim.
\end{proof}

\begin{corollary}\label{4.15'' Corollary} Let the notations and hypotheses be as be as in Theorem~\ref{4.14'' Theorem}. Then it holds
\begin{itemize}
\item[\rm{(a)}] ${\rm e}(X) = 0$ or else ${\rm e}(X) \geq \binom{d-r+2}{2}$.
\item[\rm{(b)}] ${\rm e}(X) \geq \binom{d-r+2}{2} $ if and only if $\mathbb{F}(X) = \mathbb{P}^2$.
\end{itemize}
\end{corollary}

\begin{proof}
If $\mathbb{F}(X)$ is a plane, it follows by Theorem~\ref{4.14''
Theorem} (e)(2) that ${\rm e}(X) \geq \binom{d-r+2}{2}$. If
$\mathbb{F}(X)$ is not a plane it follows by Theorem~\ref{4.11''
Theorem} that $X$ is smooth and hence satisfies ${\rm e}(X) = 0$.
This proves our claim.
\end{proof}

Our next main result is devoted to the relations between the index of normality $N(X)$, the Betti numbers $\beta_{i,j}(X)$ and the nature of the union
$X \cup \mathbb{F}(X)$, where $X$ is a surface of maximal sectional regularity with planar extremal variety $\mathbb{F}(X)$. We begin with two auxiliary results.

\begin{lemma}\label{4.16'' Lemma} Let $5 \leq r < d$, let $X \subset \mathbb{P}^r$ be a surface of degree $d$ and of maximal sectional regularity which is not a cone and
assume that $\mathbb{F}(X)$ is a plane. Let $Y := X \cup \mathbb{F}(X)$ and set $m:= \reg(Y)$. Then the following statements hold:
\begin{itemize}
\item[\rm{(a)}] For all $i \geq 1$ we have
$$\beta_{i,j}(X) = \begin{cases}
                              \beta_{i,j}(Y)                           & \mbox{for $1 \leq j \leq m-1$},\\
                              \beta_{i,j}(Y) = 0                       & \mbox{for  $m \leq j \leq d-r+1$},\\
                              \beta_{i,d-r+2}(Y) + \binom{r-2}{i-1} & \mbox{for $j=d-r+2$}.
\end{cases}$$
\item[\rm{(b)}] $m \leq d-r+2$ if and only if $\beta_{i,j}(X) = \binom{r-2}{i-1}$ for all $i \geq 1$.
\end{itemize}
\end{lemma}

\begin{proof}
Let $I$ and $L$ respectively denote the homogeneous vanishing ideals
of $X$ and $\mathbb{F}(X)$ in $S$, so that $\beta_{i,j}(X) =
\beta_{i,j}(S/I)$ and $\beta_{i,j}(Y) = \beta_{i,j}(S/I\cap L)$ for
all $i,j \in \mathbb{N}$.

\noindent (a): By statement (c) of Theorem~\ref{4.14'' Theorem} we
have an exact sequence
$$0 \rightarrow (S/L)(-d+r-3) \rightarrow S/I\cap L \rightarrow S/I \rightarrow 0,$$
which induces exact sequences
$$\Tor^S_i(K,S/L)_{i+j-d+r-3} \rightarrow \Tor^S_i(K,S/I \cap L)_{i+j} \rightarrow \Tor^S_i(K,S/I)_{i+j}$$
$$\rightarrow\Tor^S_{i-1}(K,S/L)_{(i-1)+j-d+r-2} \rightarrow \Tor^S_{i-1}(K,S/I\cap L)_{(i-1)+j+1}$$
After an appropriate change of coordinates in $S$ we may assume that $L = \langle x_3,\ldots,x_r \rangle$, and this shows that for all $k \in \mathbb{N}_0$ we have
$$\dim_K\big(\Tor^S_k(K,S/L)_{k+l}\big) = \beta_{k,l}(S/L) = \begin{cases}
                                                     0 & \mbox{if $l \neq 0$}\\
                                                     \binom{r-2}{k} & \mbox{if $l = 0$}
\end{cases}$$
Therefore, the above exact sequences make us end up with isomorphisms
$$\Tor^S_i(K,S/I \cap L)_{i+j} \cong \Tor^S_i(K,S/I)_{i+j} \mbox{ for all } i \geq 1 \mbox{ and all } j \in \{1,2,\ldots, d-r+1\}.$$
As $\reg(S/I\cap L) = \reg(Y) - 1 = m-1$, we have
$\beta_{i,j}(S/I\cap L) = 0 \mbox{ for all } i \geq 1 \mbox{ and all
} j \geq m$. So by the above isomorphisms we get the requested
values of $\beta_{i,j}(S/I)$ for all $i \geq 1$ and all $j
\in\{1,\ldots,d-r+1\}$.

As $\reg(S/I\cap L) = \reg(Y) -1 \leq d-r+2$ (see
Theorem~\ref{4.14'' Theorem} (g)), the last module in the above
exact sequences vanishes for $j=d-r+2$. So, our previous observation
on the Betti numbers $\beta_{k,l}(S/L)$ yields a short exact
sequence
$$0\rightarrow \Tor^S_i(K,S/I\cap L)_{i+d-r+2} \rightarrow \Tor^S_i(K,S/I)_{i+d-r+2} \rightarrow K^{\binom{r-2}{i-1}} \rightarrow 0 \mbox{ for all } i \geq 1,$$
which shows that $\beta_{i,d-r+2}(S/I) = \beta_{i,d-r+2}(S/I \cap L)
+ \binom{r-2}{i-1}$, and this proves our claim.
\smallskip

\noindent (b): As already said above, we have $\reg(S/I\cap L) =
\reg(Y) -1 \leq d-r+2$, whence $\beta_{i,j}(Y) = 0$ for all $i \geq
1$ and all $j \geq d-r+3$. From this we conclude by statement (a).
\end{proof}

\begin{notation and reminder}\label{4.16''' Notation and Reminder}
Let $T = \bigoplus _{n \in \mathbb{Z}} T_n$ be a graded $S$-module. Then, we denote the
\textit{socle} of $T$ by ${\rm Soc}(T)$, thus:
$${\rm Soc}(T) := (0:_T S_+) \cong \Hom_S(K,T) = \Hom_S(S/S_+,T).$$
Keep in mind that the socle of a graded Artinian $S$-module $T$ is a $K$-vector space of finite dimension which vanishes if and only if $T$ does.
\end{notation and reminder}

\begin{lemma}\label{4.16'''' Lemma}
Let the notations and hypotheses be as in Lemma~\ref{4.16'' Lemma} and let $I \subset S$ denote the homogeneous vanishing ideal of $X$.
Then we have the following statements
\begin{itemize}
\item[\rm{(a)}] ${\rm Soc}(H^1(S/I))(-r-1) \cong \Tor^S_r(K,S/I)$.
\item[\rm{(b)}] If $\depth(X) = 1$, then $H^1\big(\mathbb{P}^r,\mathcal{J}_X(N(X))\big) \cong \Tor^S_r(K,S/I)_{N(X)+r+1}.$
\item[\rm{(c)}] $N(X) \leq d-r$ if and only if $\beta_{r,d-r+2}(X) = 0$.
\end{itemize}
\end{lemma}

\begin{proof}
(a): If $\depth(X) > 1$, both of the occurring modules vanish and our claim is obvious. So, we assume that $\depth(X) = 1$ and consider the total ring of
sections $D := D_{S_+}(S/I) = \bigoplus_{n \in \mathbb{Z}} H^0(\mathbb{P}^r, \mathcal{O}_X(n))$ of $X$, as well as the short exact sequence
$$0 \longrightarrow S/I \longrightarrow D \longrightarrow H^1(S/I) \longrightarrow 0.$$
We apply the Koszul functor $K(\underline{x}; \bullet)$ with respect to the sequence $\underline{x} := x_0,x_1,\ldots,x_r$ to this sequence and end up in homology with
an exact sequence
$$H_{r+1}(\underline{x};D) \rightarrow H_{r+1}(\underline{x};H^1(S/I)) \rightarrow H_r(\underline{x}; S/I) \rightarrow H_r(\underline{x}; D).$$
As $\depth(D) > 1$ the first and the last module in this sequence vanish, so that
$$H_{r+1}(\underline{x};H^1(S/I)) \cong H_r(\underline{x}; S/I).$$
As the Koszul complex $K(\underline{x},S)$ provides a free
resolution of $K = S/S_+$ and $K(\underline{x}; S/I) \cong
K(\underline{x}; S) \otimes_S S/I$ we have $H_r(\underline{x}; S/I)
\cong \Tor^S_r(K,S/I)$. As the sequence $\underline{x}$ has length
$r+1$, we have $H_{r+1}(\underline{x}; H^1(S/I)) \cong {\rm
Soc}(H^1(S/I))(-r-1)$. Altogether, we now obtain statement (a).
\smallskip

\noindent (b): As $N(X) = {\rm end}(H^1(S/I))$, we have
$$H^1\big(\mathbb{P}^r,\mathcal{J}_X(N(X))\big) \cong H^1(S/I)_{N(X)} = {\rm Soc}(H^1(S/I))_{N(X)}.$$
Now, our claim follows immediately by statement (a).
\smallskip

\noindent (c): If $\depth(X) > 1$ we have $N(X) = -\infty$ and
$\beta_{r,d-r+2}(X) = 0$, so that our claim is true. We thus may
assume that $\depth(X) = 1$. As $\reg(X) = d-r+3$ we have $N(X) \leq
d-r+1$ and $\Tor^S_r(K,S/I)_{r + l} = 0$ for all $l \geq d-r+3$.
Now, we may conclude by statement (b).
\end{proof}

Now, we are ready to give the announced main result.

\begin{theorem}\label{4.17'' Proposition}
Let $5 \leq r < d$ and assume that the surface $X \subset \mathbb{P}$ is non-conic, has degree $d$ and is of maximal sectional
regularity.
\begin{itemize}
\item[\rm{(a)}] The following statements are equivalent:
                \begin{itemize}
                \item[\rm{(i)}]   $N(X) \leq d-r$.
                \item[\rm{(ii)}]  $\mathbb{F}(X) = \mathbb{P}^2$ and $\reg(X \cup \mathbb{F}(X)) \leq d-r+2$.
                \item[\rm{(iii)}] $\beta_{i,d-r+2}(X) = \binom{r-2}{i-1}$ for all $i \geq 1$.
                \item[\rm{(iv)}]  $\mathbb{F}(X) = \mathbb{P}^2$ and $\beta_{r,d-r+2}(X) = 0$.
                \end{itemize}
\item[\rm{(b)}] The following statements are equivalent:
               \begin{itemize}
               \item[\rm{(i)}]  $\beta_{1,d-r+2}(X) = 1$.
               \item[\rm{(ii)}] $\mathbb{F}(X) = \mathbb{P}^2$ and $I\cap L = (I_{\leq d-r+2})$, where $I$ and $L$ are the homogeneous vanishing ideals of $X$
                                respectively of $\mathbb{F}(X)$ in $S$.
               \end{itemize}
                Moreover, if the equivalent statements (b) (i) and (ii) hold, then we have (in the notations introduced in Notation and Reminder~
                \ref{4.1'' Notation and Reminder} and Definition and Remark~\ref{4.2' Definition and Remark}):
               \begin{itemize}
               \item[\rm{(1)}] If $\mathbb{L} \in \Sigma^{\circ}_{d-r+3}(X)$, then $\mathbb{L} \subset \mathbb{F}(X)$.
               \item[\rm{(2)}] $\Sec_{d-r+3}(X) = X \cup \mathbb{F}(X)$.
               \item[\rm{(3)}] $\mathbb{F}^+(X) = \mathbb{F}(X)$.
               \end{itemize}
\item[\rm{(c)}] The equivalent conditions (i) -- (iv) of statement (a) imply the equivalent conditions (i) and (ii) (and hence also the claims (1), (2) and (3))
                of statement (b).
\end{itemize}
\end{theorem}

\begin{proof}
(a): (i) $\Rightarrow$ (ii): Let $N(X) \leq d-r$. It follows by
Theorem~\ref{4.11'' Theorem} (c), that $\mathbb{F}(X) =
\mathbb{P}^2$. Let $I$ and $L \subset S$ respectively denote the
homogeneous vanishing ideals of $X$ and $\mathbb{F}(X)$. According
to Theorem~\ref{4.14'' Theorem} (d)(1) we have ${\rm
end}(H^1(S/I\cap L) = {\rm end}(H^1(S/I) = N(X) \leq d-r$. So, it
follows by Theorem~\ref{4.14'' Theorem} (d)(3),(4) that $\reg(S/I
\cap L) \leq d-r+1$, whence $\reg(X) \cup \mathbb{F}(X) = \reg(I
\cap L) \leq d-r+2$.

(ii) $\Rightarrow$ (i): As ${\rm end}(H^1(S/I)) = N(X)$, this is an easy consequence of Theorem~\ref{4.14'' Theorem} (d)(1),(3) and (4).

(ii) $\Rightarrow$ (iii): This follows by Lemma~\ref{4.16'' Lemma}.

(iii) $\Rightarrow$ (ii): Assume that statement (iii) holds. Then we have in particular that $\beta_{1,d-r+2}(X) = 1$. By Theorem~\ref{4.11'' Theorem} (d) it follows that
$\mathbb{F}(X) = \mathbb{P}^2$. Now, we may again conclude by Lemma~\ref{4.16'' Lemma}.

(iii) $\Leftrightarrow$ (iv): This is clear by Lemma~\ref{4.16''''
Lemma}.
\smallskip

\noindent (b): (i) $\Rightarrow $ (ii): Assume that
$\beta_{1,d-r+2}(X) = 1$. Then Theorem~\ref{4.11'' Theorem} (d)
implies that $\mathbb{F}(X) = \mathbb{P}^2$. If follows by
Theorem~\ref{4.14'' Theorem} (c) that $(I_{\leq d-r+2}) = I \cap L$.

(ii)$\Rightarrow$(i): This follows immediately by
Theorem~\ref{4.14'' Theorem} (c).

Assume now, that the equivalent statements (i) and (ii) hold, let
$\mathbb{L} \in \Sigma_{d-r+3}(X)$ be not contained in $X$ and let
$M \subset S$ be the homogeneous vanishing ideal of $\mathbb{L}$.
Then, $(I_{d-r+2}) \subset M$. As $I\cap L = (I_{\leq d-r+2})$ it
follows that $I \cap L \subset M$. As $\mathbb{L}$ is not contained
in $X$, the ideal $I$ is not contained in $M$. It follows that $L
\subset M$, and hence that $\mathbb{L} \subset \mathbb{F}(X)$. This
proves claim (1). Claim (2) is immediate by claim (1), as $X$ is a
union of lines and each line $\mathbb{L} \subset \mathbb{P}^r$ with
$\#(X \cap \mathbb{L}) > d-r+3$ is contained in $X$. Now claim (3)
is obvious, too.
\smallskip

\noindent (c): This follows from the fact that statement (a)(iii)
implies that $\beta_{1,d-r+2}(X) = 1$.
\end{proof}

We have seen above, that surfaces $X$ of maximal sectional regularity and sub-maximal index of normality $N(X) < d-r+1$ (see Notation and
Reminder~\ref{4.16'' Notation and Reminder}) have a planar extremal variety and show an interesting behavior of Betti numbers. We therefore can expect, that in the
extremal case $N(X) = -\infty$ -- hence in the case $\depth(X) = 2$ -- we get further detailed information on the Betti numbers if $X$ is of ``small degree``.
Our next main result is devoted to this case.

\begin{theorem}\label{3.4' Theorem}
Let $5 \leq r < d$, assume that the surface $X\subset \mathbb{P}^r$ of degree $d$ is of maximal sectional
regularity and satisfies $\depth(X) = 2$. Let $I \subset S$ be the homogeneous vanishing ideal of $X$ and let $J := (I_{\leq d-r+2}) \subset S$ be the ideal generated by
all polynomials of degree $\leq d-r+2$ in $I$.
\begin{itemize}
\item[\rm{(a)}] $\mathbb{F}(X) = \mathbb{P}^2$ and $J = I \cap L$, where $L \subset S$ is the homogeneous vanishing ideal of $\mathbb{F}(X)$ in $S$.
\item[\rm{(b)}] If the ring $S/J$ is Cohen-Macaulay -- hence if $\tau(X) = (2,3)$ -- we have $\reg(S/J) = \reg(X \cup \mathbb{F}(X)) -1 = 2$ and
$$h^2(\mathbb{P}^r,\mathcal{J}_X(n)) = \begin{cases} & \e(X) = \binom{d-r+2}{2} \mbox{ for all } n\leq 0, \\
                             & \binom{d-r-n+2}{2} \mbox{ for } 0 < n\leq d-r, \\
                             & 0 \mbox{ for } n > d-r.
               \end{cases}$$
\item[\rm{(c)}] Let $d \leq 2r-5$. Then, setting
        \[
        a_i := (d-r+1)\binom{r-1}{i}+\binom{r-2}{i-1}, \quad
        c_i := (d-1) \binom{r-2}{i}-\binom{r-2}{i+1},
        \]
        for all $i\in\{1,\ldots,r\}$, we have
        \[
        \Tor^S_i(K,S/I) = K^{u_i}(-i-1)\oplus K^{v_i}(-i-2)\oplus K^{\binom{r-2}{i-1}}(-i-d-r-2)
        \]
        with
        \begin{align*} u_1 & = \binom{r}{2}-d-1, \\
                                       u_i & = c_i-a_i \quad \quad \quad \mbox{ for } 2\leq i \leq 2r-d-3, \\
                                               u_i & \leq c_i \quad \quad \quad \quad \quad \mbox{ for } 2r-d-2 \leq i \leq r-1.
                        \end{align*}
                and
        \begin{align*} v_i & = 0 \quad \quad \quad \quad \quad \quad \quad \quad \mbox{ for } 1\leq i \leq 2r-d-4 \mbox{ and } i = r,\\
                               v_i & = u_{i+1}+a_{i+1}-c_{i+1} \quad \mbox{ for } 2r-d-3\leq i \leq r-3,\\
                                               v_{r-2} & = d-r.
                                \end{align*}

\end{itemize}
\end{theorem}

\begin{proof}
(a): This is clear by Theorem~\ref{4.17'' Proposition}.
\smallskip

\noindent (b): Let $S/J$ be Cohen Macaulay. Then, according to
statement (a) $S/I \cap L$ is Cohen Macaulay. Now, let $h \in
\mathbb{U}(X)$. As the ring $S/I \cap L$ is Cohen Macaulay and $h$
is a non-zero divisor with respect to $S/I \cap L$, it follows that
the homogeneous coordinate ring $S/\langle I\cap L, h\rangle$ of $(X
\cup \mathbb{F}(X)) \cap \mathbb{H}_h = C_h \cup \mathbb{L}_h$ is
Cohen-Macaulay. So, by \cite[Theorem 3.3]{BP2} we have
$\reg(S/\langle J, h \rangle) = \reg(S/\langle I\cap L, h\rangle) =
2$ and the fact that $h$ is a non-zero divisor with respect to $S/I
\cap L = S/J$ implies that $\reg(S/J) = 2$. By our hypothesis we
also have $H^2_{*}(\mathbb{P}^r,\mathcal{J}_{X \cup \mathbb{F}(X)})
= H^2(S/I \cap L) = 0$. So, by Theorem~\ref{4.14'' Theorem} (e) (1)
and Theorem~\ref{3.4' Theorem} (b)(2) the values of
$h^2(\mathbb{P}^r,\mathcal{J}_X(n))$ are as stated above.
\smallskip

\noindent (c): Let $d \leq 2r-5$ and let $h \in \mathbb{U}(X)$.  As
$\depth(X) > 1$, the ring $S/\langle I,h \rangle$ is the homogeneous
coordinate ring of the curve $C_h$ in $S$. As $h$ is a non-zero
divisor with respect to $S/I$,  we have $\beta^S_{i,j}(S/I) =
\beta_{i,j}(S/\langle I, h \rangle$ for all $i, j\geq 1$. Now our
claim follows immediately by the approximation of the Betti numbers
$\beta_{i,j}(S/\langle I, h \rangle) = \beta_{i,j}(C_h)$ of the
homogeneous the curve of maximal regularity $C_h \subset
\mathbb{H}_h$ given in \cite[Theorem 1.2]{BS5}.
\end{proof}

We now briefly revisit the special case of surfaces $X \subset \mathbb{P}^r$ of degree $r+1$.

\begin{remark} \label{3.3' Remark} (s. \cite{B2}, \cite{BS6})
(A) Assume that $r \geq 5$ and let our surface $X \subset \mathbb{P}^r$ be of degree $r+1$. Then,
we can distinguish $9$ cases, which show up by their numerical invariants
as presented in the following table. Here $\sigma(X)$ denotes the \textit{sectional
genus} of $X$, that is the arithmetic genus of the generic hyperplane section curve $C_h \quad
(h \in \mathbb{U}(X))$ or equivalently, the sectional genus of the polarized surface
$(X, \mathcal{O}_X(1))$ in the sense of Fujita \cite{Fu}.

\[
\begin{tabular}{| c | c | c | c | c | c | c |}
    \hline
    $\rm{Case}$ & $\sreg(X)$ & $\depth(X)$
    & $\sigma(X)$ & $\e(X)$ & $h^1_A(1)$ & $h^1_A(2)$  \\ \hline
    $1$ &$2$ &$3$ &$2$ &$0$ &$0$ &$0$ \\ \hline
    $2$ &$3$ &$2$ &$1$ &$0$ &$0$ &$0$ \\ \hline
    $3$ &$3$ &$2$ &$1$ &$1$ &$0$ &$0$ \\ \hline
    $4$ &$3$ &$1$ &$1$ &$0$ &$1$ &$0$ \\ \hline
    $5$ &$3$ &$2$ &$0$ &$2$ &$0$ &$0$ \\ \hline
    $6$ &$3$ &$1$ &$0$ &$1$ &$1$ &$\leq 1$ \\ \hline
    $7$ &$3$ &$1$ &$0$ &$0$ &$2$ &$\leq 2$ \\ \hline
    $8$ &$4$ &$2$ &$0$ &$3$ &$0$ &$0$ \\ \hline
    $9$ &$4$ &$1$ &$0$ &$0$ &$2$ &$3$  \\ \hline
    \end{tabular}
\]
\\
The case $9$ occurs only if $r = 5$. In \cite{B2} and \cite{BS6} we listed indeed two more cases $10$ and $11$, of which we did not know at that time, whether they might
occur at all. For these two cases we had $\sreg(X) = 4 = d-r+3$ and ${\rm e}(X) \in \{1,2\}$. As these surfaces would be of maximal sectional regularity, this would
contradict Corollary~\ref{4.15'' Corollary}. So, surfaces which fall under the cases $10$ and $11$ cannot occur at all. In the case $9$ we have ${\rm e}(X) = 0$, and hence
by Corollary~\ref{4.15'' Corollary} and Theorem~\ref{4.11'' Theorem} we must have $r = 5$ and $\mathbb{F}(X) = S(1,1,1)$ in this case. \\
(B) In view of Theorem~\ref{3.2' Theorem}, the surfaces of types 8 and 9 are of particular
interest, as they are the ones of maximal sectional regularity within all the 9 listed types.
Observe, that among all surfaces $X$ of degree $r+1$ in $\mathbb{P}^r$, those of type
8 are precisely the ones $X$ which are of maximal sectional regularity and of arithmetic depth
$\geq 2$. If $r \geq 6$, the surfaces of type 8 are precisely the ones which are of maximal sectional regularity.\\
(C) Observe, that in the cases 5 -- 9 we have $\sigma(X) = 0$. This means, that the surfaces which fall
under these 5 types are all sectionally rational and have finite non-normal locus. So, by Theorem~\ref{2.3' Theorem}, these surfaces are almost
non-singular projections of a rational normal surface scroll $\widetilde{X} = S(a,r+1-a)$ with $0 \leq a \leq \frac{r+1}{2}$, even if they are cones
(see \cite[Corollary 5.11]{BS5} for the non-conic case). So, according to Theorem~\ref{2.8' Proposition} the surfaces $X$ of types 5 -- 9 all satisfy
the Eisenbud-Goto inequality $\reg(X) \leq 4$, with equality in the cases 8 and 9 (see Theorem~\ref{3.2' Theorem} (b)). In the cases 1 -- 5, the values of
$h^i(\mathbb{P}^r, \mathcal{J}_X(n)) =: h^i(S/I)_n \quad (i=1,2, n\in \mathbb{Z})$ (see \cite[Reminder 2.2 (C),(D)]{BS6}) show, that $\reg(X) = 3$.
In the case 6 we may have $\reg(X) = 3$ whereas in the case 7, we know even that $\reg(X)$ my take both values $3$ and $4$
(see \cite[Example 3.5, Examples 3.4 (A),(B), (C)]{BS6}). This shows in particular, that there are sectionally rational surfaces $X \subset \mathbb{P}^r$
of degree $r+1$ with finite non-normal locus and $\sreg(X) < \reg(X)$.\\
\end{remark}

We now make explicit the Betti numbers of the surfaces $X \subset \mathbb{P}^r$ of degree $r+1$ which are of maximal sectional regularity and of arithmetic depth $1$,
thus of the surfaces which fall under the type 8 of Remark~\ref{3.3' Remark}.

\begin{corollary} \label{3.5' Corollary}
Assume that the surface $X \subset \mathbb{P}^r_K$ is of degree $r+1$. Then
\begin{itemize}
        \item[{(a)}] The following conditions are equivalent:
                \begin{itemize}
                                                \item[\rm{(i)}] The surface $X$ is of type 8.
                        \item[\rm{(ii)}] $\e(X) = 3$.
                        \item[\rm{(iii)}] $\sreg(X) = 4$ and $\depth(X) =2$.
                        \item[\rm{(iv)}] $\sreg(X) = 4$ and $X$ does not fall under the case 9 of
                                                Remark~\ref{3.3' Remark}.
                \end{itemize}
        \item[{(b)}] If the above equivalent conditions hold, and in the notations of Theorem
                \ref{3.4' Theorem} for all $i\in\{1,\ldots,r\}$ we have
                \[
                \Tor^S_i(K,S/I) = K^{u_i}(-i-1)\oplus K^{v_i}(-i-2)\oplus K^{\binom{r-2}{i-1}}(-i-3)
                \]
                with
                \begin{align*}
                        u_1 &=  \binom{r-1}{2}-3, \\
                                                u_i &=  (r-1)\binom{r-2}{i}-\binom{r-2}{i+1}-3\binom{r-2}{i-1}\mbox{ for } 2\leq i \leq r-4,\\
                                                u_{r-3} &\in \{0,r-2\} \mbox{ and } \quad u_i  = 0 \mbox{ for } i\geq r-2.
                                \end{align*}
                and
                                \begin{align*}
                        v_i & =  0 \mbox{ for } 1 \leq i \leq r-5\mbox {and } i\geq r-1,\\
                                                v_{r-4} & =  u_{r-3} + (r-1)\binom{r-2}{2}-3\binom{r-2}{3}-r+2,\\
                                                v_{r-3} & = 2r-4, \\
                                                v_{r-2} & = 3.
                \end{align*}
\end{itemize}
\end{corollary}

\begin{proof}
Statement (a) follows easily on use of the table in Remark~\ref{3.3'
Remark} (A).
\smallskip

\noindent (b): As $\depth(X) = 2$, the surface $X$ has the same
Betti numbers as its hyperplane section curve $C_h \subset
\mathbb{H}_h = \mathbb{P}^{r-1}$  ($h \in \mathbb{U}(X)$). As $X$ is
of maximal sectional regularity, we may apply \cite[Theorem
6.12]{BS1} to the curve $C_h$ in order to obtain information on the
Betti numbers of $X$. In particular we get that $u_{r-3} \in \{0,
r-3\}$. If $r = 5$ our claim follows easily from the mentioned
theorem of \cite{BS1}. If $r \geq 6$ we may conclude directly by
Theorem \ref{3.4' Theorem} (d).
\end{proof}

The last main result of this section characterizes non-conic surfaces of maximal sectional regularity in terms of projections of smooth rational surface scrolls which
are generically injective along appropriate effective divisors on these scrolls. We first prove an auxiliary result.

\begin{lemma}\label{5.x'' Lemma}
Let $d > 2, s > 1$, let $\widetilde{X} \subset \mathbb{P}^{d+1}$ be a smooth rational normal surface scroll and let $\mathbb{K} = \mathbb{P}^s \subset \mathbb{P}^{d+1}$
be such that $\widetilde{X} \cap \mathbb{K} \subset \mathbb{K}$ is a subscheme of dimension $1$ and degree $\geq s$. Then
$$\deg(\widetilde{X} \cap \mathbb{K}) = s \mbox{ and } \widetilde{X} \cap \mathbb{K} \in {\rm Div}(\widetilde{X}).$$
\end{lemma}

\begin{proof}
Assume, that $C := \widetilde{X}\cap \mathbb{K} \notin {\rm Div}(\widetilde{X})$. As $\widetilde{X}$ is smooth, this means that $C$ is not a Cartier
divisor so that the local vanishing ideal $\mathcal{J}_{C,\widetilde{x}} \subset \mathcal{O}_{\widetilde{X},\widetilde{x}}$ of $C$ is not principal for some closed point
$\widetilde{x} \in C$.  As $\mathcal{O}_{\widetilde{X},\widetilde{x}}$ is a local factorial domain of dimension $2$ and ${\rm height}(\mathcal{J}_{C,\widetilde{x}}) = 1$,
it follows that
$$\mathfrak{m}_{\widetilde{X},\widetilde{x}} \in {\rm Ass}_{\mathcal{O}_{\widetilde{X},\widetilde{x}}}\big(\mathcal{O}_{\widetilde{X},\widetilde{x}}
/\mathcal{J}_{C,\widetilde{x}}\big)$$
and hence that $\widetilde{x} \in {\rm Ass}(\mathcal{O}_{\widetilde{X}}/\mathcal{J}_C) = {\rm Ass}_C(\mathcal{O}_C)$. Now, there is a space $\mathbb{H} = \mathbb{P}^{s-1}
\subset \mathbb{K} = \mathbb{P}^s$ such that $\widetilde{x} \in \mathbb{H}$ and $\dim(\widetilde{X} \cap \mathbb{H}) = 0$. According to Lemma~\ref{4.12'' Lemma} we now get
$$\infty > \#(\widetilde{X} \cap \mathbb{H}) = \#(C \cap \mathbb{H}) > s.$$
As $\widetilde{X} \subset \mathbb{P}^{d+1}$ is a smooth rational
normal scroll and $\mathbb{H} \subset \mathbb{P}^{d+1}$ is an
$(s-1)$-space, this is a contradiction, and hence $C \in {\rm
Div}(\widetilde{X})$.

Assume now, that $\deg(C) > s$. Then for a general $(s-1)$-space
$\mathbb{H} \subset \mathbb{K}$ we have again the previous
inequalities and hence a contradiction. Therefore
$\deg(\widetilde{X} \cap \mathbb{K}) = \deg(C) = s$.
\end{proof}

\begin{theorem}\label{5.x^iv Theorem} Let $5 \leq r < d$ and assume that the surface $X \subset \mathbb{P}$ is non-conic, has degree $d$ and is of maximal sectional
regularity.
\begin{itemize}
\item[\rm{(a)}] The following statements are equivalent:
                \begin{itemize}
                \item[\rm{(i)}] $\mathbb{F}(X) = \mathbb{P}^2$.
                \item[\rm{(ii)}] $X = \widetilde{X}_{\Lambda}$, where $\widetilde{X} \subset \mathbb{P}^{d+1}$ is a smooth rational normal surface scroll and $\Lambda =
                                 \mathbb{P}^{d-r}$ is disjoint to $\widetilde{X}$ and contained in the linear span $\langle D \rangle = \mathbb{P}^{d-r+3} \subset
                                 \mathbb{P}^{d+1}$ of a divisor $D \in |H + (3-r)F|$ (where $H$ and $F$ are defined as in Notation and Reminder~
                                 \ref{4.x Notation and Reminder}) such that the restriction
                                 $$\pi'_{\Lambda} \upharpoonright :\langle D \rangle \setminus \Lambda \twoheadrightarrow \mathbb{P}^2$$
                                 of the projection map $\pi'_{\Lambda}: \mathbb{P}^{d+1} \setminus \Lambda \twoheadrightarrow \mathbb{P}^r$ is generically injective along $D$.
                 \end{itemize}
\item[\rm{(b)}] If the equivalent conditions (i) and (ii) of statement (a) hold, we have
                \begin{itemize}
                \item[\rm{(1)}] $\langle D \rangle = \mathbb{E} := \overline{(\pi'_{\Lambda})^{-1}\big(\mathbb{F}(X)\big)}$;
                \item[\rm{(2)}] $D =  (\pi_{\Lambda})^{-1}\big(X \cap \mathbb{F}(X)\big)= \widetilde{X} \cap \mathbb{E} = C + \sum_{j=1}^s m_j\mathbb{L}_j$,
                                where $C \subset \widetilde{X}$ is a curve section, $\mathbb{L}_1,\ldots,\mathbb{L}_s$ are pairwise different ruling lines of $\widetilde{X}$
                                and $m_1,\ldots,m_s$ are positive integers with $\sum_{j=1}^s m_j = d-r-\deg(C)+3$.
                \end{itemize}
\end{itemize}
\end{theorem}

\begin{proof}:
(a): (i) $\Rightarrow$ (ii): According to Theorem~\ref{3.4' Theorem}
(a) (3) we may write $X = \widetilde{X}_{\Lambda}$, where $S(a,d-a)
= \widetilde{X} \subset \mathbb{P}^{d+1}$ is a smooth rational
surface scroll and $\Lambda = \mathbb{P}^{d-r} \subset
\mathbb{P}^{d+1}$ is disjoint to $X$ and the induced projection
morphism $\pi_{\Lambda}: \widetilde{X} \twoheadrightarrow X$ is
almost non-singular.

Now, let $\mathbb{E} :=
\overline{(\pi'_{\Lambda})^{-1}(\mathbb{F}(X))} = \mathbb{P}^{d-r+3}
\subset \mathbb{P}^{d+1}$. Then $D: = \widetilde{X} \cap \mathbb{E}
= (\pi_{\Lambda})^{-1}(X \cap \mathbb{F}(X)) \subset \mathbb{E}$ is
a subscheme of dimension $1$ and degree $\geq d-r+3$. By
Lemma~\ref{5.x'' Lemma} it follows that $D \in {\rm
Div}(\widetilde{X})$ and $D \subset \mathbb{E}$ is of degree
$d-r+3$.

Now, we write
$$D = \sum_{i=1}^t n_iC_i + \sum_{j=1}^s m_j\mathbb{L}_j$$
with $s,t \in \mathbb{N}_0$, $n_1,\ldots,n_t ; m_1,\ldots,m_s \in \mathbb{N}$, with pairwise distinct prime divisors $C_i \in {\rm Div}(\widetilde{X})$ and pairwise distinct
ruling lines $\mathbb{L}_1,\ldots,\mathbb{L}_s$ of $\widetilde{X}$ such that
$$\sum_{i=1}^t n_i \deg(C_i) + \sum_{j=1}^s m_j = \deg(D) = d-r+3.$$
Now, $t = 0$ would imply that $\sum_{j=1}^s m_j = d-r+3$, and hence
that $d-r+3 < \dim(\langle \sum_{j=1}^s m_i\mathbb{L}_j \rangle) =
\dim(\langle D \rangle) \leq \dim(\mathbb{E}) = d-r+3$. This
contradiction shows that $t > 0$.

As $D = \widetilde{X} \cap \mathbb{E}$ and each of the  curves $C_i$
intersects all ruling lines $\mathbb{L}$ of $\widetilde{X}$, we must
have $t = 1$. So, writing $C := C_1$ and $n:=n_1$ we obtain $D = nC
+ \sum_{j=1}^s m_j \mathbb{L}_j$ with $n > 0$. As $\langle C \rangle
\subseteq \langle D \rangle$ is of dimension $\leq d$, the curve $C
\subset \widetilde{X}$ is a section of $\widetilde{X}$.

Moreover, we must have $n = 1$. Otherwise we would get ${\rm
mult}_C(\widetilde{X} \cap \mathbb{E}) = n > 1$, so that the tangent
plane ${\rm T}_p(\widetilde{X})$ of $\widetilde{X}$ at a general
point $p \in C$ would be a $2$-plane contained in $\mathbb{E}$. But
this would imply the contradiction, that a general ruling line
$\mathbb{L}$ of $\widetilde{X}$ is contained in $\mathbb{E} =
\mathbb{P}^{d-r+3}$. Therefore, indeed $n = 1$, and hence
$$D = C + \sum_{j=1}^s m_j \mathbb{L}_j \mbox{ with } \sum_{j=1}^s m_j  = d-r+3 - c, \mbox{ where }c:= \deg(C).$$
It follows, that $D$ is linearly equivalent to $C + (d-r+3+c)F$ --
with $C\cdot F = 1$ -- and hence that $D \in |H + (3-r)F|$.

In particular $\langle D \rangle \subset \mathbb{E}$ is of dimension
$d-r+3$, so that $\langle D \rangle = \mathbb{E} =
\mathbb{P}^{d-r+3}$. As $\pi_{\Lambda}:\widetilde{X}
\twoheadrightarrow X$ is almost non-singular, the restriction
$\pi'_{\Lambda}\upharpoonright: \mathbb{E} = \langle D \rangle
\twoheadrightarrow  \mathbb{F}(X) = \mathbb{P}^2$ of
$\pi'_{\Lambda}: \mathbb{P}^{d+1} \twoheadrightarrow \mathbb{P}^r$
is generically injective along $D$.

(ii) $\Rightarrow$ (i): Let $X = \widetilde{X}_{\Lambda}$, where
$$\widetilde{X} \subset \mathbb{P}^{d+1}, \quad \Lambda = \mathbb{P}^{d-r} \subset \langle D \rangle = \mathbb{P}^{d-r+3} \mbox{ and } D \in |H + (3-r)F|$$
are as in statement (ii). Then $D \subset \widetilde{X} \cap
\mathbb{E}$ is of dimension $1$ and of degree $d-r+3$, and
$\mathbb{F} := \pi'_{\Lambda}(\mathbb{E} \setminus \Lambda)$ is a
$2$-plane in $\mathbb{P}^r$. As $\pi'_{\Lambda}\upharpoonright
\langle D \rangle \twoheadrightarrow \mathbb{F}$ is generically
injective, the projection map $\pi_{\Lambda}:
\widetilde{X}\twoheadrightarrow X$ is birational and hence
$\pi_{\Lambda}(D) \subset X \cap \mathbb{F}$ is of dimension $1$ and
of degree $d-r+3$. Therefore $X \cap \mathbb{F}$ is of dimension $1$
and of degree $d-r+3$. By Lemma~\ref{4.4'' Lemma} (ii) it follows
that $\mathbb{F}(X) = \mathbb{F} = \mathbb{P}^2$.
\smallskip

\noindent (b): The proof of statement (a) shows in particular that
condition (i) of that statement implies the two claims (1) and (2)
of statement (b).
\end{proof}

\section{Examples and Problems}
\label{7. Examples and Problems}

\noindent The central aim of this paper is to investigate
non-degenerate irreducible surfaces $X \subset \mathbb{P}^r$ of
degree $d$ with maximal sectional regularities, hence surfaces for
which the dimension $\mathfrak{d}(X)$ of the set
$\Sigma^{\circ}_{d-r+3}(X)$ of proper extremal secant lines takes
its maximally possible value. Our first aim in this present section
is to provide examples of smooth surfaces of extremal regularity for
which $\mathfrak{d}(X) \in \{-1,0,1\}$. In particular we shall see
that there are many such surfaces with $\mathfrak{d}(X) = -1$, that
is without extremal secant lines at all. The fact, that there are a
lot of surfaces of extremal regularity without extremal secant lines
does not correspond to the general expectation which arose by the
work of Gruson-Lazarsfeld-Peskine \cite{GruLPe}.

\begin{construction and examples}\label{7.1 Construction and Examples} (A) Let $a,b,d \in \mathbb{N}$ with $a \leq b$, let $r:=a+b+3$, assume that $d > r$ and consider the
smooth threefold rational normal scroll of degree $a+b+1 = r-2$
\begin{equation*}
Z := S(1,a,b) \subset \mathbb{P}^r.
\end{equation*}
Let $H, F \in {\rm Div}(Z)$ respectively be a hyperplane section and a ruling plane of $Z$, so that each divisor on $Z$ is linearly equivalent to $mH+nF$ for some integers
$m,n$. Let $X \subset \mathbb{P}^r$ be an non-degenerate irreducible surface of degree $d$ which is contained in $Z$ as a divisor linearly equivalent to $H+(d-r+2)F$.
Then one can easily see that
\begin{equation*}
h^0 (X, \mathcal{O}_X (1)) = h^0 (Z, \mathcal{O}_Z (1))+d-r+1 = d+2
\end{equation*}
This means that the linearly normal embedding $X \subset \mathbb{P}^{d+1}$ of $X$ by means of $\mathcal{O}_X (1)$ is of minimal degree and $X$
is the image of a surface $\widetilde{X} \subset \mathbb{P}^{d+1}$ of minimal degree under an isomorphic linear projection $\widetilde{X} \stackrel{\cong}{\rightarrow} X$,
hence
$$X = \widetilde{X}_{\Lambda}, \mbox{ with } \Lambda = \mathbb{P}^{d-r} \subset \mathbb{P}^{d+1} \mbox{ and }\Lambda \cap \Sec(\widetilde{X}) = \emptyset.$$
Keep in mind, that $\widetilde{X}$ is either a smooth rational
normal surface scroll, a cone over a rational normal curve or the
Veronese surface in $\mathbb{P}^5$. As $d+1 > 5$ and as a cone does
not admit a proper isomorphic linear projection, $\widetilde{X}
\subset \mathbb{P}^{d+1}$ is a smooth rational normal surface
scroll. In particular $X$ must be smooth and sectionally rational.
In addition, $X$ is not $(d-r+1)$-normal (see Lemma \ref{4.10'''
Lemma}) and hence ${\rm reg}(X) \geq d-r+3$. By Corollary~\ref{2.6'
Corollary} it follows that $\reg(X) = d-r+3$. In particular, $X$ is
a surface of extremal regularity.

Moreover, each line section $\mathbb{L}$ of $Z$ intersects the
divisor $H + (d-r+2)F$ in $d-r+2$ ruling planes and the hyperplane
section $H$. As $X \in |H+(d-r+2)F|$ it follows that $\mathbb{L}$ is
either contained in $X$ or else a proper $(d-r+3)$-secant line.
Conversely, each proper $(d-r+3)$-secant line to $X \in |H +
(d-r+2)F|$ must intersect $d-r+2$ ruling planes and the hyperplane
section $H$, and hence must be contained in $Z$ as a line section.
Therefore we have
$$\Sigma^{\circ}_{d-r+3}(X) = \{\mathbb{L} \in \mathbb{G}(1,\mathbb{P}^r) \mid \mathbb{L} \mbox{ is a line section of } Z \mbox{ and } \mathbb{L} \nsubseteq X\}.$$

\noindent (B) Suppose first that $a \geq 2$ and that the unique line
section $S(1) = \mathbb{L}$ of $Z$ is contained in $X$. Then,
according to the observation made in part (A), we have
$$\Sigma^{\circ}_{d-r+3}(X) = \emptyset, \mbox{ and hence } \mathfrak{d}(X) = {}^*\mathfrak{d}(X) = -1.$$
So, in this case $X \subset \mathbb{P}^r$ is a smooth surface of
extremal regularity having no proper extremal secant line at all.

\noindent (C) Suppose now that yet $a \geq 2$, but that the unique
line section $S(1) = \mathbb{L}$ of $Z$ is not contained in $X$.
Then, according to part (A) we have
$$\Sigma^{\circ}_{d-r+3}(X) = \{ \mathbb{L} \}, \mbox{ and hence }  \mathfrak{d}(X) = {}^*\mathfrak{d}(X) = 0.$$
In particular, in this case we and also have
$$ {}^*\Sigma^{\circ}_{d-r+3}(X)= \{\mathbb{L}\}.$$

\noindent (D) Suppose next, that $a =1$ and $b \geq 2$. Then, by
part (A)
$$\Sigma^{\circ}_{d-r+3}(X) = \{\mathbb{L} \in \mathbb{G}(1.\mathbb{P}^r) \mid \mathbb{L} \subset S(1,1) \mbox{ and } \mathbb{L} \nsubseteq X\}.,$$
So, in this case the set of proper extremal secant lines $\Sigma^{\circ}_{d-r+3}(X)$ to $X$ is obtained from the set
$\{\mathbb{L} \in \mathbb{G}(1.\mathbb{P}^r) \mid \mathbb{L} \subset S(1,1)\}$ of all line sections of $Z$ by removing finitely many lines. Therefore
$$ \mathfrak{d}(X) = {}^*\mathfrak{d}(X) = 1,$$
and for the closed union of all proper extremal secant to $X$ lines it holds
$$\mathbb{F}^{+}(X) = \overline{\bigcup_{\mathbb{L} \in \Sigma^{\circ}_{d-r+3}(X)}\mathbb{L}} = S(1,1) \subset Z.$$

\noindent (E) Suppose finally that $a =b=1$ and hence $r=5$, so that
$Z = S(1,1,1) \subset \mathbb{P}^5$.  By statement (A) we now have
$$\Sigma^{\circ}_{d-r+3}(X) = \{\mathbb{L} \in \mathbb{G}(1.\mathbb{P}^r) \mid \mathbb{L} \subset S(1,1,1) = Z \mbox{ and } \mathbb{L} \nsubseteq X\}.$$
Now, the set $\Sigma^{\circ}_{d-r+3}(X)$ of all proper extremal secant lines to $X$ is obtained by removing from the two-dimensional set
$\{\mathbb{L} \in \mathbb{G}(1.\mathbb{P}^r) \mid \mathbb{L} \subset S(1,1,1) = Z\}$ of all lines contained in $Z$ the at most one-dimensional family of all lines
$\mathbb{L}$ contained in $X$. Therefore
$$ \mathfrak{d}(X) = {}^*\mathfrak{d}(X) = 2,$$
and the extended extremal variety of $X$ -- thus the closed union of all proper extremal secant lines to $X$ -- coincides with the extremal variety of $X$, whence
$$\mathbb{F}(X) = \mathbb{F}^{+}(X) = \overline{\bigcup_{\mathbb{L} \in \Sigma^{\circ}_{d-r+3}(X)}\mathbb{L}} = S(1,1,1) = Z.$$
Observe, that now the surface $X$ is of maximal sectional regularity and falls under the exceptional case in which the extremal variety $\mathbb{F}(X)$ of $X$ is not a plane.
\end{construction and examples}

Next, we aim to present examples which concern surfaces $X$ of maximal sectional regularity which fall under the general case, in which the extremal variety $\mathbb{F}(X)$
of $X$ is a plane. So, let $5 \leq r < d$ and let $X \subset \mathbb{P}^r$ be a non-degenerate irreducible surface of degree $d$, which is of maximal sectional regularity.
We suppose that $\mathbb{F}(X)=\mathbb{P}^2$ and set $Y := X \cup \mathbb{F}(X)$. Then, Theorem~\ref{4.14'' Theorem} (f) can be tabulated
as follows where $\tau(X)$ denotes the pair $({\rm depth}(X),{\rm depth}(Y))$ of the arithmetic depths of $X$ and $Y$:

\[
\begin{tabular}{|c||c|c|c|}\hline
             {$d$}& {$r+1 \leq d \leq 2r-4$}  &{$2r-3 \leq d \leq 3r-7$}  &$d \geq 3r-6$\\\hline
             {$\tau(X)$}& {$(2,3)$} &$(1,1),(2,2),(2,3)$ &$(1,1),(2,2)$  \\\hline
\end{tabular}
\]
\vspace{0.2 cm}

\noindent The aim of this section is to provide examples of surfaces
$X \subset \mathbb{P}^r$ of maximal sectional regularity, having all
possible $\tau(X)$ listed in the above
table. Throughout this section we assume that the characteristic of $K$ is zero. \\

\begin{construction and examples}\label{7.2 Construction and Examples} (A) We throughout assume that the characteristic of the base field $K$ is zero. Let $a,b$ be
integers such that $3 \leq a \leq b$ and consider the standard
smooth rational normal surface scroll
\begin{equation*}
\widetilde{X} := S(a,b) \subset \mathbb{P}^{a+b+1}.
\end{equation*}
We shall construct our examples by varying $a$ and $b$ and by
projecting $S(a,b)$ from appropriate linear subspaces of
$\mathbb{P}^{a+b+1}$. The occurring Betti diagrams have been
computed by means of the Computer Algebra System Singular
\cite{GrPf}.

\noindent (B) Let $a \leq b$ and let $X = \widetilde{X}_{\Lambda}
\subset \mathbb{P}^{b+3}$ be the linear projection of $\widetilde{X}
= S(a,b)$ from a general $(a-3)$-dimensional subspace $\Lambda =
\mathbb{P}^{a-2}$ of $\langle S(a) \rangle = \mathbb{P}^a$. Observe
that $X \subset \mathbb{P}^{b+3}$ is a non-degenerate irreducible
surface of degree $a+b$. Moreover, the linear projection
$\pi'_{\Lambda}:\mathbb{P}^{a+b+1} \setminus \Lambda
\twoheadrightarrow \mathbb{P}^{b+3}$ in question maps the subspace
$\mathbb{P}^a \subset \mathbb{P}^{a+b+1}$ onto a plane $\mathbb{P}^2
= \mathbb{F} \subset \mathbb{P}^{b+3}$, so that the the rational
normal curve $S(a) = S(a,b) \cap \mathbb{P}^a \subset \mathbb{P}^a$
is mapped birationally onto the plane curve $C_a \cap
\mathbb{F}\subset \mathbb{F} := \mathbb{P}^2$ of degree $a$. As
$\deg(C_a) = a = a+b-(b+3)+3$, it follows by Lemma~\ref{4.4'' Lemma}
(a) that $X \subset \mathbb{P}^{b+3}$ is of maximal sectional
regularity and $\mathbb{F}(X) = \mathbb{F} = \mathbb{P}^2$. By
Theorem~\ref{2.8' Proposition} (a) we have ${\rm reg}(X)=a$.
Finally, as $(b+3)+1 \leq a+b \leq 2(b+3) - 4$ it follows from
Theorem~\ref{4.14'' Theorem} (f) that
$$\tau(X) = (2,3).$$

\noindent (C) Assume that $b \geq 3$ and let $X =
\widetilde{X}_{\Lambda} \subset \mathbb{P}^{a+3} = \mathbb{P}^r$ be
the non-degenerate irreducible surface of degree $d := a+b$ which is
obtained by a linear projection of $\widetilde{X} = S(a,b)$ from a
general $(b-3)$-dimensional subspace $\Lambda$ of $\langle S(b)
\rangle = \mathbb{P}^b \subset \mathbb{P}^{a+b+1}$. The underlying
projection $\pi'_{\Lambda}: \mathbb{P}^{a+b+1} \setminus \Lambda
\twoheadrightarrow \mathbb{P}^r$ maps the space $\mathbb{P}^b$ onto
a plane $\mathbb{F} = \mathbb{P}^2 \subset \mathbb{P}^{a+3}$ and
induces again a birational map from the rational normal curve $S(b)
= \widetilde{X} \cap \mathbb{P}^b$ onto the plane curve $C_b = X
\cap \mathbb{F} \subset \mathbb{F} = \mathbb{P}^2$ of degree $b$. It
follows as in part (B) that $X \subset \mathbb{P}^r$ is a
surface of maximal sectional regularity, that $\mathbb{F}(X) = \mathbb{F} = \mathbb{P}^2$ and that ${\rm reg}(X)=b = d-r+3$.\\
Now, by Theorem~\ref{4.14'' Theorem} (f) we obtain:
$$ \mbox{ If } b \leq a+2, \mbox{ then } \tau(X) = (2,3).$$

\noindent (D) To provide examples for further pairs $\tau(X)$, we
assume that $a+3 \leq b$ and we vary the center $\Lambda =
\mathbb{P}^{b-3} \subset \mathbb{P}^b = \langle S(b) \rangle$ of our
projection such the curve $C_b \subset \mathbb{F} = \mathbb{P}^2$
has the appropriate shape. To do so, we first consider the canonical
isomorphisms
$$\kappa: \mathbb{P}^1 \stackrel{\cong}{\rightarrow} S(b), \quad (s:t) \mapsto (0:\ldots:0:s^b:s^{b-1}t:\ldots:st^{b-1}:t^b) \in \mathbb{P}^{a+b+1}.$$
Then we choose a homogeneous polynomial $f \in K[s,t]$ of degree $b$ which is not divisible by $s$ and by $t$. Then we chose our center of projection $\Lambda
= \mathbb{P}_{b-3} \subset \mathbb{P}^b =\langle S(b) \rangle$ such that the composition
$$\pi \circ \kappa: \mathbb{P}^1 \rightarrow C_b \subset \mathbb{F} = \mathbb{P}^2$$
of the previously defined map $\kappa: \mathbb{P}^1 \rightarrow S(b)$ with the induced finite birational projection morphism
$\pi = \pi_{\Lambda} \upharpoonright: S(b) \twoheadrightarrow C_b$ is given by
$$\pi \circ \kappa = [s^b, f, t^b]: \mathbb{P}^1 \longrightarrow \mathbb{P}^2, \quad \big((s:t) \mapsto (s^b : f(s,t) : t^b)\big).$$
We denote the corresponding projected image $\widetilde{X}_{\Lambda} \subset \mathbb{P}^r$ of $\widetilde{X} = S(a,b)$ by $X_f$ and keep in mind that according to part
(C) the surface $X_f \subset \mathbb{P}^{a+3} = \mathbb{P}^r$ is of degree $d = a+b$ and of maximal sectional regularity.
Clearly, we may identify $\mathbb{F} \subset \mathbb{P}^{a+3} = \mathbb{P}^r$ with the subspace whose non-vanishing homogeneous coordinates sit at the last three places.
If we do so, we may write
$$X_f := \{\big(us^a:us^{a-1}t:\ldots:ust^{a-1}:ut^a:vs^b:vf(s,t):vt^b\big) \mid (s,t), (u,v) \in K^2 \setminus\{(0,0)\}\}.$$
After an appropriate choice of $f$, this latter presentation is accessible to syzygetic computations.
\end{construction and examples}
\vspace{0.2 cm}

\begin{example}\label{7.3 Example}
Let $(a,b)=(3,5)$ and $f := s^4t+s^3t^2+s^2t^3+st^4$. Then $X_f \subset \mathbb{P}^6$ is of degree $d = 8 (= 2r-4)$ and the graded
Betti numbers $\beta_{i,j} = \beta_{i,j}(X)$ of $X$ are as presented in the following table.

\begin{table}[hbt]
\begin{center}
\begin{tabular}{|c|c|c|c|c|c|c|c|}\hline
                          \multicolumn{2}{|c||}{$i$}& $1$&$2$&$3$&$4$&$5$&$6$  \\\cline{1-8}
             \multicolumn{2}{|c||}{$\beta_{i,1}$}& $6$&$8$&$3$&$0$&$0$&$0$ \\\cline{1-1}
             \multicolumn{2}{|c||}{$\beta_{i,2}$}& $4$&$12$&$12$&$4$&$0$&$0$\\\cline{1-1}
             \multicolumn{2}{|c||}{$\beta_{i,3}$}& $0$&$0$&$0$&$0$&$0$&$0$\\\cline{1-1}
             \multicolumn{2}{|c||}{$\beta_{i,4}$}& $1$&$4$&$6$&$4$&$1$&$0$\\\cline{1-8}
             \end{tabular}
\end{center}
\end{table}

By Lemma~\ref{4.16'' Lemma} (a) it follows from this graded Betti diagram of $X$, that
$$\tau(X)=(2,3).$$
\end{example}

\begin{example}\label{7.4 Example}
Let $(a,b)=(3,8)$ and consider $X_{f_i} \subset \mathbb{P}^6$ $(i=1,2,3)$ for the following choices of $f_i$:
\begin{enumerate}
\item[$(1)$] $f_1 = s^7t+s^6t^2+s^5t^3+s^4t^4+s^3t^5+s^2t^6+st^7$,
\item[$(2)$] $f_2 = s^7t+s^6t^2+s^5t^3+s^4t^4+s^3t^5+s^2t^6$, and
\item[$(3)$] $f_3 = s^7t+s^6t^2+s^5t^3+s^4t^4$.
\end{enumerate}
Then $X_{f_i} \subset \mathbb{P}^6$ is of degree $d = 11 \quad (= 2r-1 = 3r-7)$ for all $i=1,2,3$. The graded Betti diagrams of $X_{f_1}$, $X_{f_2}$ and
$X_{f_3}$ are given respectively in the three tables below.

\[
\begin{tabular}{|c|c|c|c|c|c|c|c|}\hline
                          \multicolumn{2}{|c||}{$i$}& $1$&$2$&$3$&$4$&$5$&$6$  \\\cline{1-8}
             \multicolumn{2}{|c||}{$\beta_{i,1}$}& $6$&$8$&$3$&$0$&$0$&$0$ \\\cline{1-1}
             \multicolumn{2}{|c||}{$\beta_{i,2}$}& $0$& $0$&$0$& $0$&$0$&$0$  \\\cline{1-1}
             \multicolumn{2}{|c||}{$\beta_{i,3}$}& $4$&$12$&$12$&$4$&$0$&$0$\\\cline{1-1}
             \multicolumn{2}{|c||}{$\beta_{i,4}$}& $0$&$0$&$0$&$0$&$0$&$0$\\\cline{1-1}
             \multicolumn{2}{|c||}{$\beta_{i,5}$}& $1$&$4$&$6$&$4$&$1$&$0$\\\cline{1-1}
             \multicolumn{2}{|c||}{$\beta_{i,6}$}& $0$&$0$&$0$&$0$&$0$&$0$\\\cline{1-1}
             \multicolumn{2}{|c||}{$\beta_{i,7}$}& $1$&$4$&$6$&$4$&$1$&$0$\\\cline{1-8}
             \end{tabular}
\]


\[
\begin{tabular}{|c|c|c|c|c|c|c|c|}\hline
                          \multicolumn{2}{|c||}{$i$}& $1$&$2$&$3$&$4$&$5$&$6$  \\\cline{1-8}
             \multicolumn{2}{|c||}{$\beta_{i,1}$}& $5$&$5$&$0$&$0$&$0$&$0$ \\\cline{1-1}
             \multicolumn{2}{|c||}{$\beta_{i,2}$}& $1$& $0$&$1$& $0$&$0$&$0$  \\\cline{1-1}
             \multicolumn{2}{|c||}{$\beta_{i,3}$}& $1$&$9$&$11$&$4$&$0$&$0$\\\cline{1-1}
             \multicolumn{2}{|c||}{$\beta_{i,4}$}& $4$&$18$&$32$&$28$&$12$&$2$\\\cline{1-1}
             \multicolumn{2}{|c||}{$\beta_{i,5}$}& $0$&$0$&$0$&$0$&$0$&$0$\\\cline{1-1}
             \multicolumn{2}{|c||}{$\beta_{i,6}$}& $0$&$0$&$0$&$0$&$0$&$0$\\\cline{1-1}
             \multicolumn{2}{|c||}{$\beta_{i,7}$}& $1$&$4$&$6$&$4$&$1$&$0$\\\cline{1-8}
             \end{tabular}
\]

\[
\begin{tabular}{|c|c|c|c|c|c|c|}\hline
                          \multicolumn{2}{|c||}{$i$}& $1$&$2$&$3$&$4$&$5$  \\\cline{1-7}
             \multicolumn{2}{|c||}{$\beta_{i,1}$}& $3$&$2$&$0$&$0$&$0$ \\\cline{1-1}
             \multicolumn{2}{|c||}{$\beta_{i,2}$}& $10$& $27$&$24$& $7$&$0$  \\\cline{1-1}
             \multicolumn{2}{|c||}{$\beta_{i,3}$}& $0$&$0$&$0$&$0$&$0$\\\cline{1-1}
             \multicolumn{2}{|c||}{$\beta_{i,4}$}& $0$&$0$&$0$&$0$&$0$\\\cline{1-1}
             \multicolumn{2}{|c||}{$\beta_{i,5}$}& $0$&$0$&$0$&$0$&$0$\\\cline{1-1}
             \multicolumn{2}{|c||}{$\beta_{i,6}$}& $0$&$0$&$0$&$0$&$0$\\\cline{1-1}
             \multicolumn{2}{|c||}{$\beta_{i,7}$}& $1$&$4$&$6$&$4$&$1$\\\cline{1-7}
             \end{tabular}
\]
By Lemma~\ref{4.16'' Lemma} (a) we can see from these tables that
$$\tau (X_{f_1}) = (2,2), \quad \tau(X_{f_2}) = (1,1) \mbox{ and  } \tau (X_{f_3}) = (2,3).$$
\end{example}

\begin{example}\label{7.5 Example}
Let $(a,b)=(3,9)$ and consider $X_f \subset \mathbb{P}^6$, $(i=1,2)$ for the two choices
\begin{enumerate}
\item[$(1)$] $f_1 = s^8t+s^7t^2+s^6t^3+s^5t^4+s^4t^5+s^3t^6+s^2t^7+st^8$ and
\item[$(2)$] $f_2 = s^8t+s^7t^2+s^6t^3+s^5t^4+s^4t^5+s^3t^6+s^2t^7$.
\end{enumerate}
Then $X_{f_i} \subset \mathbb{P}^6$ is of degree $d = 12 \quad (= 2r = 3r-6)$ for $i=1,2$. The graded Betti diagrams of $X_{f_1}$ and $X_{f_2}$ are
given respectively in the tables below, respectively.

\[
\begin{tabular}{|c|c|c|c|c|c|c|c|}\hline
                          \multicolumn{2}{|c||}{$i$}& $1$&$2$&$3$&$4$&$5$&$6$  \\\cline{1-8}
             \multicolumn{2}{|c||}{$\beta_{i,1}$}& $6$&$8$&$3$&$0$&$0$&$0$ \\\cline{1-1}
             \multicolumn{2}{|c||}{$\beta_{i,2}$}& $0$& $0$&$0$& $0$&$0$&$0$  \\\cline{1-1}
             \multicolumn{2}{|c||}{$\beta_{i,3}$}& $2$&$4$&$0$&$0$&$0$&$0$\\\cline{1-1}
             \multicolumn{2}{|c||}{$\beta_{i,4}$}& $1$&$4$&$10$&$6$&$1$&$0$\\\cline{1-1}
             \multicolumn{2}{|c||}{$\beta_{i,5}$}& $0$&$0$&$0$&$0$&$0$&$0$\\\cline{1-1}
             \multicolumn{2}{|c||}{$\beta_{i,6}$}& $1$&$4$&$6$&$4$&$1$&$0$\\\cline{1-1}
             \multicolumn{2}{|c||}{$\beta_{i,7}$}& $0$&$0$&$0$&$0$&$0$&$0$\\\cline{1-1}
             \multicolumn{2}{|c||}{$\beta_{i,8}$}& $1$&$4$&$6$&$4$&$1$&$0$\\\cline{1-8}
             \end{tabular}
\]
\[
\begin{tabular}{|c|c|c|c|c|c|c|c|}\hline
                          \multicolumn{2}{|c||}{$i$}& $1$&$2$&$3$&$4$&$5$&$6$  \\\cline{1-8}
             \multicolumn{2}{|c||}{$\beta_{i,1}$}& $5$&$5$&$0$&$0$&$0$&$0$ \\\cline{1-1}
             \multicolumn{2}{|c||}{$\beta_{i,2}$}& $0$& $0$&$1$& $0$&$0$&$0$  \\\cline{1-1}
             \multicolumn{2}{|c||}{$\beta_{i,3}$}& $5$&$15$&$15$&$5$&$0$&$0$\\\cline{1-1}
             \multicolumn{2}{|c||}{$\beta_{i,4}$}& $0$&$0$&$0$&$0$&$0$&$0$\\\cline{1-1}
             \multicolumn{2}{|c||}{$\beta_{i,5}$}& $5$&$23$&$42$&$38$&$17$&$3$\\\cline{1-1}
             \multicolumn{2}{|c||}{$\beta_{i,6}$}& $0$&$0$&$0$&$0$&$0$&$0$\\\cline{1-1}
             \multicolumn{2}{|c||}{$\beta_{i,7}$}& $0$&$0$&$0$&$0$&$0$&$0$\\\cline{1-1}
             \multicolumn{2}{|c||}{$\beta_{i,8}$}& $1$&$4$&$6$&$4$&$1$&$0$\\\cline{1-8}
             \end{tabular}
\]

\vspace{.2cm} 
By Lemma~\ref{4.16'' Lemma} (a) we can verify that
$$\tau(X_{f_1}) = (2,2) \mbox{ and } \tau(X_{f_2}) = (1,1).$$
\end{example}

\begin{remark}
The previously given examples of Betti tables have been computed
over a base field of characteristic $0$. It turns out, that in some
of these examples the Betti table varies with the characteristic of
the base field. The SINGULAR files with our computations are
available upon request to the authors.
\end{remark}

\begin{problem and remark}\label{7.6 Problem and Remark} (A) Let $5 \leq r <d$ and let $X \subset \mathbb{P}^r$ be a non-degenerate surface of degree $d$ which is of maximal
sectional regularity. We consider the three conditions
\begin{itemize}
\item[\rm{(i)}]   $N(X) \leq d-r$.
\item[\rm{(ii)}]  $\beta_{1,d-r+2}(X) = 1$.
\item[\rm{(iii)}] $\mathbb{F}(X) = \mathbb{P}^2$.
\end{itemize}

\noindent (B) By the implication (i) $\Rightarrow$ (iii) given in
statement (a) of Theorem~\ref{4.17'' Proposition} we have the
implication (i) $\Rightarrow$ (ii) among the above three conditions.
By the implication (i) $\Rightarrow$ (ii) given in statement (b) of
Theorem~\ref{4.17'' Proposition} we have the implication (ii)
$\Rightarrow$ (iii) among the above three conditions.

We expect, that the converse of both implications holds but could
not prove this. So we aim to pose the problem
\begin{itemize}
\item[\rm{(P)}] \textit{ Are the three conditions} (i), (ii) \textit{and} (iii) \textit{of part} (A) \textit{equivalent ?}
\end{itemize}
Observe, that in view of Theorem~\ref{4.11'' Theorem} (e) an affirmative answer to this would also answer affirmatively the question, whether the extended extremal variety
and the extremal variety of $X$ coincide, hence the question whether
\begin{itemize}
\item[\rm{(Q)}] $\quad \mathbb{F}^+(X) = \mathbb{F}(X)$ \textit{?}
\end{itemize}
\end{problem and remark}
\vspace{0.4 cm}

{\bf Acknowledgement.} The first named author thanks to the Korea University Seoul, to the Mathematisches Forschungsinstitut Oberwolfach, to the Martin-Luther 
Universit\"at Halle and to the Deutsche Forschungsgemeinschaft for their hospitality and the financial support provided during the preparation of this work. The
second named author was supported by the Nation Researcher program 2010-0020413 of NRF and MEST. The third named author was supported by the NRF-DAAD GEnKO Program 
(NRF-2011-0021014). The fourth named author thanks to the Korea University Seoul, to the Mathematisches Forschungsinstitut Oberwolfach and to the Deutsche 
Forschungsgemeinschaft for their hospitality respectively financial support offered during the preparation of this work.



\medskip




\begin{thebibliography}{00}

\bibitem{AK} {\sc Ahn, J., Kwak, S.}: {\em Graded mapping cone theorem, multisecants and syzygies}, Journal of
        Algebra 331 (2011) 243 - 262.

\bibitem{AlB} {\sc Albertini, C, Brodmann, M.}: {\em A bound on certain
    local cohomology modules and application to ample divisors},
    Nagoya Mathematical Journal 163 (2001) 87-106.

\bibitem{B0} {\sc Brodmann, M.}: {\em A few remarks on blowing-up and connectedness}, Journal f\"ur die reine und angewandte Mathematik 370 (1986) 51 - 60.

\bibitem{B2} {\sc Brodmann, M.}: {\em Cohomology of certain projective
    surfaces with low sectional genus and degree}, in: D. Eisenbud (Ed.)
    ``Commutative Algebra, Algebraic Geometry and Computational Methods'',
    172 - 200, Springer, 1999.

\bibitem{BP2} {\sc Brodmann, M., Park, E.}: {\em On varieties of almost minimal
    degree III: Tangent spaces and embedding scrolls}. Journal of Pure and Applied Algebra
    215 (2011) 2859 - 2872.

\bibitem{BS1} {\sc Brodmann, M., Schenzel, P.}: {\em Curves of degree
    $   r + 2$ in ${\mathbb P}^r$: Cohomological, geometric, and homological
    aspects}. Journal of Algebra 242 (2001), 577 - 623.

\bibitem{BS2} {\sc Brodmann, M., Schenzel, P.}: {\em On projective
    curves of maximal regularity}. Mathematische Zeitschrift 244 (2003), 271 - 289.

\bibitem{BS5} {\sc Brodmann, M., Schenzel, P.}: {\em Projective curves with
    maximal regularity and applications to syzygies and surfaces}. Manuscripta Mathematica
    135 (2011) 469 - 495.

\bibitem{BS6} {\sc Brodmann, M., Schenzel, P.}: {\em Projective surfaces of degree $r+1$
    in Projective $r$-space and almost non-singular projections}. Journal of Pure and
        Applied Algebra 216 (2012) 2241-2255.

\bibitem{BSh} {\sc Brodmann, M., Sharp, R.Y.}: {\em Local cohomology --
        an algebraic introduction with geometric applications}. Cambridge
        Studies in Advanced Mathematics Vol. 60, Cambridge University
        Press, Cambridge, UK, 1998.

\bibitem{C} {\sc Castelnuovo, G.}:
{\em Sui multipli di une serie lineare di gruppi di punti
appartenente ad une curva algebraic}, Rend. Circ. Mat. Palermo (2)
7 (1893), 89-110.


\bibitem{CaS} {\sc Caviglia, G., Sbarra, E.}: {\em Characterisic-free bounds for the Castelnuovo-Mumford
        regularity} Compositio Mathematicae 141 (2005) 1365 - 1373.

\bibitem{ChFN} {\sc Chardin, M., Fall, A.M., Nagel, U.}: {\em Bounds for the Castelnuovo-Mumford regularity
        of modules}. Mathematische Zeitschrift 258 (2008) 69 - 80.

\bibitem{GrPf} {\sc Decker, M., Greuel, G.M., Sch\"onemann, H.}: \newblock{Singular}
        $3-1-2$ -- {A computer algebra system for polynomial computations}. \newblock{http://www.singular.uni-kl.de} (2011).

\bibitem{EG} {\sc Eisenbud, D., G\^{o}to, S.}: {\em Linear free resolutions
        and minimal multiplicity}. J. Algebra 88 (1984) 89 - 133.

\bibitem{Ek} {\sc Ekedal, T.}: {\em Canonical models of surfaces of general type in positive characteristic}.
        Publications Math\'ematiques de l' I.H.E.S, 67 (1988) 97 - 144.

\bibitem{FlOV} {\sc Flenner, H., O'Caroll, L., Vogel, W.}: {\em Joins and intersections}.
        Springer Monographs in Mathematics, Springer Verlag, Berlin/Heidelberg/New York, 1990.

\bibitem{Fu} {\sc Fujita, T.}: {\em Classification theories of
        polarized varieties}, London Mathematical Society Lecture
        Notes Series 155, Cambridge University Press, 1990.

\bibitem{GL} {\sc Green, M., Lazarsfeld, R.}: {\em Some results on the
        syzygies of finite sets and algebraic curves}. Compositio Mathematica 67
        (1988), 301 - 314.

\bibitem{GruLPe} {\sc Gruson, L., Lazarsfeld, R., Peskine, C.}: {\em On a theorem
        of Castelnuovo and the equations defining space curves}. Inventiones Mathematicae 72
        (1983), 491 - 506.

\bibitem{H} {\sc Harris, J.}: {\em Algebraic geometry: A first course}.
        Graduate Texts in Mathematics, Vol. 133, Springer-Verlag, New York,
        1992.

\bibitem{Ha} {\sc Hartshorne, R.}: {\em Algebraic Geometry}. Springer Verlag,
        Berlin, 1977.

\bibitem{K} {\sc Kleiman, S.}: {\em The transversality of a genearal translate}. Compositio Mathematica 28 (1974) 287 - 297.

\bibitem{KwP} {\sc Kwak, S., Park, E.}: {\em Some effects of property $N_p$ on the higher normality and defining
        equations of nonlinearly normal varieties}. Journal f\"ur die reine und angewandte Mathematik 582 (2005) 87 - 105.

\bibitem{L} {\sc Lazarsfeld, R.}: {\em A sharp Castelnuovo bound for smooth surfaces}. Duke Mathematical Jornal 55 (1987) 423 - 429.

\bibitem{M} {\sc Mumford, D.}: {\em Lectures on curves on an algebraic surface}. Annals of Mathematics Studies 59,
        Princeton University Press, Princeton, 1966.

\bibitem{No} {\sc Noma, A.}: {\em A bound on the Castelnuovo-Mumford regularity for curves}. Mathematische
        Annalen 322 (2002) 69 - 74.

\bibitem{P} {\sc Park, E.}: {\em On syzygies of divisors of rational normal scrolls}, Preprint, 2013.

\bibitem{Pi} {\sc Pinkham, H.}: {\em A Castelnuovo bound for smooth surfaces}.
        Inventiones Mathematicae 83 (1986), 321 - 332.

\bibitem{R} {\sc Reid, M.}: {\em Chapters on algebraic surfaces}, Complex algebraic geometry (Park City, UT,
1993), 3--159, IAS/Park City Math. Ser., 3, Amer. Math. Soc., Providence, RI, 1997.

\end{thebibliography}
\end{document}